\documentclass[11pt,a4paper]{amsart}

\usepackage[english]{babel}
\usepackage[utf8]{inputenc}
\usepackage{amsmath}
\usepackage{amssymb}
\usepackage{amsthm}
\usepackage{bm}
\usepackage{dsfont}
\usepackage{enumitem}
\usepackage{scalerel}
\usepackage{mathtools}
\usepackage{fullpage}
\usepackage{xcolor}

\theoremstyle{plain}
\newtheorem{theorem}{Theorem}[section]

\newtheorem{lemma}[theorem]{Lemma}
\newtheorem{proposition}[theorem]{Proposition}

\theoremstyle{definition}

\newtheorem{definition}[theorem]{Definition}
\newtheorem{notation}[theorem]{Notation}

\theoremstyle{remark}
\newtheorem{remark}[theorem]{Remark}

\newcommand{\N}{\mathbb{N}}

\newcommand{\Q}{\mathbb{Q}}
\newcommand{\R}{\mathbb{R}}

\newcommand{\A}{\mathbb{A}}
\newcommand{\B}{\mathbb{B}}

\newcommand{\I}{\mathbb{I}}

\newcommand{\X}{\mathbb{X}}
\newcommand{\Y}{\mathbb{Y}}

\newcommand{\GG}{\mathcal{G}}

\newcommand{\KK}{\mathcal{K}}
\newcommand{\MM}{\mathcal{M}}

\newcommand{\OO}{\mathcal{O}}

\newcommand{\TT}{\mathcal{T}}

\newcommand{\set}[2]{\left\{#1:#2\right\}}
\newcommand{\compl}[1]{#1^{\textnormal{c}}}

\newcommand{\cl}[1]{\overline{#1}}
\newcommand{\cupdot}{\mathbin{\dot{\cup}}}

\newcommand{\abs}[2][]{\left|#2\right|_{#1}}

\DeclareMathOperator{\id}{id}

\DeclareMathOperator{\im}{Im} 

\DeclareMathOperator{\Img}{Im}
\DeclareMathOperator{\Dom}{Dom}
\DeclareMathOperator{\Sym}{Sym}
\DeclareMathOperator{\Aut}{Aut}
\DeclareMathOperator{\Inj}{Inj}
\DeclareMathOperator{\Emb}{Emb}
\DeclareMathOperator{\End}{End}
\DeclareMathOperator{\Surj}{Surj}

\DeclareMathOperator{\Dc}{Dc}
\DeclareMathOperator{\Cont}{Cont}
\DeclareMathOperator{\LP}{LP}

\usepackage{soul}
\usepackage{hyperref}
\hypersetup{colorlinks=true,linkcolor=blue,citecolor=red}
\usepackage{tikz}
\usetikzlibrary{matrix,arrows,calc,babel,patterns,decorations.pathmorphing,decorations.pathreplacing}

\newcommand{\PPXb}{Pseudo\nobreakdash-Property~$\mathbf{\overline{X}}$}
\theoremstyle{plain}
\newtheorem{theoremalph}{Theorem}

\newtheorem{questionalph}{Question}[section]

\usepackage{multicol}

\title{The semigroup of increasing functions on the rational~numbers
 has a unique Polish topology}

\author{Michael Pinsker}
\thanks{This research was funded in whole or in part by the Austrian
  Science Fund (FWF) [P 32337, I 5948]. For the purpose of Open
  Access, the authors have applied a CC BY public copyright licence to
  any Author Accepted Manuscript (AAM) version arising from this
  submission. This research is also funded by the European Union (ERC,
  POCOCOP, 101071674). Views and opinions expressed are however those
  of the author(s) only and do not necessarily reflect those of the
  European Union or the European Research Council Executive Agency.
  Neither the European Union nor the granting authority can be held
  responsible for them.}
\author{Clemens Schindler}
\thanks{The second author is a recipient of a DOC Fellowship of the
  Austrian Academy of Sciences at the Institute of Discrete
  Mathematics and Geometry, TU Wien.}

\address{Institut für Diskrete Mathematik und Geometrie, FG Algebra,
  TU Wien, Austria}
\email{marula@gmx.at}
\address{Institut für Diskrete Mathematik und Geometrie, FG Algebra,
  TU Wien, Austria}
\email{clemens.schindler@tuwien.ac.at}


\begin{document}
\begin{abstract}
  The set of increasing functions on the rational numbers, equipped
  with the composition operation, naturally forms a topological
  semigroup with respect to the topology of pointwise convergence in
  which a sequence of increasing functions converges if and only if it
  is eventually constant at every argument. We develop new techniques
  to prove there is no other Polish topology turning this semigroup
  into a topological one, and show that previous techniques are
  insufficient for this matter.
\end{abstract}
\maketitle

\section{Introduction}
\label{sec:introduction}

\subsection{The result}
\label{subsec:result}

The space $\MM_{\Q}$ of all -- not necessarily strictly -- increasing
functions on the rational numbers $\Q$ (with respect to the usual
order) is equipped with several interesting kinds of structure. We
focus on a topological and an algebraic one: on the one hand, it
carries the subspace topology of the product topology on $\Q^{\Q}$
where each copy of $\Q$ is endowed with the discrete topology (not the
standard Euclidean topology on $\Q$). In this topology, a sequence
$(f_{n})_{n\in\N}$ converges to $f$ if and only if for each $x\in\Q$,
the sequence $(f_{n}(x))_{n\in\N}$ is eventually constant with value
$f(x)$. It will be referred to as the \emph{pointwise topology} and
has several interesting properties: most importantly, it turns out to
be a Polish (completely metrisable and second countable) topology. On
the other hand, $\MM_{\Q}$ carries the composition operation $\circ$
of functions which forms a monoid structure (semigroup structure with
a neutral element). Furthermore, the two structures are compatible in
the sense that the operation is continuous as a function
$\circ\colon \MM_{\Q}\times \MM_{\Q}\to \MM_{\Q}$, where
$\MM_{\Q}\times \MM_{\Q}$ is endowed with the product topology. We say
that the pointwise topology is a \emph{semigroup topology} and that
$\MM_{\Q}$ together with the pointwise topology is a \emph{topological
  semigroup}.

At this point, it is natural to ask how much information about a
topology can be \emph{reconstructed} from the knowledge that it is
compatible with a given algebraic structure. This problem has been
studied from various angles and for many classes of algebraic
structures over the years, using techniques from several areas of
mathematics. Particular interest has been given to additional
requirements on the topology, for instance that the topology be Polish
or Hausdorff. We give two examples. Concerning vector spaces, it is a
folklore result in functional analysis that the finite dimensional
$\R$-vector space $\R^{n}$ carries a unique Hausdorff vector space
topology\footnote{A topology $\TT$ on $\R^{n}$ is called a
  \emph{vector space topology} if it is a group topology with respect
  to addition and if the scalar multiplication
  $(\lambda,x)\mapsto \lambda x$ is continuous as a map
  $(\R,\TT_{\text{Eucl}})\times (\R^{n},\TT)\to (\R^{n},\TT)$; note
  that the scalar field $\R$ is to carry the standard Euclidean
  topology.}, namely the standard Euclidean topology. However, its
additive group does carry multiple Hausdorff -- even Polish -- group
topologies already for $n=1$: the groups $(\R,+)$ and $(\R^{2},+)$ are
algebraically isomorphic (consider them as additive groups of vector
spaces over $\Q$ of equal dimension) but $\R$ and $\R^{2}$ equipped
with the Euclidean topologies are not homeomorphic. Hence, if we pull
back the Euclidean topology on $\R^{2}$ to $\R$ via the isomorphism
$(\R,+)\to(\R^{2},+)$, we obtain a Polish group topology different
from the Euclidean topology on $(\R,+)$. Note that this construction
requires the axiom of choice (to find the algebraic isomorphism
$(\R,+)\to (\R^{2},+)$). This is not a coincidence since
Solovay~\cite{Solovay} and Shelah~\cite{Shelah84} showed the
consistency of ZF (without choice) with the fact that \emph{any}
Polish group has a unique Polish group topology. In the realm of
semigroups, it was shown in~\cite{EJMMMP-zariski} that the pointwise
topology on the so-called full transformation monoid of all functions
on a countably infinite set is the unique Polish semigroup topology
(with respect to function composition), meaning that the topological
and the algebraic structure on this monoid are closely connected.

With this paradigm of reconstruction in mind, we arrive at the
following problem:

\begin{questionalph}\label{q:end-Q-upp}
  Is the pointwise topology the only Polish semigroup topology on the
  space $\MM_{\Q}$ of increasing functions on $\Q$?
\end{questionalph}
Note that a natural alternative topology is the subspace topology of
the product topology on~$\Q^{\Q}$ where instead each copy of $\Q$ is
endowed with the Euclidean topology on $\Q$. However, this topology is
not Polish. And indeed, the goal of the present paper is to prove the
following result:
\begin{theoremalph}\label{thm:end-Q-upp}
  The pointwise topology is the unique Polish semigroup topology on
  $\MM_{\Q}$.
\end{theoremalph}

\subsection{Context}
\label{subsec:context}

The reconstruction problem has been considered in a significantly
wider setting: if $A$ is a countably infinite set, the pointwise
topology on $A^{A}$ induces a Polish semigroup topology on any
subsemigroup of $A^{A}$ which is a $G_{\delta}$ set; we refer to this
subspace topology also as \emph{pointwise topology}. The most
prominent examples of such $G_{\delta}$ subsemigroups are $\Sym(A)$,
the space of all permutations of $A$, as well as -- more generally --
the automorphism group $\Aut(\A)$ and the endomorphism monoid
$\End(\A)$ of any given (model-theoretic) structure $\A$. In the first
two cases, the algebraic structure is even a group and the topology is
a \emph{group topology}, meaning that the inversion map on the group
is also continuous. Our question concerning the increasing functions
on the rational numbers fits into this framework since
$\MM_{\Q}=\End(\Q,\leq)$. As it turns out, the algebraic and the
topological structure on many subsemigroups of $A^{A}$ are very
closely intertwined, yielding results of the following shape (which
our Theorem~\ref{thm:end-Q-upp} is parallel to):
\begin{quote}
  On the (semi\nobreakdash-)group $S\subseteq A^{A}$, the pointwise
  topology is the unique Polish (semi\nobreakdash-)group topology.

  \noindent\qquad (hereafter: $S$ has the Unique Polish Property or
  UPP for short)
\end{quote}
However, Question~\ref{q:end-Q-upp} escaped existing techniques.

For the class of groups, UPP has been extensively studied; examples
include the full symmetric group $\Sym(A)$ (\cite{Gaughan} combined
with~\cite{Lascar}) and the automorphism group of the random
(di\nobreakdash-)graph (\cite{HodgesHodkinsonLascarShelah} combined
with~\cite{KechrisRosendal}). Additionally, $\Aut(\Q,\leq)$ --
explicitly: the space of all increasing permutations of $\Q$ -- also
has UPP (\cite{RosendalSolecki} combined with~\cite{Lascar}).

Recent years brought results for the case of endomorphism monoids as
well; apart from the full transformation monoid $A^{A}$ mentioned
above, it turns out that the endomorphism monoids of the random graph,
the random digraph and the equivalence relation with countably many
equivalence classes of countable size all have UPP,
see~\cite{EJMPP-polish}. One notices that the examples from these
lists either contain only bijective functions (the groups) or contain
both non-injective and non-surjective functions. This is essential for
UPP to hold: by constructions given in~\cite{EJMMMP-zariski}, both the
monoid $\Inj(A)$ of all injective functions on $A$ and the monoid
$\Surj(A)$ of all surjective functions on $A$ carry multiple Polish
semigroup topologies. The construction on $\Inj(A)$ also applies to
the so-called self-embedding monoid $\Emb(\A)$ of any
\emph{$\omega$\nobreakdash-categorical} (defined
e.g.~in~\cite{Hodges}) structure $\A$, see~\cite{EJMPP-polish}. This
in particular encompasses the monoid of all strictly increasing (but
not necessarily surjective) functions on $\Q$. Thus, if we were to
replace $\MM_{\Q}=\End(\Q,\leq)$ by $\End(\Q,<)$ in
Question~\ref{q:end-Q-upp}, the resulting question would have a simple
(negative) answer.

\subsection{Known technique}
\label{subsec:known-technique}

The papers~\cite{EJMMMP-zariski} and subsequently~\cite{EJMPP-polish}
show uniqueness of Polish semigroup topologies in two natural steps:
\begin{enumerate}[label=(\arabic*)]
\item Show that the pointwise topology is coarser than any Polish
  semigroup topology.
\item Show that the pointwise topology is finer than any Polish
  semigroup topology.
\end{enumerate}
Usually, Step~(2) takes considerably more work than Step~(1). The
essential tool for the first step is often the so-called \emph{Zariski
  topology}, a topology guaranteed to be coarser than any Hausdorff
semigroup topology; it then clearly is sufficient to prove that the
Zariski topology coincides with the pointwise topology. The second
step is accomplished by means of lifting from a subset, usually the
automorphism group, to the endomorphism monoid. To this end, the
following crucial instrument called \emph{Property~\textbf{X}} was
introduced in~\cite{EJMMMP-zariski}:
\begin{definition}\label{def:prop-X}
  Let $(S,\TT)$ be a topological semigroup and let $D\subseteq S$ be a
  subset of $S$. Then $(S,\TT)$ has \emph{Property~\textbf{X}} with
  respect to $D$ if for all $s\in S$ there exist $f_{s},g_{s}\in S$
  and $a_{s}\in D$ such that
  \begin{enumerate}[label=(\roman*)]
  \item $s=g_{s}a_{s}f_{s}$
  \item for every $\TT$-neighbourhood $O\subseteq S$ of $a_{s}$, the
    set $g_{s}(O\cap D)f_{s}$ is a $\TT$-neighbourhood of $s$.
  \end{enumerate}
\end{definition}
When trying to use the above outline on $\MM_{\Q}$, Step~(1) works
smoothly via a direct construction and has already been executed
in~\cite{EJMPP-polish} (see also Theorem~\ref{thm:ejm-pointw-coarser};
our proof implicitly shows that the Zariski topology coincides with
the pointwise topology). However, the common technique to perform
Step~(2) is not directly applicable since Property~\textbf{X} does not
hold.

\subsection{The proof}
\label{subsec:proof}

Our strategy to prove Theorem~\ref{thm:end-Q-upp} is a threefold
generalisation of the technique involving Property~\textbf{X}, two of
these aspects being essential extensions and one merely a technical
one.  First, we consider topologies that are finer than the pointwise
topology in intermediate steps, showing that $\MM_{\Q}$ endowed with a
finer topology has a form of Property~\textbf{X} with respect to the
automorphism group -- evidently, we subsequently have to reduce from
that richer topology to the pointwise topology in an additional step;
this is not necessary in the previous proofs using
Property~\textbf{X}.  Second, we admit the postcomposition of a
specific endomorphism (left-invertible is the key) on the left hand
side of the term $s=g_{s}a_{s}f_{s}$ serving as basis for
Property~\textbf{X}. Third, and this is only a technical complication,
we increase the length of the right hand side of the term. All in all,
we will consider $e_{s}s=h_{s}b_{s}g_{s}a_{s}f_{s}$, leading to a
generalisation of Property~\textbf{X} that we call \emph{\PPXb}; see
Definition~\ref{def:ppxb}.

Examining our proof more closely, we in fact show a stronger result:
We do not apply the full power of Polishness to prove that a Polish
semigroup topology $\TT$ on $\MM_{\Q}$ coincides with the pointwise
topology. Instead, we only use that $\TT$ is a semigroup topology
which is second countable, Hausdorff and regular, as well as that a
countable intersection of dense open sets is dense, i.e. that
$(\MM_{\Q},\TT)$ is a \emph{Baire space} (the conclusion of Baire's
theorem holds). By Urysohn's metrisation theorem (see
e.g.~\cite{Willard-topology}), a second countable space is metrisable
if and only if it is Hausdorff and regular, so we can reformulate our
result in the following form: The pointwise topology is the unique
second countable metrisable Baire semigroup topology on
$\MM_{\Q}$. Comparing with~\cite{EJMPP-polish}, most examples treated
there (in particular, all examples listed above) do not need full
Polishness either; the results instead yield that the respective
pointwise topologies are the unique second countable metrisable
semigroup topologies. Hence, no completeness-type assumption is
necessary for these structures. For $\MM_{\Q}$, we additionally need
that the topology is Baire, namely in the new reduction from the rich
topology to the pointwise topology which does not occur
in~\cite{EJMPP-polish}.

The paper is structured as follows: In
Section~\ref{sec:prelim-known-facts}, we introduce some relevant
notions and state the known results which we will use in the sequel,
in particular the fact that the pointwise topology on $\MM_{\Q}$ is
coarser than any Polish semigroup topology on $\MM_{\Q}$. The
remainder of this paper is devoted to showing that, conversely, the
pointwise topology is finer than any Polish semigroup topology on
$\MM_{\Q}$. We start by presenting the proof strategy more thoroughly
and defining the so-called \emph{rich} topology on~$\MM_{\Q}$ in
Section~\ref{sec:finest-topol-strat}. The details of the proof are
contained in Sections~\ref{sec:rich-top-ppxb}
and~\ref{sec:reduct-pointw-topol}, where the former proves that the
rich topology has \PPXb{} with respect to the pointwise topology on
the automorphism group and the latter focuses on reducing the rich
topology to the pointwise topology.
\subsection{Related notions}
\label{subsec:related-notions}

Several notions related to the existence of a unique Polish
(semi\nobreakdash-)group topology have been considered over the years,
all of them capturing various degrees to which the topological
structure can be \emph{reconstructed} from the algebraic one. For the
purposes of this paper, the most important ones are \emph{automatic
  homeomorphicity} as studied
e.g.~in~\cite{BodirskyEvansKompatscherPinsker,Reconstruction,behrisch-truss-vargas,EvansHewitt,PechPech-HomeoMon,pech-saturated}
and \emph{automatic continuity} as discussed
e.g.~in~\cite{BPP-projective-homomorphisms,EJMMMP-zariski,EJMPP-polish,Herwig98,HodgesHodkinsonLascarShelah,KechrisRosendal,Lascar,PaoliniShelahCountStbl,RosendalSolecki,Truss}. Similar
concepts have also been studied,
e.g.~in~\cite{MacphersonBarbina,behrisch-vargasgarcia-stronger,maissel-rubin,PaoliniShelahReconstructing,PaoliniShelahSSIPFreeHom,Rubin}.

Automatic homeomorphicity, the former notion, means that any algebraic
isomorphism from the closed sub(\nobreakdash-semi\nobreakdash-)group
$S$ of $\Sym(A)$ (or $A^{A}$) in question to another closed
sub(\nobreakdash-semi\nobreakdash-)group $T$ of $\Sym(A)$ (or $A^{A}$)
is indeed a homeomorphism between the respective pointwise topologies
on $S$ and $T$. This property clearly is a weakening of UPP; it can be
paraphrased as ``unique \emph{pointwise-like} semigroup
topology''. In~\cite{behrisch-truss-vargas}, it was shown that
$\MM_{\Q}$ has automatic homeo\-morphicity -- this result as well as the
fact that UPP holds for $\Aut(\Q,\leq)$ form the central motivation
for the present paper.

Automatic continuity, the latter notion, is in fact a template for
results, parametrised by a class $\KK$ of topological
(semi\nobreakdash-)groups; traditionally:
\begin{quote}
  For the topological (semi\nobreakdash-)group $S$, all algebraic
  homomorphisms $S\to H$, where $H\in\KK$, are continuous.
\end{quote}
As it turns out, automatic continuity and the known techniques to show
UPP revolving around Property~\textbf{X} are very closely connected:
if the automorphism group of the structure in question has automatic
continuity (with respect to the class of second countable topological
(semi\nobreakdash-)groups) and the techniques shall be applicable,
then the endomorphism monoid necessarily has automatic continuity as
well. This fact plays an important role in the present paper:
structures whose endomorphism monoid do not have automatic continuity
while the automorphism group does (like $\langle\Q,\leq\rangle$, see
Proposition~\ref{prop:aut-Q-auto-cont-semigrp} and
Proposition~\ref{prop:end-Q-non-auto-cont}) \emph{require}
generalisations of the techniques, like the ones we develop
here. Another such structure is the \emph{dual graph} (or
\emph{complement graph} without loops) of the equivalence relation
with countably many equivalence classes of countable size. This is one
of the \emph{homogeneous graphs} on countably infinite domains as
classified in~\cite{LachlanWoodrow}. In~\cite{EJMPP-polish}, it is
asked which of these graphs have endomorphism monoids satisfying UPP;
we conjecture that the dual graph of the equivalence relation with
countably many equivalence classes of countable size does, however, a
proof cannot rely only the known techniques. We suspect that it could
be within the scope of our methods.

\section{Preliminaries and known facts}
\label{sec:prelim-known-facts}

In this section, we make precise the terminology used in the following
and give some known results crucial to our reasoning.

\subsection{Structures and functions}
\label{subsec:str-funct}
A \emph{(relational) structure} $\A=\langle A,(R_{i})_{i\in I}\rangle$
consists of a \emph{domain} $A$ endowed with relations $R_{i}$ on $A$
of arity $m_{i}$, i.e.~$R_{i}\subseteq A^{m_{i}}$. If no
misunderstandings can arise, we will not strictly distinguish between
the structure $\A$ and its domain $A$. Contrary to some works in the
area (e.g.~\cite{EJMMMP-zariski,EJMPP-polish}), we write the
application of a function $f$ to an element $a$ as $f(a)$ and compose
functions from right to left, i.e.~$fg:=f\circ g:=(a\mapsto
f(g(a)))$. If $\bar{a}=(a_{1},\dots,a_{m})$ is a tuple of elements of
the domain of $f$, we write $f(\bar{a})=(f(a_{1}),\dots,f(a_{m}))$. A
function $f\colon A\to A$ and a relation $R$ on $A$ are
\emph{compatible} if whenever $\bar{a}\in R$, we have
$f(\bar{a})\in R$. An \emph{endomorphism} of a structure
$\A=\langle A,(R_{i})_{i\in I}\rangle$ is a function $f\colon A\to A$
which is compatible with all the relations $R_{i}$. The semigroup of
all endomorphisms is denoted by $\End(\A)$. An \emph{automorphism} of
$\A$ is a bijective function $f\colon A\to A$ such that both $f$ and
$f^{-1}$ are endomorphisms. The group of all automorphisms is denoted
by $\Aut(\A)$. In the sequel, we will focus on
$\A=\langle\Q,\leq\rangle$, the rational numbers equipped with the
non-strict order. As already noted in the introduction,
\begin{align*}
  \MM_{\Q}&:=\End(\Q,\leq)=\left\{f\colon\Q\to\Q\mid f\text{ increasing}\right\},\\
  \GG_{\Q}&:=\Aut(\Q,\leq)=\left\{f\colon\Q\to\Q\mid f\text{ bijective, (strictly) increasing}\right\}.
\end{align*}
Additionally, it will be useful to embed $\Q$ into the real numbers
$\R$. Consequently, we will allow intervals with irrational boundary
points as well. Differing from standard notation, we only consider the
\emph{rational} points in this interval, unless explicitly mentioned
otherwise: for $\gamma_{1},\gamma_{2}\in\R\cup\{\pm\infty\}$, we put
$(\gamma_{1},\gamma_{2}):=\set{q\in\Q}{\gamma_{1}<q<\gamma_{2}}$. To
avoid lengthy typesetting, we will denote $s((-\infty,\gamma))$ by
$s(-\infty,\gamma)$ et cetera. In the same spirit, we will write
$\sup\Img(s)$ as $\sup s$ and $\inf\Img(s)$ as $\inf s$. Finally, we
abbreviate $\I:=\R\setminus\Q$ and distinguish intervals as follows:
An interval is \emph{rational} if its boundary points are contained in
$\Q\cup\{\pm\infty\}$, and \emph{irrational} if its boundary points
are contained in $\I\cup\{\pm\infty\}$.

\subsection{Topologies}
\label{sec:topologies}

We endow $A^{A}$ with the \emph{pointwise topology} $\TT_{pw}$, that
is the product topology where each copy of $A$ carries the discrete
topology. Explicitly, a basis for the pointwise topology is given by
the sets $\set{f\in A^{A}}{f(x_{1})=y_{1},\dots,f(x_{n})=y_{n}}$,
where $n\geq 1$ and $x_{1},\dots,x_{n},y_{1},\dots,y_{n}\in A$. If $A$
is countable, it is a folklore fact that the pointwise topology on
$A^{A}$ is \emph{Polish} (second countable and completely metrisable)
as a countable product of Polish topologies. As can be easily seen,
the composition operation $\circ\colon A^{A}\times A^{A}\to A$ is
continuous with respect to this topology; hence, the pointwise
topology is a \emph{semigroup topology}. The induced topology on any
$G_{\delta}$ subsemigroup of $A^{A}$ with respect to the pointwise
topology is a Polish semigroup topology as well; notable examples are
the spaces $\End(\A)$ and $\Aut(\A)$ -- the former is closed in
$A^{A}$ with respect to the pointwise topology, the latter is closed
in the set $\Sym(A)$ of all permutations on $A$, which in turn is
readily seen to be $G_{\delta}$ in $A^{A}$. On any of these spaces, we
will refer to the induced topology also as \emph{pointwise topology}
and always denote it by $\TT_{pw}$, unless the underlying set is not
clear from the context.

We will make frequent use of the left and right translations, defined
on any semigroup $S$ as follows: Given a fixed $t\in S$, let
\begin{align*}
  \lambda_{t}&\colon S\to S,\quad \lambda_{t}(s):=ts\\
  \rho_{t}&\colon S\to S, \quad \rho_{t}(s):=st
\end{align*}
denote the \emph{left} and \emph{right translation} on $S$ by $t$,
respectively. If $S$ is a topological semigroup, then $\lambda_{t}$
and $\rho_{t}$ are continuous maps for any $t\in S$.

In the sequel, we will have to distinguish multiple topologies on the
same set; whenever the topology is not clear from the context, we will
write $(S,\TT)$ for the space $S$ endowed with the topology $\TT$.

\subsection{Automatic continuity}
\label{subsec:automatic-continuity}
In Subsection~\ref{subsec:related-notions}, we presented a version of
automatic continuity for topological (semi\nobreakdash-)groups which
we now generalise in a straightforward way:
\begin{definition}
  Let $S$ be a (semi\nobreakdash-)group and let $\TT$ be a topology on
  $S$ (which need not be a (semi\nobreakdash-)group topology). Given a
  class $\KK$ of topological (semi\nobreakdash-)groups, we say that
  $(S,\TT)$ has \emph{automatic continuity} with respect to $\KK$ if
  for any $(H,\OO)\in\KK$, all algebraic homomorphisms $S\to H$ are
  continuous as maps $(S,\TT)\to (H,\OO)$.
\end{definition}
The present subsection contains references to results which imply that
$(\MM_{\Q},\TT_{pw})$ does not have automatic continuity while
$(\GG_{\Q},\TT_{pw})$ satisfies a very strong version of automatic
continuity. We begin with the negative result, reformulating it to
match our terminology.
\begin{proposition}[{\cite[Proposition~9]{Reconstruction}}]\label{prop:bpp-non-auto-cont}
  Let $M$ be a $\TT_{pw}$-closed submonoid of $A^{A}$ for a countable
  set $A$. Suppose that $M$ contains a submonoid $N$ such that
  \begin{enumerate}[label=(\arabic*)]
  \item $N$ is not $\TT_{pw}$-closed in $M$;
  \item composing any element of $M$ with an element outside $N$
    yields an element outside $N$.
  \end{enumerate}
  Then $(M,\TT_{pw})$ does not have automatic continuity with respect
  to the class of all subsemigroups of~$A^{A}$, equipped with the
  respective pointwise topologies.
\end{proposition}
By a straightforward application of this observation, we obtain:
\begin{proposition}\label{prop:end-Q-non-auto-cont}
  $(\MM_{\Q},\TT_{pw})$ does not have automatic continuity with
  respect to the class of second countable topological semigroups.
\end{proposition}
\begin{proof}
  We set $M:=\MM_{\Q}$ as well as
  \begin{displaymath}
    N:=\set{f\in\MM_{\Q}}{\inf f=-\infty\text{ and }\sup f=+\infty}
  \end{displaymath}
  and check the assumptions of
  Proposition~\ref{prop:bpp-non-auto-cont}. Clearly, $N$ is a
  submonoid of $M$. If we define $f_{n}\in\MM_{\Q}$ by
  \begin{displaymath}
    f_{n}(x):=
    \begin{cases}
      0,& -n<x<n\\
      x,& x\geq -n\text{ or }x\leq n
    \end{cases}
  \end{displaymath}
  we have $f_{n}\in N$ but the sequence $(f_{n})_{n\in\N}$ converges
  with respect to the pointwise topology, namely to the constant
  function with value $0$ -- which is not in $N$. Hence, $N$ is not
  $\TT_{pw}$-closed. Finally, if $g\in M$ and $f\notin N$,
  i.e. $\Img(f)\subseteq [u,v]$, then $\Img(fg)\subseteq [u,v]$ and
  $\Img(gf)\subseteq [g(u),g(v)]$, so $fg,gf\notin N$.

  By Proposition~\ref{prop:bpp-non-auto-cont}, the topological
  semigroup $(\MM_{\Q},\TT_{pw})$ does not have automatic continuity
  with respect to the class of all subsemigroups of~$A^{A}$. Since all
  these subsemigroups are second countable, $(\MM_{\Q},\TT_{pw})$ in
  particular does not have automatic continuity with respect to the
  class of second countable topological semigroups.
\end{proof}
On the other hand, $\GG_{\Q}$ with the pointwise topology does have
automatic continuity by the following result by Rosendal and Solecki
(which we again reformulate to fit our notation).
\begin{theorem}[{\cite[Corollary~5]{RosendalSolecki}} combined with
  the remarks
  before~{\cite[Corollary~3]{RosendalSolecki}}]\label{thm:aut-Q-auto-cont-grp}
  $(\GG_{\Q},\TT_{pw})$ has automatic continuity with respect to the
  class of second countable topological groups.
\end{theorem}
Explicitly, this means: If $(H,\OO)$ is a second countable topological
group, then any group homomorphism
$\varphi\colon (\GG_{\Q},\TT_{pw})\to (H,\OO)$ is
continuous. Considering $\GG_{\Q}$ as a semigroup and forgetting that
it is in fact a group, one can ask whether it is sufficient that $H$
be a second countable topological semigroup and $\varphi$ be
multiplicative. By~\cite[Proposition~4.1]{EJMMMP-zariski}, the notions
of automatic continuity with respect to the classes of second
countable topological groups and second countable topological
semigroups are indeed equivalent, so we obtain:
\begin{proposition}\label{prop:aut-Q-auto-cont-semigrp}
  $(\GG_{\Q},\TT_{pw})$ has automatic continuity with respect to the
  class of second countable topological semigroups, explicitly: If
  $(H,\OO)$ is a second countable topological semigroup, then any
  semigroup homomorphism $(\GG_{\Q},\TT_{pw})\to (H,\OO)$ is
  continuous.
\end{proposition}

\subsection{Coarsest topology}
\label{subsec:coarsest-topology}

As it turns out, even very mild assumptions to a topology on
$\MM_{\Q}$ are enough for ascertaining that the pointwise topology is
coarser than the given topology. This was shown in~\cite{EJMPP-polish}
using a construction from~\cite{EJMMMP-zariski}. In order to keep the
present paper as self-contained as possible, we include a proof for
our special case (we refer to~\cite{EJMMMP-zariski} for more details,
in particular on the abstract properties that are used and on the link
to the Zariski topology mentioned in
Subsection~\ref{subsec:known-technique}):
\begin{theorem}[{\cite[Lemma~5.1]{EJMMMP-zariski}} and
  {\cite[Theorem~3.3]{EJMPP-polish}}]\label{thm:ejm-pointw-coarser}
  Let $\TT$ be a Polish semigroup topology on $\MM_{\Q}$. Then
  $\TT_{pw}\subseteq\TT$, i.e.~$\TT_{pw}$ is coarser than $\TT$.
\end{theorem}
\begin{proof}
  Consider the maps $h_{y}\colon\Q\to\Q$ defined for all $y\in\Q$ by
  \begin{displaymath}
    h_{y}(x):=
    \begin{cases}
      y-1,& x<y\\
      y,& x=y\\
      y+1,& x>y
    \end{cases}
  \end{displaymath} as well as the constant functions $c_{x}$
  with value $x\in\Q$. 
  Clearly, both $h_{y}$ and $c_{x}$ are elements of~$\MM_{\Q}$.

  It suffices to show that the subbasic open sets
  $\set{f\in\MM_{\Q}}{f(x)=y}$, where $x,y\in\Q$, are contained in
  $\TT$. In other words, we must prove that
  $\set{f\in\MM_{\Q}}{f(x)\neq y}$ is $\TT$-closed. For
  $f\in\MM_{\Q}$, we observe
  \begin{displaymath}
    f(x)\neq y\Leftrightarrow fc_{x}\neq c_{y}\Leftrightarrow
    h_{y}fc_{x}\neq c_{y}\Leftrightarrow h_{y}fc_{x}\in\{c_{y-1},c_{y+1}\}.
  \end{displaymath}
  The finite set $\{c_{y-1},c_{y+1}\}$ is $\TT$-closed since $\TT$ is
  Polish (and thus satisfies the first separation axiom). Using the
  $\TT$-continuity of the translations $\lambda_{h_{y}}$ and
  $\rho_{c_{x}}$, we obtain that
  \begin{displaymath}
    \set{f\in\MM_{\Q}}{f(x)\neq y}=\lambda_{h_{y}}^{-1}(\rho_{c_{x}}^{-1}(\{c_{y-1},c_{y+1}\}))
  \end{displaymath}
  is closed as well.
\end{proof}

\subsection{Back\&Forth}
\label{sec:back-and-forth}
In our proofs, we will repeatedly use the ``Back\&Forth'' method, see
for instance~\cite{Hodges}.
\begin{definition}\label{def:back-and-forth}
  Let $\X$ and $\Y$ be countably infinite structures in the same
  language and let $\mathcal{S}$ be a set of finite partial homomorphisms from
  $\X$ to $\Y$.
  \begin{enumerate}[label=(\roman*)]
  \item $\mathcal{S}$ is a \emph{Forth system} between $\X$ and $\Y$ if for
    all $m\in\mathcal{S}$ and all $x\in\X$ with $x\notin\Dom(m)$, there exists
    $m'\in\mathcal{S}$ such that $m'$~extends~$m$ and $x\in\Dom(m')$.
  \item $\mathcal{S}$ is a \emph{Back system} between $\X$ and $\Y$ if for all
    $m\in\mathcal{S}$ and all $y\in\Y$ with $y\notin\Img(m)$, there exists
    $m'\in\mathcal{S}$ such that $m'$~extends~$m$ and $y\in\Img(m')$.
  \item $\mathcal{S}$ is a \emph{Back\&Forth system} between $\X$ and $\Y$ if
    it is both a Back system and a Forth system.
  \end{enumerate}
\end{definition}
Iteratively extending finite partial homomorphisms so that their domains
exhaust the entire structure $\X$ (Forth) or in an alternating fashion
so that their domains and images exhaust $\X$ and $\Y$, respectively
(Back\&Forth), one obtains the following folklore result:
\begin{lemma}\label{lem:back-and-forth}
  Let $\X$ and $\Y$ be countably infinite structures in the same
  language.
  \begin{enumerate}[label=(\roman*)]
  \item If $\mathcal{S}$ is a Forth system between $\X$ and $\Y$ which is
    closed under restriction, then any $m\in\mathcal{S}$ can be extended to a
    total homomorphism $s\colon\X\to\Y$ such that every finite
    restriction of $s$ is contained in $\mathcal{S}$. In particular, if $\mathcal{S}$
    consists of injective finite partial homomorphisms, then $s$ can
    be picked to be injective as well.
  \item If $\mathcal{S}$ is a Back\&Forth system between $\X$ and $\Y$ which
    is closed under restriction, then any $m\in\mathcal{S}$ can be extended to
    a total and surjective homomorphism $s\colon\X\to\Y$ such that
    every finite restriction of $s$ is contained in $\mathcal{S}$. In
    particular, if $\mathcal{S}$ consists of finite partial isomorphisms, then
    $s$ can be picked to be an automorphism.
  \end{enumerate}
\end{lemma}
In the sequel, we will repeatedly need an answer to the following
question: given $s,f\in\MM_{\Q}$, under which conditions does there
exist a map $s'\in\MM_{\Q}$ such that $s=fs'$?
\begin{lemma}\label{lem:left-translation-preimage}
  Let $s,f\in\MM_{\Q}$ such that $\Img(s)\subseteq\Img(f)$.
  \begin{enumerate}[label=(\roman*)]
  \item\label{item:left-translation-preimage-i} Any finite partial
    increasing map $m_{0}$ from $\Q$ to $\Q$ satisfying
    $s(p)=fm_{0}(p)$ for all $p\in\Dom(m_{0})$ can be extended to
    $s'\in\MM_{\Q}$ with $s=fs'$.
  \item\label{item:left-translation-preimage-ii} Additionally suppose
    that for each $w\in\Img(f)$ the preimage $f^{-1}\{w\}$ is an
    irrational interval. Then any finite partial increasing
    \emph{injective} map $m_{0}$ from $\Q$ to $\Q$ satisfying
    $s(p)=fm_{0}(p)$ for all $p\in\Dom(m_{0})$ can be extended to an
    \emph{injective} $s'\in\MM_{\Q}$ with $s=fs'$.
  \end{enumerate}
\end{lemma}
\begin{proof}
  The proofs of both statements are almost parallel: one verifies that
  the system $\mathcal{S}$ of all finite partial increasing
  [for~\ref{item:left-translation-preimage-ii}: strictly increasing]
  maps $m$ from $\Q$ to $\Q$ satisfying $s(p)=fm(p)$ for all
  $p\in\Dom(m)$ is a Forth system and applies
  Lemma~\ref{lem:back-and-forth}.
\end{proof}

In Section~\ref{sec:reduct-pointw-topol}, we will also employ the
following variant.
\begin{definition}\label{def:back-and-forth-variant}
  Let $\X$ and $\Y$ be countably infinite structures in the same
  language and let $A\subseteq\X$ as well as $C\subseteq\Y$. Let
  further $\mathcal{S}$ be a set of finite partial homomorphisms from $\X$ to
  $\Y$.
  \begin{enumerate}[label=(\roman*)]
  \item $\mathcal{S}$ is an \emph{$(A,C)$-Back system} between $\X$ and $\Y$
    if the following holds:

    \noindent For all $m\in\mathcal{S}$ and all\footnote{Note: Contrary to
      ``Back'' from above, $y\in\Img(m)$ is in general possible!}
    $y\in C$, there exists $m'\in\mathcal{S}$ such that $m'$ extends $m$ and
    $\exists x\in A\cap\Dom(m')\colon m'(x)=y$.
  \item $\mathcal{S}$ is an \emph{$(A,C)$-Back\&Forth system} between $\X$ and
    $\Y$ if it is both an $(A,C)$-Back system and a Forth system.
  \end{enumerate}
\end{definition}
\begin{lemma}\label{lem:back-and-forth-variant}
  Let $\X$ and $\Y$ be countably infinite structures in the same
  language and let $A\subseteq\X$ as well as $C\subseteq\Y$. If $\mathcal{S}$
  is an $(A,C)$-Back\&Forth system between $\X$ and $\Y$, then any
  $m\in\mathcal{S}$ can be extended to a total homomorphism $s\colon\X\to\Y$
  such that
  \begin{displaymath}
    \forall y\in C\colon s^{-1}\{y\}\cap A\neq\emptyset.
  \end{displaymath}
\end{lemma}
\begin{proof}
  The argument proceeds in almost the same way as a standard
  Back\&Forth construction: Instead of applying a Back step to all
  elements of $\Y\setminus\Img(m)$, one applies an $(A,C)$-Back step
  to all elements of $C$ (even if they are contained in~$\Img(m)$).
\end{proof}

\section{Finest topology -- strategy and definitions}
\label{sec:finest-topol-strat}
In this section, we elaborate on our strategy for the proof of
Theorem~\ref{thm:end-Q-upp} by introducing our generalisation of
Property~\textbf{X} called \emph{\PPXb{}} and defining the \emph{rich}
topology on $\MM_{\Q}$.
\subsection{Pseudo-Property~$\mathbf{\overline{X}}$}
\label{subsec:ppxb}
\begin{definition}\label{def:ppxb}
  Let $S$ be a monoid with neutral element $1_{S}$ endowed with a
  topology\footnote{Note: $(S,\TT)$ need not be a topological
    semigroup!} $\TT$, let $D\subseteq S$ be a subset of $S$ endowed
  with a topology $\TT_{D}$ and let $m\geq 1$. Then $(S,\TT)$ has
  \emph{\PPXb} of length $m$ with respect to $(D,\TT_{D})$ if the
  following holds: For all $s\in S$ there exist
  $e_{s},h_{s}^{(1)},\dots,h_{s}^{(m+1)}\in S$ and
  $a_{s}^{(1)},\dots,a_{s}^{(m)}\in D$ such that
  \begin{enumerate}[label=(\roman*)]
  \item\label{item:ppxb-i} $e_{s}$ is \emph{left-invertible} in $S$,
    i.e.~there exists $p\in S$ such that $pe_{s}=1_{S}$.
  \item\label{item:ppxb-ii}
    $e_{s}s=h_{s}^{(m+1)}a_{s}^{(m)}h_{s}^{(m)}a_{s}^{(m-1)}\dots
    a_{s}^{(1)}h_{s}^{(1)}$.
  \item\label{item:ppxb-iii} For all $V^{(1)},\dots,V^{(m)}\in\TT_{D}$
    with $a_{s}^{(i)}\in V^{(i)}$, there exists $U\in\TT$ with
    $s\in U$ such that
    \begin{displaymath}
      e_{s}U\subseteq h_{s}^{(m+1)}V^{(m)}h_{s}^{(m)}V^{(m-1)}\dots V^{(1)}h_{s}^{(1)}.
    \end{displaymath}
  \end{enumerate}
\end{definition}
\begin{remark}\label{rem:ppxb}
  \PPXb{} of length~$m$ can thus be verified as follows: Given
  $s\in S$, we find suitable
  $e_{s},h_{s}^{(1)},\dots,h_{s}^{(m+1)}\in S$ with $e_{s}$
  left-invertible and devise a method to write
  $e_{s}s=h_{s}^{(m+1)}a_{s}^{(m)}h_{s}^{(m)}a_{s}^{(m-1)}\dots
  a_{s}^{(1)}h_{s}^{(1)}$ (where $a_{s}^{(i)}\in D$) in such a way
  that for arbitrary $\TT_{D}$-neighbourhoods $V^{(i)}$ of
  $a_{s}^{(i)}$, there exists a $\TT$-neighbourhood $U$ of $s$ such
  that our method applied to any $\tilde{s}\in U$ yields
  $\tilde{a}^{(i)}\in V^{(i)}$ (not just $\tilde{a}^{(i)}\in D$) with
  $e_{s}\tilde{s}=h_{s}^{(m+1)}\tilde{a}^{(m)}h_{s}^{(m)}\tilde{a}^{(m-1)}\dots
  \tilde{a}^{(1)}h_{s}^{(1)}$. Thus, this neighbourhood $U$ must be
  small enough to ensure two properties: first, it must encode enough
  information about $s$ to make sure that the \emph{same} auxiliary
  elements $e_{s},h_{s}^{(1)},\dots,h_{s}^{(m+1)}$ can be used for
  $\tilde{s}$; second, it must ascertain that $\tilde{s}$ is ``close
  enough'' to $s$ so that the resulting elements $\tilde{a}^{(i)}$ are
  ``close enough'' to $a_{s}^{(i)}$. Note the following equilibrium at
  the heart of \PPXb: Increasing the length $m$, it becomes easier to
  decompose a large class of elements $s$ in the desired
  form. However, there are more conditions
  $\tilde{a}^{(i)}\in V^{(i)}$ to be taken care of, potentially
  interacting with each other and yielding a more complex situation.

  The notation $\mathbf{\overline{X}}$ instead of \textbf{X} refers to
  the arbitrary number $m$ of elements of $D$ on the right hand side,
  while the term ``Pseudo'' refers to the composition with the
  left-invertible element $e_{s}$ on the left hand side,
  see~\cite{BartoPinskerDichotomy,Pseudo-loop,Topo}. Thus, the
  ``traditional'' Property~\textbf{X} from Definition~\ref{def:prop-X}
  corresponds in our terminology to Property~$\mathbf{\overline{X}}$
  of length 1 (without ``Pseudo'').
\end{remark}
We will apply \PPXb{} via the following proposition which generalises
parts of~\cite[Theorem~3.1]{EJMMMP-zariski}:
\begin{proposition}\label{prop:ppxb-auto-cont-lifting}
  Let $S$ be a monoid endowed with a topology $\TT$ and let
  $D\subseteq S$ be a subset of $S$ endowed with a topology
  $\TT_{D}$. If $(S,\TT)$ has \PPXb{} with respect to $(D,\TT_{D})$,
  then the following statements hold:
  \begin{enumerate}[label=(\roman*)]
  \item\label{item:ppxb-auto-cont-lifting-i} If $(H,\OO)$ is a
    topological semigroup and $\varphi\colon S\to H$ is a homomorphism
    such that the restriction $\varphi\vert_{D}$ is continuous as a
    map $\varphi\vert_{D}\colon (D,\TT_{D})\to (H,\OO)$, then
    $\varphi$ is continuous as a map
    $\varphi\colon (S,\TT)\to (H,\OO)$.
  \item\label{item:ppxb-auto-cont-lifting-ii} If $D$ is a semigroup
    such that $(D,\TT_{D})$ has automatic continuity with respect to a
    class~$\KK$ of topological semigroups, then $(S,\TT)$ also has
    automatic continuity with respect to~$\KK$.
  \end{enumerate}
\end{proposition}
\begin{proof}
  Since~\ref{item:ppxb-auto-cont-lifting-i} immediately
  implies~\ref{item:ppxb-auto-cont-lifting-ii}, we only prove the
  former.

  We denote the neutral element of $S$ by $1_{S}$. Without loss of
  generality, $\varphi$ is surjective. Therefore, $H$ can be assumed
  to be a monoid with neutral element $\varphi(1_{S})$. Let $O\in\OO$
  and $s\in S$ such that $\varphi(s)\in O$. We need to find $U\in\TT$
  such that $s\in U$ and $\varphi(U)\subseteq O$. Let $m$ be the
  length of \PPXb. Thus, there exist
  $e_{s},h_{s}^{(1)},\dots,h_{s}^{(m+1)}\in S$ and
  $a_{s}^{(1)},\dots,a_{s}^{(m)}\in D$ with $e_{s}$ left-invertible
  such that
  \begin{displaymath}
    e_{s}s=h_{s}^{(m+1)}a_{s}^{(m)}h_{s}^{(m)}a_{s}^{(m-1)}\dots
    a_{s}^{(1)}h_{s}^{(1)}
  \end{displaymath}
  and such that for arbitrary $V^{(1)},\dots,V^{(m)}\in\TT_{D}$ with
  $a_{s}^{(i)}\in V^{(i)}$, there exists $U\in\TT$ with $s\in U$
  satisfying
  \begin{displaymath}
    e_{s}U\subseteq h_{s}^{(m+1)}V^{(m)}h_{s}^{(m)}V^{(m-1)}\dots
    V^{(1)}h_{s}^{(1)}.
  \end{displaymath}
  Denote the left inverse of $e_{s}$ by $p$. The left translations
  \begin{displaymath}
    \lambda_{\varphi(e_{s})}\colon (H,\OO)\to (H,\OO)\quad\text{and}\quad\lambda_{\varphi(p)}\colon (H,\OO)\to (H,\OO)
  \end{displaymath}
  are continuous (since $\OO$ is a semigroup topology). Further,
  \begin{displaymath}
    \lambda_{\varphi(e_{s})}\colon (H,\OO)\to
    (\varphi(e_{s})H,\OO\vert_{\varphi(e_{s})H})\quad\text{and}\quad
    \lambda_{\varphi(p)}\colon
    (\varphi(e_{s})H,\OO\vert_{\varphi(e_{s})H})\to
    (H,\OO)
  \end{displaymath}
  form inverse maps because $\varphi(p)$ is a left inverse of
  $\varphi(e_{s})$ -- here we use that $H$ is a monoid with neutral
  element $\varphi(1_{S})$. Thus,
  $\lambda_{\varphi(e_{s})}\colon (H,\OO)\to
  (\varphi(e_{s})H,\OO\vert_{\varphi(e_{s})H})$ is a homeomorphism and
  we obtain
  $\varphi(e_{s})O=\lambda_{\varphi(e_{s})}(O)=P\cap\varphi(e_{s})H$
  for some $P\in\OO$. Consequently,
  \begin{displaymath}
    \varphi(h_{s}^{(m+1)})\varphi(a_{s}^{(m)})\varphi(h_{s}^{(m)})\varphi(a_{s}^{(m-1)})\dots\varphi(a_{s}^{(1)})\varphi(h_{s}^{(1)})=\varphi(e_{s})\varphi(s)\in
    P\cap\varphi(e_{s})H.
  \end{displaymath}
  Using that the map
  $(b^{(1)},\dots,b^{(m)})\mapsto\varphi(h_{s}^{(m+1)})b^{(m)}\varphi(h_{s}^{(m)})b^{(m-1)}\dots
  b^{(1)}\varphi(h_{s}^{(1)})$ is continuous with respect to $\OO$
  (since $\OO$ is a semigroup topology) yields sets $W^{(i)}\in\OO$
  such that $\varphi(a_{s}^{(i)})\in W^{(i)}$ and
  \begin{displaymath}
    \varphi(h_{s}^{(m+1)})W^{(m)}\varphi(h_{s}^{(m)})W^{(m-1)}\dots
    W^{(1)}\varphi(h_{s}^{(1)})\subseteq P.
  \end{displaymath}
  By the assumed continuity of
  $\varphi\vert_{D}\colon (D,\TT_{D})\to (H,\OO)$, the preimages
  $V^{(i)}:=\varphi\vert_{D}^{-1}(W^{(i)})$ are contained in
  $\TT_{D}$. Thus, we can invoke \PPXb{} to obtain a set $U\in\TT$
  such that $s\in U$ and
  \begin{displaymath}
    e_{s}U\subseteq h_{s}^{(m+1)}V^{(m)}h_{s}^{(m)}V^{(m-1)}\dots
    V^{(1)}h_{s}^{(1)}.
  \end{displaymath}
  Applying $\varphi$, we conclude
  \begin{displaymath}
    \varphi(e_{s})\varphi(U)\subseteq \varphi(h_{s}^{(m+1)})W^{(m)}\varphi(h_{s}^{(m)})W^{(m-1)}\dots
    W^{(1)}\varphi(h_{s}^{(1)})\subseteq P,
  \end{displaymath}
  and thus
  $\varphi(e_{s})\varphi(U)\subseteq
  P\cap\varphi(e_{s})H=\varphi(e_{s})O$. Multiplying with $\varphi(p)$
  from the left, we obtain $\varphi(U)\subseteq O$ as desired.
\end{proof}
The previous proposition implies the already mentioned fact that
$(\MM_{\Q},\TT_{pw})$ \emph{cannot} have \PPXb{} \emph{of any length}
with respect to $(\GG_{\Q},\TT_{pw})$: for otherwise,
Proposition~\ref{prop:ppxb-auto-cont-lifting} in tandem with
Proposition~\ref{prop:aut-Q-auto-cont-semigrp} would yield that
$(\MM_{\Q},\TT_{pw})$ has automatic continuity with respect to the
class of second countable topological semigroups, in violation of
Proposition~\ref{prop:end-Q-non-auto-cont}. Therefore, we have to
improve our approach by enriching the topology on $\MM_{\Q}$ to make
\PPXb{} possible. To motivate, consider $s\in\MM_{\Q}$ which is
``unbounded on both sides'', i.e.~$\inf s=-\infty$ and
$\sup s=+\infty$. We strive for a representation of the form
$e_{s}s=h_{s}^{(m+1)}a_{s}^{(m)}h_{s}^{(m)}a_{s}^{(m-1)}\dots
a_{s}^{(1)}h_{s}^{(1)}$ with $e_{s}$ left-invertible. In particular,
$e_{s}$ has to be unbounded on both sides as well, thus so too is the
right hand side
$h_{s}^{(m+1)}a_{s}^{(m)}h_{s}^{(m)}a_{s}^{(m-1)}\dots
a_{s}^{(1)}h_{s}^{(1)}$ and consequently each $h_{s}^{(i)}$. For any
$V^{(i)}\subseteq\GG_{\Q}$, the set
$h_{s}^{(m+1)}V^{(m)}h_{s}^{(m)}V^{(m-1)}\dots V^{(1)}h_{s}^{(1)}$
therefore only contains functions which are unbounded on both
sides. Hence, any set $U$ such that
$e_{s}U\subseteq h_{s}^{(m+1)}V^{(m)}h_{s}^{(m)}V^{(m-1)}\dots
V^{(1)}h_{s}^{(1)}$ must consist of such functions. Thus, in any
topology on $\MM_{\Q}$ which yields \PPXb, the set of all functions
which are unbounded on both sides must have nonempty interior. Similar
reasonings apply to the remaining kinds of ``boundedness behaviour''.

We define several types of subsets of $\MM_{\Q}$:
\begin{definition}\label{def:types}
  \hspace{0mm}
  \begin{enumerate}[label=(\arabic*),ref=\arabic*,start=0]
  \item\label{item:types-0}
    $O_{x,y}^{(0)}:=\set{s\in\MM_{\Q}}{s(x)=y}$;\hfill (pointwise)

    \quad$x,y\in\Q$
  \item\label{item:types-1}
    $O_{I,J}^{(1)}:=\set{s\in\MM_{\Q}}{s(I)\subseteq J}$;\hfill
    (generalised pointwise)

    \quad$I=(-\infty,p)$ and either $J=(-\infty,q]$ or $J=(-\infty,q)$
    OR

    \quad$I=(p,+\infty)$ and either $J=[q,+\infty)$ or
    $J=(q,+\infty)$\quad\quad\quad\quad for $p,q\in\Q$
  \end{enumerate}
  \begin{enumerate}[label=(\arabic*\textsuperscript{cls}),ref=\arabic*\textsuperscript{cls},start=1]
  \item\label{item:types-1cls}
    $O_{I,J}^{(1)}:=\set{s\in\MM_{\Q}}{s(I)\subseteq J}$;\hfill
    (generalised pointwise,

    \quad$I=(-\infty,p)$ and $J=(-\infty,q]$ OR\hfill closed image
    constraint)

    \quad$I=(p,+\infty)$ and $J=[q,+\infty)$ \quad\quad\quad\quad for
    $p,q\in\Q$
  \end{enumerate}
  \begin{enumerate}[label=(\arabic*),ref=\arabic*,start=2]
  \item\label{item:types-2}
    $O_{LU}^{(2)}:=\set{s\in\MM_{\Q}}{\inf\im(s)\in L,\ \sup\im(s)\in
      U}$;\hfill (boundedness types)

    \quad $L=\R$ or $L=\{-\infty\}$ \quad AND \quad $U=\R$ or
    $U=\{+\infty\}$
    
    \noindent Explicitly, these are the following four sets:
    
    \noindent
    $O_{\R,\R}^{(2)}:=\set{s\in\MM_{\Q}}{\inf\im(s)\in\R,\
      \sup\im(s)\in\R}$\hfill (bounded-bounded)

    \noindent
    $O_{-\infty,\R}^{(2)}:=\set{s\in\MM_{\Q}}{\inf\im(s)=-\infty,\
      \sup\im(s)\in\R}$\hfill (unbounded-bounded)

    \noindent
    $O_{\R,+\infty}^{(2)}:=\set{s\in\MM_{\Q}}{\inf\im(s)\in\R,\
      \sup\im(s)=+\infty}$\hfill (bounded-unbounded)

    \noindent
    $O_{-\infty,+\infty}^{(2)}:=\set{s\in\MM_{\Q}}{\inf\im(s)=-\infty,\
      \sup\im(s)=+\infty}$\hfill (unbounded-unbounded)
  \item\label{item:types-3}
    $O_{K}^{(3)}:=\set{s\in\MM_{\Q}}{\Img(s)\cap K=\emptyset}$;\hfill
    (avoiding)

    \quad $K=[q_1,q_2]$ or $K=[q_1,q_2)$ or

    \quad $K=(q_1,q_2]$ or $K=(q_1,q_2)$ \quad\quad\quad\quad for
    $q_{1},q_{2}\in\Q\cup\{\pm\infty\}$, $q_{1}\leq q_{2}$
  \end{enumerate}
  \begin{enumerate}[label=(\arabic*\textsuperscript{opn}),ref=\arabic*\textsuperscript{opn},start=3]
  \item\label{item:types-3opn}
    $O_{K}^{(3)}:=\set{s\in\MM_{\Q}}{\Img(s)\cap K=\emptyset}$;\hfill
    (avoiding, open constraint)

    \quad $K=(q_1,q_2)$\qquad for $q_1,q_2\in\Q\cup\{\pm\infty\}$,
    $q_1<q_2$
  \end{enumerate}
\end{definition}
We mention explicitly that the sets formed analogously to
type~\ref{item:types-1} but with closed intervals $I$ are already
encompassed by type~\ref{item:types-0}, i.e.~the pointwise
topology. For instance, if $I=(-\infty,p]$ and $J=(-\infty,q]$ or
$J=(-\infty,q)$, then
$\set{s\in\MM_{\Q}}{s(I)\subseteq J}=\bigcup_{y\in
  J}\set{s\in\MM_{\Q}}{s(p)=y}$. We will make use of this fact in
Section~\ref{sec:reduct-pointw-topol}.

The types of sets defined above yield a template for constructing
topologies.
\begin{definition}\label{def:rich-top}
  If $M\subseteq\{0,1,1^{cls},2,3,3^{opn}\}$, then $\TT_{M}$ is the
  topology generated by the sets of the types occurring in $M$. We
  further define the \emph{rich topology} $\TT_{rich}:=\TT_{0123}$;
  explicitly, this is the topology generated by
  \begin{align*}
    \set{O_{x,y}^{(0)}}{x,y\in\Q}&\cup\set{O_{I,J}^{(1)}}{I=(-\infty,p),J\in\{(-\infty,q],(-\infty,q)\},p,q\in\Q}\\
                                 &\cup\set{O_{I,J}^{(1)}}{I=(p,+\infty),J\in\{[q,+\infty),(q,+\infty)\},p,q\in\Q}\\
                                 &\cup\left\{O_{\R,\R}^{(2)},O_{-\infty,\R}^{(2)},O_{\R,+\infty}^{(2)},O_{-\infty,+\infty}^{(2)}\right\}\\
                                 &\cup\set{O_{K}^{(3)}}{K\in\{[q_{1},q_{2}],[q_{1},q_{2}),(q_{1},q_{2}],(q_{1},q_{2})\},
                                   q_{1},q_{2}\in\Q\cup\{\pm\infty\},q_{1}\leq q_{2}}.
  \end{align*}
  If $x_{1},\dots,x_{n},y_{1},\dots,y_{n}\in\Q$, it will sometimes be
  convenient to abbreviate
  $O_{\bar{x},\bar{y}}^{(0)}:=\bigcap_{i=1}^{n}O_{x_{i},y_{i}}^{(0)}$.
\end{definition}
With this terminology, we can formulate our main technical results.
\begin{proposition}\label{prop:rich-top-ppxb}
  $(\MM_{\Q},\TT_{rich})$ has \PPXb{} of length~2 with respect to
  $(\GG_{\Q},\TT_{pw})$.
\end{proposition}
Let us note that we deem it unlikely that $\MM_{\Q}$ equipped with any
meaningful topology could have \PPXb{} of length~1 (so
Pseudo-Property~\textbf{X}) with respect to $(\GG_{\Q},\TT_{pw})$
since it is only the second automorphism which gives us enough
flexibility and control over discontinuity points (see
Definition~\ref{def:dis-continuity} for this notion).
\begin{proposition}\label{prop:reduct-pointw-topol}
  Let $\TT$ be a Polish semigroup topology on $\MM_{\Q}$ such that
  $\TT_{pw}\subseteq\TT\subseteq\TT_{rich}$. Then $\TT=\TT_{pw}$.
\end{proposition}
Before we get to their proofs, let us comment on how
Theorem~\ref{thm:end-Q-upp} follows from these results.
\begin{proof}[Proof (of Theorem~\ref{thm:end-Q-upp} given
  Propositions~\ref{prop:rich-top-ppxb}
  and~\ref{prop:reduct-pointw-topol}).]
  Let $\TT$ be a Polish semigroup topology on $\MM_{\Q}$. By
  Theorem~\ref{thm:ejm-pointw-coarser}, we obtain
  $\TT_{pw}\subseteq\TT$. On the other hand, we note that
  $(\GG_{\Q},\TT_{pw})$ has automatic continuity with respect to the
  class of second countable topological semigroups by
  Proposition~\ref{prop:aut-Q-auto-cont-semigrp}. Combining
  Proposition~\ref{prop:rich-top-ppxb} with
  Proposition~\ref{prop:ppxb-auto-cont-lifting}\ref{item:ppxb-auto-cont-lifting-ii}
  yields that $(\MM_{\Q},\TT_{rich})$ has automatic continuity with
  respect to the class of second countable topological semigroups as
  well. Since $(\MM_{\Q},\TT)$ is second countable, the identity map
  $\id\colon (\MM_{\Q},\TT_{rich})\to (\MM_{\Q},\TT)$ is therefore
  continuous, in other words $\TT\subseteq\TT_{rich}$. By
  Proposition~\ref{prop:reduct-pointw-topol}, we finally conclude
  $\TT=\TT_{pw}$.
\end{proof}
The proofs of Propositions~\ref{prop:rich-top-ppxb}
and~\ref{prop:reduct-pointw-topol} are the subject of
Sections~\ref{sec:rich-top-ppxb} and~\ref{sec:reduct-pointw-topol},
respectively.

In the former section, we will find \emph{generic} maps
$e,f,g,h\in\MM_{\Q}$ (with $e$ left-invertible) so that the
compositions $fahbg$ for $a,b\in\GG_{\Q}$ exhaust the maps $es$ for a
great variety of $s\in\MM_{\Q}$. Further, we will -- roughly speaking
-- analyse how the compositions $fahbg$ change with varying
$a,b\in\GG_{\Q}$. Some of the complexity arises from the requirement
that $a,b$ be automorphisms.

The latter section has a different flavour in that we can allow maps
to vary within $\MM_{\Q}$, yielding less intricate
constructions. Nonetheless, most (but not all) intermediate results
can be reformulated as a -- albeit easier --
Property~$\mathbf{\overline{X}}$-type statement, namely with respect
to the entire semigroup $\MM_{\Q}$ (equipped with different
topologies) instead of $\GG_{\Q}$. One major exception
(Proposition~\ref{lem:023opn-squig-03opn}) crucially employs
regularity of the topology $\TT$ in combination with Polishness and
cannot be reformulated as a
(Pseudo\nobreakdash-)Property~$\mathbf{\overline{X}}$-type statement,
for good reason: if the proof of
Proposition~\ref{prop:reduct-pointw-topol} just consisted of a series
of such statements, we could start from
Proposition~\ref{prop:rich-top-ppxb} and repeatedly apply
Proposition~\ref{prop:ppxb-auto-cont-lifting} to show that
$(\MM_{\Q},\TT_{pw})$ has automatic continuity with respect to the
class of second countable topological semigroups, contradicting
Proposition~\ref{prop:end-Q-non-auto-cont}.
\section{The rich topology has
  Pseudo-Property~$\mathbf{\overline{X}}$}
\label{sec:rich-top-ppxb}
This section is devoted to proving
Proposition~\ref{prop:rich-top-ppxb}. With Remark~\ref{rem:ppxb} in
mind, we want to find a decomposition
$e_{s}s=f_{s}a_{s}h_{s}b_{s}g_{s}$ of a given $s\in\MM_{\Q}$ with
$e_{s},f_{s},g_{s},h_{s}\in\MM_{\Q}$ and $a_{s},b_{s}\in\GG_{\Q}$ as
well as a $\TT_{rich}$-neighbourhood $U$ of $s$ such that for any
$\tilde{s}\in U$, we can similarly decompose
$e_{s}\tilde{s}=f_{s}\tilde{a}h_{s}\tilde{b}g_{s}$ with
$\tilde{a},\tilde{b}\in\GG_{\Q}$ and the \emph{same} maps
$e_{s},f_{s},g_{s},h_{s}$. Given $\TT_{pw}$-neighbourhoods $V$ and $W$
of $a_{s}$ and $b_{s}$, respectively, we additionally have to make
sure that $U$ can be taken small enough that for any
$\tilde{s}\in U$, we can pick $\tilde{a}\in V$
and $\tilde{b}\in W$. This means that $\tilde{a}$ and $a_{s}$ need to
have the same behaviour on a given finite set, as do $\tilde{b}$ and
$b_{s}$.

We will proceed in three steps. First, we will derive ``compatibility
conditions'' such that $e_{s}s$ can be written in the form
$f_{s} a_{s}\iota_{s}$. These conditions exhibit such a tight
connection between $s$ and $\iota_{s}$ that $U$ can never force
$e_{s}\tilde{s}$ to satisfy these conditions for all $\tilde{s}\in U$
and a fixed $\iota_{s}$. In a second step, we will therefore expand
$\iota_{s}$ in the form $\iota_{s}=h_{s}b_{s}g_{s}$ for fixed
$g_{s},h_{s}\in\MM_{\Q}$ and varying $b_{s}\in\GG_{\Q}$, yielding
indeed $e_{s}s=f_{s}a_{s}h_{s}b_{s}g_{s}$. For $\tilde{s}\in U$, it
turns out that we can pick $\tilde{\iota}=h_{s}\tilde{b}g_{s}$ which
is compatible with $e_{s}\tilde{s}$ to obtain
$e_{s}\tilde{s}=f_{s}\tilde{a}\tilde{\iota}=f_{s}\tilde{a}h_{s}\tilde{b}g_{s}$. All
the while, we have to make sure that $\tilde{a}$ and $a_{s}$ as well
as $\tilde{b}$ and $b_{s}$ coincide on given finite sets, resulting in
a third major step.


\subsection{Generic surjections, generic injections, sparse injections
and basic formulas}
\begin{definition}\label{def:dis-continuity}
  Let $s\in\MM_{\Q}$. We set
  \begin{align*}
    \Cont(s)&:=\set{\gamma\in\R}{\sup s(-\infty,\gamma)=\inf
              s(\gamma,+\infty)}\\
    \Dc(s)&:=\set{\gamma\in\R}{\sup s(-\infty,\gamma)<\inf
            s(\gamma,+\infty)},
  \end{align*}
  the sets of \emph{continuity points} and \emph{discontinuity points}
  of $s$, respectively. Additionally, we write
  $\Dc^{\I}(s):=\Dc(s)\cap\I$ for notational simplicity. Finally, we
  extend $s$ to an increasing map $\bar{s}\colon\R\to\R$ by setting
  $\bar{s}(\gamma):=\sup s(-\infty,\gamma)$ for all $\gamma\in\I$.
\end{definition}
We will frequently use the notion of limit points in the following
sense:
\begin{definition}
  Let $A\subseteq\Q$ and $\gamma\in\R$. We say that $\gamma$ is a
  \emph{limit point} of $A$ if $\gamma$ is contained in the closure of
  $A\setminus\{\gamma\}$ with respect to the standard topology on
  $\R$. The set of all limit points of $A$ will be denoted by
  $\LP(A)$. If $s\in\MM_{\Q}$, we will abbreviate $\LP(\Img(s))$ as
  $\LP(s)$ for better readability.
\end{definition}
We collect a few easy facts:
\begin{lemma}\label{lem:cont-lp-easy-facts}
  Let $s\in\MM_{\Q}$.
  \begin{enumerate}[label=(\roman*)]
  \item\label{item:cont-lp-easy-facts-i} $\Dc(s)$ is at most
    countable.
  \item\label{item:cont-lp-easy-facts-ii} If $s$ is injective and
    $\LP(s)\subseteq\I$, then $\Q\subseteq\Dc(s)$ and
    $\bar{s}(\I)\subseteq\I$. In fact,
    $\sup s(-\infty,q)<s(q)<\inf s(q,+\infty)$ for all
    $q\in\Q$. Additionally, $(\R\setminus\Img(\bar{s}))\cap\I$ is
    topologically dense in $\R$.
  \item\label{item:cont-lp-easy-facts-iii} If $b\in\GG_{\Q}$, then
    $\Dc(b)=\emptyset$ and $\bar{b}\colon\R\to\R$ is a strictly
    increasing bijection with $\bar{b}(\I)\subseteq\I$. Additionally,
    any increasing extension $\beta$ of $b$ to a set $M\subseteq\R$
    coincides with $\bar{b}\vert_{M}$.
  \end{enumerate}
\end{lemma}
We define three kinds of \emph{generic} maps in $\MM_{\Q}$.
\begin{definition}
  \hspace{0mm}
  \begin{enumerate}[label=(\roman*)]
  \item A map $f\in\MM_{\Q}$ is called a \emph{generic surjection} if
    it is surjective and for each $q\in\Q$, the preimage $f^{-1}\{q\}$
    is an irrational interval, i.e.~$f^{-1}\{q\}=(r_{q},t_{q})$ for
    $r_{q},t_{q}\in\I$.
  \item A map $g\in\MM_{\Q}$ is called a \emph{generic injection} if
    it is injective and unbounded-unbounded with
    $\Dc^{\I}(g)=\emptyset$ and $\LP(g)\subseteq\I$.
  \item A map $h\in\MM_{\Q}$ is called a \emph{sparse injection} if it
    is injective, $\Dc^{\I}(h)$ is topologically dense in $\R$ and
    $\LP(h)\subseteq\I$.
  \end{enumerate}
\end{definition}
It is an easy observation that such maps really exist.
\begin{lemma}\label{lem:existence-generic}
  \hspace{0mm}
  \begin{enumerate}[label=(\roman*)]
  \item\label{item:existence-generic-i} For every $A\subseteq\Q$,
    there exists a map $f\in\MM_{\Q}$ with $\Img(f)=A$ such that the
    $f$-preimages of single elements are irrational intervals. In
    particular, there exists a generic surjection.
  \item\label{item:existence-generic-ii} For every finite or countably
    infinite $A\subseteq\I$ and every boundedness type $O_{LU}^{(2)}$,
    there exists an injective map $\iota\in O_{LU}^{(2)}$ which
    satisfies $\Dc(\iota)=A\cupdot\Q$ as well as
    $\LP(\iota)\subseteq\I$.
  \item\label{item:existence-generic-iii} There exists a generic
    injection in $\MM_{\Q}$.
  \item\label{item:existence-generic-iv} There exists a sparse
    injection in $\MM_{\Q}$ of any boundedness type.
  \end{enumerate}
\end{lemma}
\begin{proof}
  $ $\par\nobreak\ignorespaces\textbf{(i).} We put $M:=A\times\Q$ and
  set $<_{M}$ to be the lexicographic order on $M$ where the first
  component is the significant one. Define $\pi\colon M\to\Q$ by
  $\pi(w,q):=w$. Since $(M,<_{M})$ is countably infinite and densely
  ordered without greatest or least element, there exists an order
  isomorphism $\alpha\colon\Q\to M$. Setting $f:=\pi\circ\alpha$, we
  obtain a map as desired.

  \textbf{(ii).} We only consider the case
  $O_{LU}^{(2)}=O_{\R,+\infty}$; the others are treated
  analogously. Put
  \begin{displaymath}
    M:=(A\cupdot\Q\cupdot\{-\infty\})\times\Q
  \end{displaymath}
  and set $<_{M}$ to be the lexicographic order on $M$ where the first
  component is the significant one. Define $j\colon\Q\to M$ by
  $j(x):=(x,0)$. Since $(M,<_{M})$ is countably infinite and densely
  ordered without greatest or least element, there exists an order
  isomorphism $\beta\colon M\to\Q$. Setting
  $\iota:=\beta\circ j\in\MM_{\Q}$, we obtain a map as desired.

  \textbf{(iii).} Setting $A=\emptyset$ as well as
  $O_{LU}^{(2)}=O_{-\infty,+\infty}^{(2)}$ and
  using~\ref{item:existence-generic-ii}, we obtain a generic
  injection.

  \textbf{(iv).} Setting $A\subseteq\I$ to be a countably infinite
  topologically dense set and using~\ref{item:existence-generic-ii},
  we obtain a sparse injection with any boundedness type.
\end{proof}

Another useful notion is given by the \emph{generalised inverse} of
maps in $\MM_{\Q}$.
\begin{definition}\label{def:gen-inv}
  Let $s\in\MM_{\Q}$ and $y\in\Q$. Define\footnote{We put
    $\sup\emptyset:=-\infty$ and $\inf\emptyset:=+\infty$.}
  $s^{L}(y):=\sup s^{-1}(-\infty,y)\in\R\cup\{\pm\infty\}$ and
  $s^{R}(y):=\inf s^{-1}(y,+\infty)\in\R\cup\{\pm\infty\}$. If
  $s^{L}(y)$ and $s^{R}(y)$ coincide, we define
  $s^{\dagger}(y):=s^{L}(y)=s^{R}(y)$ (the \emph{generalised inverse}
  of $s$ at $y$).
\end{definition}
The following observations are easily deduced directly from the
definitions.
\begin{lemma}\label{lem:easy-gen-inv}
  \hspace{0mm}
  \begin{enumerate}[label=(\roman*)]
  \item\label{item:easy-gen-inv-i} Let $s\in\MM_{\Q}$. If
    $y\in\Img(s)$ and if $x\in\Q$ is the only $s$-preimage of $y$,
    then $s^{\dagger}(y)=x$.
  \item\label{item:easy-gen-inv-ii} Let $s\in\MM_{\Q}$ be
    injective. Then $s^{\dagger}(y)$ is a welldefined real number for
    all elements $y\in (\inf s,\sup s)$.
  \item\label{item:easy-gen-inv-iii} Let $g\in\MM_{\Q}$ be a generic
    injection. Then $g^{\dagger}(y)\in\Q$ for all $y\in\Q$. In
    particular, $g$ is left-invertible in $\MM_{\Q}$ with left inverse
    $g^{\dagger}$. Moreover, for all rational intervals $J$ with
    boundary points $q_{1}$ and $q_{2}$ in $\Q\cup\{\pm\infty\}$, the
    preimage $g^{-1}(J)$ is again a rational interval with boundary
    points\footnote{Note, however, that e.g.~$g^{-1}[q_{1},q_{2}]$
      need not be closed.}  $g^{\dagger}(q_{1})$ and
    $g^{\dagger}(q_{2})$.
  \item\label{item:easy-gen-inv-iv} Let $g\in\MM_{\Q}$ be a generic
    injection. Then the translation $\lambda_{g}\colon s\mapsto gs$ is
    continuous as a map\footnote{Since $\TT_{rich}$ is not a semigroup
      topology -- a fact on which most of
      Section~\ref{sec:reduct-pointw-topol} hinges -- this cannot be
      taken for granted and depends on the genericity of $g$.}
    $(\MM_{\Q},\TT_{rich})\to (\MM_{\Q},\TT_{rich})$.

    \noindent (combine the $\TT_{pw}$-continuity of $\lambda_{g}$
    with~\ref{item:easy-gen-inv-iii} and the fact that $s$ and $gs$
    have the same boundedness type since $g$ is unbounded-unbounded)
  \end{enumerate}
\end{lemma}
As mentioned in the introduction of Section~\ref{sec:rich-top-ppxb},
we will derive compatibility conditions such that
$e_{s}s=f_{s}a_{s}\iota_{s}$ for maps
$s,e_{s},f_{s},\iota_{s}\in\MM_{\Q}$ and $a_{s}\in\GG_{\Q}$. It will
be convenient to consider a slightly more general situation and aim
for $\sigma=\pi a\iota$; this will then be applied for $\pi=f_{s}$ and
$\sigma:=e_{s}s$, and later on for
$\tilde{\sigma}:=e_{s}\tilde{s}$. Again by the introductory remarks,
we will need to make sure that the function $a$ maps
$\bar{x}\mapsto\bar{y}$ for given tuples $\bar{x},\bar{y}$.

First, we reformulate our problem in model-theoretic language. As a
starting point, note that $\sigma=\pi a\iota$ is equivalent to the
fact that $a(\iota(q))\in\pi^{-1}\{\sigma(q)\}$ for all
$q\in\Q$. Since $\pi$ is increasing, the preimage
$\pi^{-1}\{\sigma(q)\}$ is an interval. Thus, if
$\sigma(q)=\sigma(q')$ for some $q,q'\in\Q$, then not only do
$\iota(q)$ and $\iota(q')$ have to be mapped to the same interval, but
all points \emph{between} $\iota(q)$ and $\iota(q')$ have to be as
well. This motivates the following definition:
\begin{definition}\label{def:aux-str-prop-x}
  \hspace{0mm}
  \begin{enumerate}[label=(\roman*)]
  \item Let $\set{P_{q}}{q\in\Q}$ be a set of unary relation symbols
    and define the language $L$ by\footnote{Note: $<$ instead of
      $\leq$!}  $L:=\{<\}\cup\set{P_{q}}{q\in\Q}$.
  \item Let $\sigma,\pi,\iota\in\MM_{\Q}$. For $q\in\Q$, we set
    \begin{align*}
      P_{q}^{\A}&:=\text{convex hull of }\iota(\sigma^{-1}\{\sigma(q)\})\\
      P_{q}^{\B}&:=\pi^{-1}\{\sigma(q)\}
    \end{align*}
    and define $L$-structures
    $\A=\left\langle\Q,<,(P_{q}^{\A})_{q\in\Q}\right\rangle$ and
    $\B=\left\langle\Q,<,(P_{q}^{\B})_{q\in\Q}\right\rangle$. If
    $\sigma,\pi,\iota$ are not clear from the context, we will write
    $\A(\sigma,\pi,\iota)$ and $\B(\sigma,\pi,\iota)$.
  \end{enumerate}
\end{definition}
In the sequel, $L$, $\A$ and $\B$ will always denote the objects just
defined. Note that a surjective $L$-homomorphism $a\colon\A\to\B$ is
automatically contained in $\GG_{\Q}$ and satisfies
$\sigma=\pi a\iota$. Thus, our aim is to construct a surjective
$L$-homomorphism extending a given map $\bar{x}\mapsto\bar{y}$. We
will do so using the Back\&Forth method, see
Subsection~\ref{sec:back-and-forth}.
\begin{definition}\label{def:basic-formula}
  A formula $\psi(\bar{z})$ over $L$ is called \emph{basic} if it is
  one of the formulas
  \begin{enumerate}[label=(\roman*)]
  \item\label{item:basic-formula-i} $P_{q}(z_{i})$, \quad $q\in\Q$
  \item\label{item:basic-formula-ii} $z_{i}<z_{j}$
  \item\label{item:basic-formula-iii}
    $L_{q}(z_{i}):\leftrightarrow\exists u\colon u<z_{i}\land P_{q}(u)$, \quad
    $q\in\Q$
  \item\label{item:basic-formula-iv}
    $R_{q}(z_{i}):\leftrightarrow\exists u\colon u>z_{i}\land P_{q}(u)$, \quad
    $q\in\Q$
  \end{enumerate}
  For a basic formula $\psi(\bar{z})$ and a tuple $\bar{x}$ in $\A$,
  we write $\A\models\psi(\bar{x})$ if $\bar{x}$ satisfies the formula
  $\psi(\bar{z})$ in $\A$; we analogously define
  $\B\models\psi(\bar{y})$. If $m$ is a (potentially partial) map from
  $\A$ to $\B$, then $m$ is said to \emph{preserve} $\psi(\bar{z})$ if
  $\A\models\psi(\bar{x})$ implies $\B\models\psi(m(\bar{x}))$ for all
  tuples $\bar{x}$ in the domain of $m$.
\end{definition}
Note that basic formulas contain only existential and no universal
quantifiers, so total homomorphisms $\A\to\B$ always preserve all
basic formulas. In the following, we will work with partial maps from
$\A$ to $\B$ preserving all basic formulas, either extending maps
without losing that property or analysing when a given map indeed
preserves all basic formulas.

\subsection{Proving Proposition~\ref{prop:rich-top-ppxb}}
The crucial technical results necessary for the proof of
Proposition~\ref{prop:rich-top-ppxb} are three lemmas, one for each of
the steps mentioned in the introduction of
Section~\ref{sec:rich-top-ppxb}: the Sandwich
Lemma~\ref{lem:sandwich}, the Preconditioning
Lemma~\ref{lem:preconditioning} and the Variation
Lemma~\ref{lem:variation}. In this subsection, we formulate them and
demonstrate how they are used to show
Proposition~\ref{prop:rich-top-ppxb}. For proofs of the three lemmas,
we refer to the following subsections. We start by fixing some
notation for the sake of brevity:
\begin{definition}
  We say that $\sigma,\pi,\iota\in\MM_{\Q}$ are \emph{compatible} if
  \begin{enumerate}[label=(\alph*)]
  \item\label{item:compatible-a} $\sigma\in\MM_{\Q}$ satisfies
    $\LP(\sigma)\subseteq\I$,
  \item\label{item:compatible-b} $\pi\in\MM_{\Q}$ is a generic
    surjection,
  \item\label{item:compatible-c} $\iota\in\MM_{\Q}$ is injective with
    $\LP(\iota)\subseteq\I$, has the same boundedness type as
    $\sigma$ and satisfies $\Dc^{\I}(\iota)=\Dc^{\I}(\sigma)$.
  \end{enumerate}
\end{definition}
\begin{lemma}[Sandwich Lemma]\label{lem:sandwich} Let
  $\sigma,\pi,\iota\in\MM_{\Q}$ be \emph{compatible}.

  Then the following statements hold:
  \begin{enumerate}[label=(\roman*)]


  \item\label{item:sandwich-conclusion-i} The set of all finite
    partial $L$-homomorphisms $m$ from $\A$ to $\B$ preserving all
    basic formulas is a Back\&Forth system.
  \item\label{item:sandwich-conclusion-ii} There exists $a\in\GG_{\Q}$
    such that $\sigma=\pi a\iota$. Indeed, if $m$ is a finite partial
    $L$-homomorphism from $\A$ to $\B$ preserving all basic formulas,
    there exists $a\in\GG_{\Q}$ extending $m$ such that
    $\sigma=\pi a\iota$.
  \end{enumerate}
\end{lemma}
Referring back to the overview presented in the introduction of
Section~\ref{sec:rich-top-ppxb}, we can now precisely state why our
approach requires aiming for \PPXb{} of length~2: to apply the
Sandwich Lemma~\ref{lem:sandwich}, we need that $\sigma=e_{s}s$ and
$\iota$ have the same irrational discontinuity points, so the
irrational discontinuity points of $s$ and $\iota$ need to be closely
connected. Since no $\TT_{rich}$-neighbourhood $U$ can encode
$\Dc^{\I}(s)$, we cannot use a fixed map $\iota$ for all $\tilde{s}$
in $U$. Thus, we need to adapt $\iota$ to $\tilde{s}$. We will write
$\iota=hbg$, where $b$ varies in $\GG_{\Q}$ and $g,h\in\MM_{\Q}$ are
fixed elements. As it will turn out, it is crucial that we are very
free in stipulating finite pointwise behaviour not only of $b$ on $\Q$
but also of the extension $\bar{b}$ on $\I$.
\begin{lemma}[Preconditioning Lemma]\label{lem:preconditioning}
  Let $g\in\MM_{\Q}$ be a generic injection, let $h\in\MM_{\Q}$ be a
  sparse injection and let $A\subseteq\I$ be finite or countably
  infinite\footnote{When applying this lemma, we will put either
    $A=\Dc(e_{s}s)$ or $A=\Dc(e_{s}\tilde{s})$.}.  Then there exists
  $b\in\GG_{\Q}$ such that $\iota:=hbg$ satisfies $\Dc^{\I}(\iota)=A$
  as well as $\LP(\iota)\subseteq\I$, namely any $b\in\GG_{\Q}$ with
  $\bar{b}^{-1}(\Dc^{\I}(h))\cap\Img(\bar{g})=\bar{g}(A)$. The
  boundedness type of $\iota$ coincides with the boundedness type of
  $h$.

  Moreover, suppose that $\bar{z}$ and $\bar{w}$ are tuples in $\Q$,
  that $\bar{z}'$ and $\bar{w}'$ are tuples in
  $(\R\setminus\Img(\bar{g}))\cap\I$ and $\Dc^{\I}(h)$, respectively,
  and that $\bar{z}''$ and $\bar{w}''$ are tuples in $\bar{g}(A)$ and
  $\Dc^{\I}(h)$, respectively. If the partial map sending
  $\bar{z}\mapsto\bar{w}$, $\bar{z}'\mapsto\bar{w}'$ and
  $\bar{z}''\mapsto\bar{w}''$ is strictly increasing, then
  $b\in\GG_{\Q}$ can be picked so that $\bar{b}$ extends this map.
\end{lemma}
Combining the Preconditioning Lemma~\ref{lem:preconditioning} (putting
$A=\Dc^{\I}(e_{s}s)$, see the proof of
Proposition~\ref{prop:rich-top-ppxb} below) with the Sandwich
Lemma~\ref{lem:sandwich}, we can show that $\sigma:=e_{s}s$ can be
written in the form $\pi a_{s}\iota=f_{s}a_{s}h_{s}b_{s}g_{s}$ with
$a_{s},b_{s}\in\GG_{\Q}$ if $\pi=f_{s}$ is a generic surjection,
$g_{s}$ and $e_{s}$ are generic injections and $h_{s}$ is a sparse
injection with the same boundedness type as $s$ -- note in particular
that the choice of the maps $e_{s},f_{s},g_{s},h_{s}$ only depends on
the boundedness type of $s$. For the remaining part of \PPXb, we have
to prove the following: If $a_{s}(\bar{x})=\bar{y}$ as well as
$b_{s}(\bar{z})=\bar{w}$, there is a $\TT_{rich}$-neighbourhood $U$ of
$s$ such that for all $\tilde{s}\in U$ one can write
$e_{s}\tilde{s}=f_{s}\tilde{a}h_{s}\tilde{b}g_{s}$, where
$\tilde{a},\tilde{b}\in\GG_{\Q}$ with $\tilde{a}(\bar{x})=\bar{y}$ as
well as $\tilde{b}(\bar{z})=\bar{w}$. By the Preconditioning
Lemma~\ref{lem:preconditioning}, we could find $\tilde{b}\in\GG_{\Q}$
with $\tilde{b}(\bar{z})=\bar{w}$ such that
$\tilde{\sigma}:=e_{s}\tilde{s},f_{s},\tilde{\iota}:=h_{s}\tilde{b}g_{s}$
are compatible. Thus, the Sandwich Lemma~\ref{lem:sandwich} would
yield $\tilde{a}\in\GG_{\Q}$ with
$e_{s}\tilde{s}=f_{s}\tilde{a}h_{s}\tilde{b}g_{s}$ -- however, this
automorphism $\tilde{a}$ need not satisfy the condition
$\tilde{a}(\bar{x})=\bar{y}$. To improve upon this strategy, the final
statement of the Sandwich Lemma~\ref{lem:sandwich} suggests we
construct $\tilde{b}$ in such a way that the finite partial map
defined by $\bar{x}\mapsto\bar{y}$ preserves all basic formulas when
considered as a map from $\A(\tilde{\sigma},f_{s},\tilde{\iota})$ to
$\B(\tilde{\sigma},f_{s},\tilde{\iota})$.
\begin{lemma}[Variation Lemma]\label{lem:variation}
  Let $\sigma,f,g,h\in\MM_{\Q}$ and $a,b\in\GG_{\Q}$ such that
  $\sigma=fahbg$, where $\LP(\sigma)\subseteq\I$, $f$ is a generic
  surjection, $g$ is a generic injection, $h$ is a sparse injection
  with the same boundedness type as $\sigma$, and finally
  $\bar{b}^{-1}(\Dc^{\I}(h))\cap\Img(\bar{g})=\bar{g}(\Dc^{\I}(\sigma))$. Let
  further $\bar{x},\bar{y},\bar{z},\bar{w}$ be tuples in $\Q$ such
  that $a(\bar{x})=\bar{y}$ and $b(\bar{z})=\bar{w}$.

  \noindent Then there exists a $\TT_{rich}$-neighbourhood $O$ of
  $\sigma$ such that the following holds:
  \begin{quote}
    For any $\tilde{\sigma}\in O$ with
    $\LP(\tilde{\sigma})\subseteq\I$, there exist tuples $\bar{z}^{*}$
    and $\bar{w}^{*}$ in $\Q$ and tuples $\bar{z}'$ and $\bar{w}'$ in
    $(\R\setminus\Img(\bar{g}))\cap\I$ and $\Dc^{\I}(h)$,
    respectively, and tuples $\bar{z}''$ and $\bar{w}''$ in
    $\bar{g}(\Dc^{\I}(\tilde{\sigma}))$ and $\Dc^{\I}(h)$,
    respectively, such that
    \begin{itemize}
    \item the finite partial map $\bar{z}\mapsto\bar{w}$,
      $\bar{z}^{*}\mapsto\bar{w}^{*}$, $\bar{z}'\mapsto\bar{w}'$,
      $\bar{z}''\mapsto\bar{w}''$ is strictly increasing,
    \item if $\tilde{b}\in\GG_{\Q}$ satisfies
      $\tilde{b}(\bar{z})=\bar{w}$,
      $\tilde{b}(\bar{z}^{*})=\bar{w}^{*}$,
      $\bar{\tilde{b}}(\bar{z}')=\bar{w}'$,
      $\bar{\tilde{b}}(\bar{z}'')=\bar{w}''$ and is such that
      $\tilde{\sigma},f,\tilde{\iota}:=h\tilde{b}g$ are compatible,
      then $\bar{x}\mapsto\bar{y}$ preserves all basic formulas when
      considered as a finite partial map from
      $\A(\tilde{\sigma},f,\tilde{\iota})$ to
      $\B(\tilde{\sigma},f,\tilde{\iota})$.
    \end{itemize}
  \end{quote}

\end{lemma}
Combining these results, we can prove that $\MM_{\Q}$ equipped with
the rich topology has \PPXb{} of length~$2$ with respect to
$(\GG_{\Q},\TT_{pw})$:
\begin{proof}[Proof (of Proposition~\ref{prop:rich-top-ppxb} given
  Lemmas~\ref{lem:sandwich},~\ref{lem:preconditioning}
  and~\ref{lem:variation}).]
  Let $s\in\MM_{\Q}$. We follow the strategy outlined in
  Figure~\ref{fig:proof-rich-top-ppxb}. First, we construct a
  decomposition $e_{s}s=f_{s}a_{s}h_{s}b_{s}g_{s}$.

  We use Lemma~\ref{lem:existence-generic} to find a generic injection
  $e_{s}\in\MM_{\Q}$, a generic surjection $f_{s}\in\MM_{\Q}$, a
  generic injection $g_{s}\in\MM_{\Q}$ and a sparse injection
  $h_{s}\in\MM_{\Q}$ with the same boundedness type as $s$. By
  Lemma~\ref{lem:easy-gen-inv}\ref{item:easy-gen-inv-iii}, the map
  $e_{s}$ is left-invertible. Since $e_{s}$ is unbounded-unbounded,
  $\sigma:=e_{s}s$ has the same boundedness type as $s$ (and as
  $h_{s}$) and satisfies
  $\LP(\sigma)\subseteq\LP(e_{s})\subseteq\I$. Applying the
  Preconditioning Lemma~\ref{lem:preconditioning} with
  $A=\Dc^{\I}(\sigma)$, we obtain $b_{s}\in\GG_{\Q}$ such that
  $\iota_{s}:=h_{s}b_{s}g_{s}$ is compatible with $\sigma$ and
  $f_{s}$, namely $b_{s}\in\GG_{\Q}$ with
  $\bar{b}_{s}^{-1}(\Dc^{\I}(h))\cap\Img(\bar{g}_{s})=\bar{g}_{s}(\Dc^{\I}(\sigma))$. Using
  the Sandwich Lemma~\ref{lem:sandwich}, we obtain $a_{s}\in\GG_{\Q}$
  such that
  \begin{displaymath}
    e_{s}s=\sigma=f_{s}a_{s}\iota_{s}=f_{s}a_{s}h_{s}b_{s}g_{s}.
  \end{displaymath}
  This proves conditions~\ref{item:ppxb-i} and~\ref{item:ppxb-ii} in
  the definition of~\PPXb.

  For condition~\ref{item:ppxb-iii}, let $V,W\subseteq\GG_{\Q}$ be
  open sets in the pointwise topology on $\GG_{\Q}$ with $a_{s}\in V$
  and $b_{s}\in W$. We need to find $U\in\TT_{rich}$ with $s\in U$
  such that $e_{s}U\subseteq f_{s}Vh_{s}Wg_{s}$. By shrinking the sets
  if necessary, we can assume that
  $V=\set{\tilde{a}\in\GG_{\Q}}{\tilde{a}(\bar{x})=\bar{y}}$ and
  $W=\set{\tilde{b}\in\GG_{\Q}}{\tilde{b}(\bar{z})=\bar{w}}$ for
  tuples $\bar{x},\bar{y},\bar{z},\bar{w}$ in $\Q$. We apply the
  Variation Lemma~\ref{lem:variation} for
  $\sigma=f_{s}a_{s}h_{s}b_{s}g_{s}$ to obtain a
  $\TT_{rich}$-neighbourhood $O$ of $\sigma$ with the following
  property: If $\tilde{s}\in\MM_{\Q}$ is such that
  $\tilde{\sigma}:=e_{s}\tilde{s}\in O$, there exist tuples
  $\bar{z}^{*}$ and $\bar{w}^{*}$ in $\Q$ and tuples $\bar{z}'$ and
  $\bar{w}'$ in $(\R\setminus\Img(\bar{g}))\cap\I$ and $\Dc^{\I}(h)$,
  respectively, and tuples $\bar{z}''$ and $\bar{w}''$ in
  $\bar{g}(\Dc^{\I}(\tilde{\sigma}))$ and $\Dc^{\I}(h)$, respectively,
  such that $\bar{z}\mapsto\bar{w}$, $\bar{z}^{*}\mapsto\bar{w}^{*}$,
  $\bar{z}'\mapsto\bar{w}'$, $\bar{z}''\mapsto\bar{w}''$ is strictly
  increasing and, additionally, if $\tilde{b}\in\GG_{\Q}$ satisfies
  $\tilde{b}(\bar{z})=\bar{w}$, $\tilde{b}(\bar{z}^{*})=\bar{w}^{*}$,
  $\bar{\tilde{b}}(\bar{z}')=\bar{w}'$,
  $\bar{\tilde{b}}(\bar{z}'')=\bar{w}''$ and is such that
  $\tilde{\sigma},f,\tilde{\iota}:=h\tilde{b}g$ are compatible, then
  $\bar{x}\mapsto\bar{y}$ preserves all basic formulas when considered
  as a finite partial map from
  $\A(\tilde{\sigma},f_{s},\tilde{\iota})$ to
  $\B(\tilde{\sigma},f_{s},\tilde{\iota})$. Given such
  $\tilde{s}\in\MM_{\Q}$, the Preconditioning
  Lemma~\ref{lem:preconditioning} with $A=\Dc^{\I}(\tilde{\sigma})$ as
  well as $\bar{z}\cup\bar{z}^{*}$ and $\bar{w}\cup\bar{w}^{*}$ in
  place of $\bar{z}$ and $\bar{w}$, respectively, yields
  $\tilde{b}\in\GG_{\Q}$ with the above properties. Hence,
  $\bar{x}\mapsto\bar{y}$ preserves all basic formulas when considered
  as a finite partial map from
  $\A(\tilde{\sigma},f_{s},\tilde{\iota})$ to
  $\B(\tilde{\sigma},f_{s},\tilde{\iota})$, and the Sandwich
  Lemma~\ref{lem:sandwich} gives $\tilde{a}\in\GG_{\Q}$ with
  $\tilde{a}(\bar{x})=\bar{y}$ and
  $e_{s}\tilde{s}=f_{s}\tilde{a}\tilde{\iota}=f_{s}\tilde{a}h_{s}\tilde{b}g_{s}\in
  f_{s}Vh_{s}Wg_{s}$.

  In other words, setting $U:=\lambda_{e_{s}}^{-1}(O)$ gives
  $e_{s}U\subseteq f_{s}Vg_{s}Wh_{s}$ as desired. Noting that $U$ is a
  $\TT_{rich}$-neighbourhood of $s$ by
  Lemma~\ref{lem:easy-gen-inv}\ref{item:easy-gen-inv-iv} finishes the
  proof.
\end{proof}
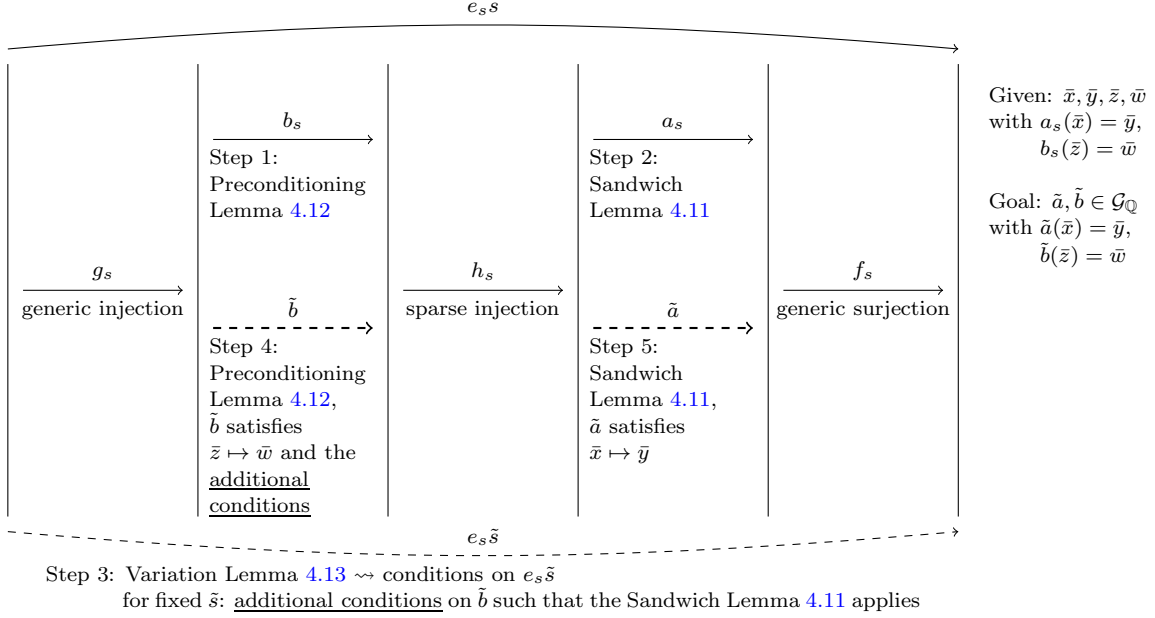
\begin{figure}[h]
  \begin{center}
    \begin{tikzpicture}
      \draw (0,-3) -- (0,3);
      \draw (2.5,-3) -- (2.5,3);
      \draw (5,-3) -- (5,3);
      \draw (7.5,-3) -- (7.5,3);
      \draw (10,-3) -- (10,3);
      \draw (12.5,-3) -- (12.5,3);
      \draw[->,font=\scriptsize] (0.2,0) -- node[above]{$g_{s}$} node[below]{generic injection} (2.3,0);
      \draw[->,font=\scriptsize] (2.7,2) -- node[above]{$b_{s}$}
      node[below]{\parbox{2.2cm}{Step 1:\\Preconditioning\\Lemma~\ref{lem:preconditioning}}} (4.8,2);
      \draw[->,dashed,font=\scriptsize,thick] (2.7,-0.5) --
            node[above]{$\tilde{b}$} node[below]{\parbox{2.2cm}{Step
                4:\\Preconditioning\\Lemma~\ref{lem:preconditioning},\\$\tilde{b}$
                satisfies\\$\bar{z}\mapsto\bar{w}$ and
                the\\\underline{additional}\\\underline{conditions}}}
            (4.8,-0.5);
      \draw[->,font=\scriptsize] (5.2,0) -- node[above]{$h_{s}$} node[below]{sparse injection} (7.3,0);
      \draw[->,font=\scriptsize] (7.7,2) -- node[above]{$a_{s}$}
      node[below]{\parbox{2.2cm}{Step 2:\\Sandwich\\Lemma~\ref{lem:sandwich}}} (9.8,2);
      \draw[->,dashed,font=\scriptsize,thick] (7.7,-0.5) --
            node[above]{$\tilde{a}$}
            node[below]{\parbox{2.2cm}{Step
                5:\\Sandwich\\Lemma~\ref{lem:sandwich},\\$\tilde{a}$
                satisfies\\$\bar{x}\mapsto\bar{y}$}} (9.8,-0.5);
      \draw[->,font=\scriptsize] (10.2,0) -- node[above]{$f_{s}$} node[below]{generic surjection}
      (12.3,0);
      \draw[->,font=\scriptsize] (0,3.2) to[bend left=5] node[above]{$e_{s}s$} (12.5,3.2);
      \draw[->,dashed,font=\scriptsize] (0,-3.2) to[bend right=5]
      node[above]{$e_{s}\tilde{s}$}node[below]{\parbox{11.5cm}{Step 3:
          Variation Lemma~\ref{lem:variation} $\rightsquigarrow $
          conditions on $e_{s}\tilde{s}$\\\textcolor{white}{Step 3: }for fixed $\tilde{s}$: \underline{additional conditions} on $\tilde{b}$ such that the Sandwich Lemma~\ref{lem:sandwich} applies}} (12.5,-3.2);
      \node at (14,1.5) {\scriptsize \parbox{2.2cm}{Given:
        $\bar{x},\bar{y},\bar{z},\bar{w}$\\with
        $a_{s}(\bar{x})=\bar{y}$,\\\textcolor{white}{with
        }$b_{s}(\bar{z})=\bar{w}$\\\\Goal: $\tilde{a},\tilde{b}\in\GG_{\Q}$\\with
        $\tilde{a}(\bar{x})=\bar{y}$,\\\textcolor{white}{with
        }$\tilde{b}(\bar{z})=\bar{w}$}};
    \end{tikzpicture}
  \end{center}
  \caption{Illustration of the proof of Proposition~\ref{prop:rich-top-ppxb}.}\label{fig:proof-rich-top-ppxb}
\end{figure}

\subsection{Proving the Sandwich Lemma~\ref{lem:sandwich}}
\label{sec:prov-sandw-lemma}
The proof of the Sandwich Lemma~\ref{lem:sandwich} requires two
additional auxiliary facts.

Since $\pi$ is a generic surjection, the preimages $\pi^{-1}\{z\}$
have neither a greatest nor a least element. This implies the
following simple yet crucial interpretation of the formulas
$P_{q}(z)$, $L_{q}(z)$ and $R_{q}(z)$ in $\B$:
\begin{lemma}\label{lem:aux-pred-interpretation-b}
  Let $\pi\in\MM_{\Q}$ be a generic surjection. Then the following
  holds for all $q,y\in\Q$:
  \begin{enumerate}[label=(\roman*)]
  \item $\B\models P_{q}(y)$ if and only if $\sigma(q)=\pi(y)$.
  \item $\B\models L_{q}(y)$ if and only if $\sigma(q)\leq\pi(y)$.
  \item $\B\models R_{q}(y)$ if and only if $\sigma(q)\geq\pi(y)$.
  \end{enumerate}
  In particular, $\B\models P_{q}(y)$ implies $\B\models L_{q}(y)$ as
  well as $\B\models R_{q}(y)$.
\end{lemma}
The following straightforward lemma intuitively means that our
definition of $P_{q}^{\A}$ is the ``correct'' one:
\begin{lemma}\label{lem:aux-str-prop-x-coll-pred}
  Let $\sigma,\pi,\iota\in\MM_{\Q}$. For all $q,q'\in\Q$, we have
  \begin{displaymath}
    P_{q}^{\A}\cap P_{q'}^{\A}\neq\emptyset\Leftrightarrow
    P_{q}^{\A}=P_{q'}^{\A}\Leftrightarrow \sigma(q)=\sigma(q')\Leftrightarrow
    P_{q}^{\B}=P_{q'}^{\B}\Leftrightarrow P_{q}^{\B}\cap P_{q'}^{\B}\neq\emptyset.
  \end{displaymath}
\end{lemma}
Now we can prove the Sandwich Lemma~\ref{lem:sandwich}:
\begin{proof}[Proof (of the Sandwich Lemma~\ref{lem:sandwich}).]
  Since~(ii) follows by combining~(i) with
  Lemma~\ref{lem:back-and-forth} to obtain a surjective
  $L$-homomorphism $a\colon\A\to\B$ extending~$m$, we only have to
  show~(i). We will verify that the set of all finite partial
  $L$-homomorphisms $m$ from $\A$ to $\B$ preserving all basic
  formulas has the Forth property and the Back property. Let $m$ be
  such a homomorphism.

  \textbf{Forth.} Given $x\in\A\setminus\Dom(m)$, we need to find
  $y\in\B\setminus\Img(m)$ such that the extension $m'$ of $m$ by
  $x\mapsto y$ is a finite partial $L$-homomorphism preserving all
  basic formulas. We will use the following general strategy: We first
  identify the desired position of $y$ with respect to the predicates
  $P_{q},L_{q},R_{q}$, and then employ the fact that $m$ preserves all
  basic formulas to find $y$ such that $m$ and $x\mapsto y$ are
  additionally order-compatible.

  Let $\bar{a}=(a_{1},\dots,a_{n})$ be an ascending enumeration of
  $\Dom(m)$ and let $\bar{b}:=m(\bar{a})$. Since $m$ is strictly
  increasing, $\bar{b}$ is an ascending enumeration of
  $\Img(m)$. Setting $a_{0}:=-\infty$ and $a_{n+1}:=+\infty$ as well
  as $b_{0}:=-\infty$ and $b_{n+1}:=+\infty$, there exists an index
  $i_{0}\in\{0,\dots,n\}$ such that $a_{i_{0}}<x<a_{i_{0}+1}$. We
  distinguish two cases:

  \textit{Case~1} ($\exists q_{0}\in\Q\colon x\in P_{q_{0}}^{\A}$):
  Since $\pi$ is a generic surjection
  (property~\ref{item:compatible-b} of compatibility), it suffices to
  find $y$ with
  \begin{displaymath}
    \sigma(q_{0})=\pi(y)\text{ and }b_{i_{0}}<y<b_{i_{0}+1};
  \end{displaymath}
  note that even though we do not know whether $x$ satisfies
  $L_{q_{0}}$ and $R_{q_{0}}$ in $\A$, the element $y$ certainly
  satisfies $L_{q_{0}}$ and $R_{q_{0}}$ in $\B$, see
  Lemma~\ref{lem:aux-pred-interpretation-b}. Applying that $m$
  preserves $R_{q_{0}}$ and $L_{q_{0}}$, one obtains
  $\pi(b_{i_{0}})\leq\sigma(q_{0})\leq\pi(b_{i_{0}+1})$ via
  Lemma~\ref{lem:aux-pred-interpretation-b} which yields the existence
  of $y$ with the desired properties (by
  property~\ref{item:compatible-b} of compatibility, the preimage
  $\pi^{-1}\{\sigma(q_{0})\}$ does not have a greatest or least
  element).

  \textit{Case~2} ($\nexists q\in\Q\colon x\in P_{q}^{\A}$): In this
  case, we have
  $J_{-}:=\set{q\in\Q}{\A\models L_{q}(x)}=\iota^{-1}(-\infty,x)$ as
  well as
  $J_{+}:=\set{q\in\Q}{\A\models R_{q}(x)}=\iota^{-1}(x,+\infty)$, and
  further $\Q=J_{-}\cupdot J_{+}$ where the common boundary point of
  $J_{-}$ and $J_{+}$ is $\iota^{\dagger}(x)$; note that $J_{\pm}$
  could be empty, in which case
  $\iota^{\dagger}(x)=\pm\infty$. Similarly to~Case~1, it suffices to
  find $y$ with (for $J_{-}=\emptyset$, we put
  $\sup\sigma(J_{-})=-\infty$; analogously for $J_{+}=\emptyset$)
  \begin{displaymath}
    \sup\sigma(J_{-})\leq\pi(y)\leq\inf\sigma(J_{+})\text{ and }b_{i_{0}}<y<b_{i_{0}+1}.
  \end{displaymath}
  This is accomplished by verifying
  \begin{align}
    \label{eq:proof-sandwich-i} \sup\sigma(J_{-})&<\inf\sigma(J_{+})\\
    \label{eq:proof-sandwich-ii} \pi(b_{i_{0}})&\leq\inf\sigma(J_{+})\\
    \label{eq:proof-sandwich-iii}
    \sup\sigma(J_{-})&\leq\pi(b_{i_{0}+1})\\
    \label{eq:proof-sandwich-iv}
    \exists u_{0}\in\Q\colon
    \max(\sup\sigma(J_{-}),\pi(b_{i_{0}}))&\leq u_{0}\leq\min(\inf\sigma(J_{+}),\pi(b_{i_{0}+1})).
  \end{align}
  If $u_{0}$ is as in~\eqref{eq:proof-sandwich-iv}, there exists
  $y\in\pi^{-1}\{u_{0}\}$ with $b_{i_{0}}<y<b_{i_{0}+1}$. Any such $y$
  has the desired properties.

  By our assumption for the current case, the element $x$ is in a
  ``gap'' of $\A$; the inequality~\eqref{eq:proof-sandwich-i}
  expresses that there exists a matching ``gap'' of $\B$. To verify,
  one distinguishes by $\iota^{\dagger}(x)$ and applies convenient
  parts of the properties~\ref{item:compatible-a}
  and~\ref{item:compatible-c} of compatibility: If
  $\iota^{\dagger}(x)=-\infty$, i.e.~$J_{-}=\emptyset$ and $J_{+}=\Q$,
  then $\iota$ is bounded below, so $\sigma$ is bounded below by
  property~\ref{item:compatible-c} of compatibility,
  yielding~\eqref{eq:proof-sandwich-i}. For
  $\iota^{\dagger}(x)=+\infty$, one argues analogously. If $J_{-}$ has
  a greatest element $q$, observe that $\sigma(J_{+})$ consists of
  elements strictly greater than $\sigma(q)$. Use
  $\LP(\sigma)\subseteq\I$ (property~\ref{item:compatible-a} of
  compatibility) combined with
  Lemma~\ref{lem:aux-str-prop-x-coll-pred} to see that
  $\inf\sigma(J_{+})$ is either contained in $\sigma(J_{+})$ or
  irrational. Conclude $\sigma(q)<\inf\sigma(J_{+})$ which
  yields~\eqref{eq:proof-sandwich-i}. If $J_{+}$ has a least element,
  one argues analogously. It remains to consider the case that
  $\iota^{\dagger}(x)\in\R$ and neither $J_{-}$ nor $J_{+}$ has a
  greatest or least element, respectively. Then
  $\iota^{\dagger}(x)\in\I$ and, since $x\notin\LP(\iota)$ by
  property~\ref{item:compatible-c} of compatibility, also
  $\iota^{\dagger}(x)\in\Dc(\iota)$. Another application of
  property~\ref{item:compatible-c} of compatibility yields
  $\iota^{\dagger}(x)\in\Dc(\sigma)$ and
  thus~\eqref{eq:proof-sandwich-i}.
      
  The inequalities~\eqref{eq:proof-sandwich-ii}
  and~\eqref{eq:proof-sandwich-iii} are clear since $m$ preserves all
  basic formulas, the inequality~\eqref{eq:proof-sandwich-iv} is
  immediate from the previous ones.

  \textbf{Back.} Given $y\in\B\setminus\Img(m)$, we need to find
  $x\in\A\setminus\Dom(m)$ such that the extension $m'$ of $m$ by
  $x\mapsto y$ is a finite partial $L$-homomorphism preserving all
  basic formulas. We proceed similarly to the Forth step.

  As before, let $\bar{a}=(a_{1},\dots,a_{n})$ be an ascending
  enumeration of $\Dom(m)$ and let $\bar{b}:=m(\bar{a})$ be the
  corresponding ascending enumeration of $\Img(m)$. We again set
  $a_{0}:=-\infty$ and $a_{n+1}:=+\infty$ as well as $b_{0}:=-\infty$
  and $b_{n+1}:=+\infty$, and define the index $i_{0}\in\{0,\dots,n\}$
  such that $b_{i_{0}}<y<b_{i_{0}+1}$. We further set
  $I_{-}:=\sigma^{-1}(-\infty,\pi(y))$, $I:=\sigma^{-1}\{\pi(y)\}$ and
  $I_{+}:=\sigma^{-1}(\pi(y),+\infty)$. If $x$ satisfies
  \begin{displaymath}
    \sup\iota(I_{-})<x<\inf\iota(I_{+})\text{ and }a_{i_{0}}<x<a_{i_{0}+1},
  \end{displaymath}
  then $m'$ extending $m$ by $x\mapsto y$ preserves all basic formulas
  since $\sup_{q\in I_{-}}P_{q}^{\A}=\sup\iota(I_{-})$ by
  Lemma~\ref{lem:aux-str-prop-x-coll-pred} (and analogously for
  $I_{+}$). If $I\neq\emptyset$, note that even though we cannot
  predict whether $x$ will be contained in $P_{q}^{\A}$ or will be
  below or above $P_{q}^{\A}$ for one (and thus for all) $q\in I$, the
  element $y$ satisfies $P_{q}$ as well as $L_{q},R_{q}$ in $\B$.

  One finds the desired element $x$ by verifying
  \begin{align}
    \label{eq:proof-sandwich-vi} \sup\iota(I_{-})&<\inf\iota(I_{+})\\
    \label{eq:proof-sandwich-vii} a_{i_{0}}&<\inf\iota(I_{+})\\
    \label{eq:proof-sandwich-viii} \sup\iota(I_{-})&<a_{i_{0}+1}
  \end{align}
  and picking any $x$ with
  $\max(\sup\iota(I_{-}),a_{i_{0}})<x<\min(\inf\iota(I_{+}),a_{i_{0}+1})$. Using
  the properties~\ref{item:compatible-a} and~\ref{item:compatible-c}
  of compatibility, the
  inequality~\eqref{eq:proof-sandwich-vi} follows just
  as~\eqref{eq:proof-sandwich-i} did in the Forth step. For the
  inequality~\eqref{eq:proof-sandwich-vii}, observe that
  $a_{i_{0}}<\iota(q)$ for all $q\in I_{+}$ (for otherwise, $m$ would
  not preserve $P_{q}$ and $R_{q}$) and that $\inf\iota(I_{+})$ is
  either contained in $\iota(I_{+})$ or irrational. The same argument
  yields the inequality~\eqref{eq:proof-sandwich-viii}.
\end{proof}

\subsection{Proving the Preconditioning
  Lemma~\ref{lem:preconditioning}}
\label{sec:prov-prec-lemma}
To find $\tilde{b}$, we use a Back\&Forth strategy.
\begin{proof}[Proof (of the Preconditioning
  Lemma~\ref{lem:preconditioning}).]
  Combining the facts that $g$ is continuous at all irrational points
  (by definition), that $\bar{g}(\I)\subseteq\I$ (by
  Lemma~\ref{lem:cont-lp-easy-facts}\ref{item:cont-lp-easy-facts-ii})
  and that any $\bar{b}$ for $b\in\GG_{\Q}$ is continuous at all
  irrational points (by
  Lemma~\ref{lem:cont-lp-easy-facts}\ref{item:cont-lp-easy-facts-iii}),
  we obtain that $bg$ will always be continuous at all irrational
  points as well. Thus,
  \begin{displaymath}
    \Dc^{\I}(hbg)=\bar{g}^{-1}\Big(\bar{b}^{-1}(\Dc(h))\cap\bar{g}(\I)\Big)\cap\I.
  \end{displaymath}
  If we use that $\bar{g}(\I),\bar{b}(\I)\subseteq\I$, we conclude
  \begin{displaymath}
    \Dc^{\I}(hbg)=\bar{g}^{-1}\left(\bar{b}^{-1}(\Dc^{\I}(h))\cap\Img(\bar{g})\right)\cap\I.
  \end{displaymath}
  Hence, it is sufficient to construct $b$ in such a way that
  \begin{equation}\label{eq:proof-preconditioning-i}
    \bar{b}^{-1}(\Dc^{\I}(h))\cap\Img(\bar{g})=\bar{g}(A);
  \end{equation}
  if we set $\iota:=hbg$, then $\LP(\iota)\subseteq\LP(h)\subseteq\I$
  and $\iota$ has the same boundedness type as $h$ since both $g$ and
  $b$ are unbounded-unbounded.
  
  To fulfil~\eqref{eq:proof-preconditioning-i}, note first that there
  exists a countable set $D\subseteq(\R\setminus\Img(\bar{g}))\cap\I$
  which is topologically dense in $\R$: by
  Lemma~\ref{lem:cont-lp-easy-facts}\ref{item:cont-lp-easy-facts-ii},
  the set $(\R\setminus\Img(\bar{g}))\cap\I$ is topologically dense in
  $\R$, so it suffices to pick $D$ to be topologically dense in
  $(\R\setminus\Img(\bar{g}))\cap\I$ (which is possible since the
  latter is a subset of a separable metric space and therefore
  separable itself). In doing so, we can make sure that $D$ contains
  all entries of $\bar{z}'$.

  Instead of directly constructing a map $b\colon\Q\to\Q$ which
  satisfies~\eqref{eq:proof-preconditioning-i}, we will find an order
  isomorphism
  $\beta\colon\Q\cupdot (\bar{g}(A)\cup D)\to \Q\cupdot \Dc^{\I}(h)$
  satisfying
  \begin{align*}
    \beta(\Q)=\Q\quad &\text{and}\quad \beta(\bar{g}(A)\cup
                        D)=\Dc^{\I}(h)\qquad\text{as well as}\\
    \beta(\bar{z})=\bar{w}\quad &\text{and}\quad
                                  \beta(\bar{z}')=\bar{w}', \beta(\bar{z}'')=\bar{w}''
  \end{align*}
  Setting $b:=\beta\vert_{\Q}$ then yields the map as
  in~\eqref{eq:proof-preconditioning-i} since
  $\bar{b}\vert_{\Q\cupdot (\bar{g}(A)\cup D)}=\beta$ by uniqueness of
  the increasing extension (see
  Lemma~\ref{lem:cont-lp-easy-facts}\ref{item:cont-lp-easy-facts-iii})
  and therefore
  \begin{displaymath}
    \bar{b}^{-1}(\Dc^{\I}(h))\cap\Img(\bar{g})=(\bar{g}(A)\cup
    D)\cap\Img(\bar{g})=\bar{g}(A).
  \end{displaymath}
  To obtain $\beta$, we show that the system $\mathcal{S}$ of all
  finite partial order isomorphisms $m$ from
  $\Q\cupdot (\bar{g}(A)\cup D)$ to $\Q\cupdot\Dc^{\I}(h)$ such that
  \begin{displaymath}
    m\big(\Q\cap\Dom(m)\big)=\Q\cap\Img(m)\quad \text{and}\quad m\big((\bar{g}(A)\cup
    D)\cap\Dom(m)\big)=\Dc^{\I}(h)\cap\Img(m)
  \end{displaymath}
  is a Back\&Forth system -- by Lemma~\ref{lem:back-and-forth}, the
  finite partial order isomorphism defined by $\bar{z}\mapsto\bar{w}$,
  $\bar{z}'\mapsto\bar{w}'$ and $\bar{z}''\mapsto\bar{w}''$ (which is
  a member of $\mathcal{S}$) can then be extended to a map $\beta$
  with the desired properties. For the Back step, suppose
  $m\in\mathcal{S}$ and $\gamma\notin\Img(m)$. Let $\gamma''$ be the
  greatest element of $\Img(m)\cap (-\infty,\gamma)$ (or~$-\infty$ if
  no such element exists) and, dually, let $\gamma'$ be the least
  element of $\Img(m)\cap (\gamma,+\infty)$ (or~$+\infty$ if no such
  element exists). If $\gamma\in\Q$,
  pick\footnote{$(m^{-1}(\gamma''),m^{-1}(\gamma'))_{\R}$ refers to
    the interval of all \emph{real} numbers between $m^{-1}(\gamma'')$
    and $m^{-1}(\gamma')$.}
  $\delta\in (m^{-1}(\gamma''),m^{-1}(\gamma'))_{\R}\cap\Q$; if
  $\gamma\in\Dc^{\I}(h)$, pick
  $\delta\in (m^{-1}(\gamma''),m^{-1}(\gamma'))_{\R}\cap
  (\bar{g}(A)\cup D)$ -- by topological density of $\Q$ and
  $\bar{g}(A)\cup D\supseteq D$ in $\R$, this is always possible. Then
  the extension of $m$ by $\delta\mapsto\gamma$ is an element of
  $\mathcal{S}$ as well. For the Forth step, one argues analogously.
\end{proof}

\subsection{Proving the Variation Lemma~\ref{lem:variation}, special
  cases}
\label{sec:prov-var-lemma-special-cases}
In a series of lemmas, we first consider the cases that can occur in
the special situation that $\bar{x}$ and $\bar{y}$ consist of a single
element. In Subsection~\ref{sec:prov-var-lemma-full}, we will then
amalgamate these special cases to a full proof. We will always
consider the same setup:
\begin{notation}\label{not:setup-ppxb}
  We say that $(*)$ holds if we are in the following situation:
  \begin{enumerate}[label=(\alph*)]
  \item $\sigma,\tilde{\sigma}\in\MM_{\Q}$ satisfy
    $\LP(\sigma),\LP(\tilde{\sigma})\subseteq\I$ and have the same
    boundedness type,
  \item $f\in\MM_{\Q}$ is a generic surjection,
  \item $g\in\MM_{\Q}$ is a generic injection,
  \item $h\in\MM_{\Q}$ is a sparse injection with the same boundedness
    type as $\sigma$ and $\tilde{\sigma}$,
  \item $a\in\GG_{\Q}$,
  \item $b\in\GG_{\Q}$ satisfies
    $\bar{b}^{-1}(\Dc^{\I}(h))\cap\Img(\bar{g})=\bar{g}(\Dc^{\I}(\sigma))$,
  \item $\sigma=fahbg$, so $a$ preserves all basic formulas as a map
    from $\A(\sigma,f,\iota)$ to $\B(\sigma,f,\iota)$, where
    $\iota:=hbg$.
  \end{enumerate}
\end{notation}
To simplify the arguments, we reformulate the property of preserving
all basic formulas central to the Sandwich Lemma~\ref{lem:sandwich}.
\begin{lemma}\label{lem:sandwich-reform}
  Let $\sigma,\pi,\iota\in\MM_{\Q}$ such that $\pi$ is a generic
  surjection. Then the map $x\mapsto y$ preserves all basic formulas
  when considered as a finite partial map from $\A$ to $\B$ if and
  only if the following two conditions hold:
  \begin{equation}\label{eq:sandwich-reform}
    \sigma\left(\iota^{-1}(-\infty,x]\right)\subseteq
    (-\infty,\pi(y)]\quad\text{and}\quad \sigma\left(\iota^{-1}[x,+\infty)\right)\subseteq[\pi(y),+\infty).
  \end{equation}
\end{lemma}
\begin{proof}
  Assume first that~\eqref{eq:sandwich-reform} holds. Since the finite
  partial map in question has a one-element domain and image, we do
  not have to consider formulas of the form $z_{i}<z_{j}$. If
  $\A\models P_{q}(x)$ for some $q\in\Q$, then there exist
  $q',q''\in\Q$ such that $\iota(q'')\leq x\leq\iota(q')$ and
  $\sigma(q'')=\sigma(q)=\sigma(q')$. By~\eqref{eq:sandwich-reform},
  we obtain $\sigma(q)=\sigma(q'')\leq\pi(y)\leq\sigma(q')=\sigma(q)$,
  so $y\in P_{q}^{\B}$. If $\A\models L_{q}(x)$ for some $q\in\Q$,
  there exists $u\in P_{q}^{\A}$ such that $u<y$. By the previous
  argument we have $\sigma(q)=\pi(u)\leq\pi(y)$, so
  $\B\models L_{q}(y)$ (see
  Lemma~\ref{lem:aux-pred-interpretation-b}). Finally, if
  $\A\models R_{q}(x)$ for some $q\in\Q$, we argue analogously.

  Now assume that $x\mapsto y$ preserves all basic formulas as a
  finite partial map from $\A$ to $\B$. We only show
  $\sigma\left(\iota^{-1}(-\infty,x]\right)\subseteq
  (-\infty,\pi(y)]$. Let $q\in \iota^{-1}(-\infty,x]$,
  i.e.~$\iota(q)\leq x$. If $\iota(q)=x$, then $x\in P_{q}^{\A}$, so
  $y\in P_{q}^{\B}$ by assumption and thus $\sigma(q)=\pi(y)$. If
  $\iota(q)<x$, then $x\in L_{q}^{\A}$, so $y\in L_{q}^{\B}$ and thus
  $\sigma(q)\leq\pi(y)$ (see
  Lemma~\ref{lem:aux-pred-interpretation-b}).
\end{proof}
After these preparations, we can state and prove the series of
auxiliary lemmas.
\begin{lemma}[see
  Figure~\ref{fig:preserve-basic-pp-aux-1}]\label{lem:preserve-basic-pp-aux-1}
  Let $\sigma,\tilde{\sigma},f,g,h,a,b$ such that $(*)$ holds. Let
  further $x,y\in\Q$ such that $a(x)=y$.

  Suppose that $x\notin (\inf h,\sup h)$. Then one of
  the following two cases occurs:
  \begin{enumerate}[label=(\arabic*)]
  \item\label{item:preserve-basic-pp-aux-1-i} \begin{enumerate}[label=(\roman*)]
    \item $\Img(h)\subseteq (-\infty,x)$ and
      $\Img(\sigma)\subseteq (-\infty,f(y)]$.
    \item If
      \begin{displaymath}
        \Img(\tilde{\sigma})\subseteq (-\infty,f(y)],\quad\text{i.e.~}\Img(\tilde{\sigma})\cap (f(y),+\infty)=\emptyset,
      \end{displaymath}
      then for any $\tilde{b}\in\GG_{\Q}$, the map $x\mapsto y$
      preserves all basic formulas as a finite partial
      map from $\A(\tilde{\sigma},f,\tilde{\iota})$ to
      $\B(\tilde{\sigma},f,\tilde{\iota})$ where
      $\tilde{\iota}:=h\tilde{b}g$.
    \end{enumerate}
  \item\label{item:preserve-basic-pp-aux-1-ii} \begin{enumerate}[label=(\roman*)]
    \item $\Img(h)\subseteq (x,+\infty)$ and
      $\Img(\sigma)\subseteq [f(y),+\infty)$.
    \item If
      \begin{displaymath}
        \Img(\tilde{\sigma})\subseteq [f(y),+\infty),
        \quad\text{i.e.~}\Img(\tilde{\sigma})\cap (-\infty,f(y))=\emptyset,
      \end{displaymath}
      then for any $\tilde{b}\in\GG_{\Q}$, the map $x\mapsto y$
      preserves all basic formulas as a finite partial
      map from $\A(\tilde{\sigma},f,\tilde{\iota})$ to
      $\B(\tilde{\sigma},f,\tilde{\iota})$ where
      $\tilde{\iota}:=h\tilde{b}g$.
    \end{enumerate}
  \end{enumerate}
\end{lemma}
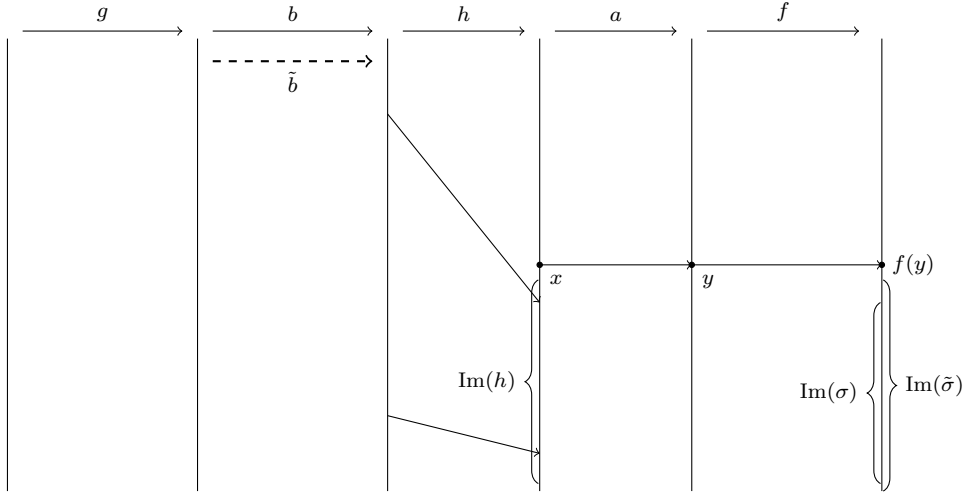
\begin{figure}[h]
  \begin{center}
    \begin{tikzpicture}
      \draw (0,-3) -- (0,3);
      \draw (2.5,-3) -- (2.5,3);
      \draw (5,-3) -- (5,3);
      \draw (7,-3) -- (7,3);
      \draw (9,-3) -- (9,3);
      \draw (11.5,-3) -- (11.5,3);
      \draw[->,font=\scriptsize] (0.2,3.1) -- node[above]{$g$} (2.3,3.1);
      \draw[->,font=\scriptsize] (2.7,3.1) -- node[above]{$b$} (4.8,3.1);
      \draw[->,dashed,font=\scriptsize,thick] (2.7,2.7) --
            node[below]{$\tilde{b}$} (4.8,2.7);
      \draw[->,font=\scriptsize] (5.2,3.1) -- node[above]{$h$} (6.8,3.1);
      \draw[->,font=\scriptsize] (7.2,3.1) -- node[above]{$a$} (8.8,3.1);
      \draw[->,font=\scriptsize] (9.2,3.1) -- node[above]{$f$}
      (11.2,3.1);
      \coordinate (x) at (7,0);
      \coordinate (y) at (9,0);
      \coordinate (fy) at (11.5,0);
      \filldraw (x) circle (1pt) node[below right]{\scriptsize $x$};
      \filldraw (y) circle (1pt) node[below right]{\scriptsize $y$};
      \filldraw (fy) circle (1pt) node[right]{\scriptsize $f(y)$};
      \draw[->] (x) -- (y);
      \draw[->] (y) -- (fy);
      
      \draw[->] (5,2) -- (7,-0.5);
      \draw[->] (5,-2) -- (7,-2.5);
      \draw[decorate,decoration={brace,amplitude=5,raise=0.5}] (7,-2.9) --
      (7,-0.2) node[pos=0.5,left=5,black]{\scriptsize
        $\Img(h)$};
      \draw[decorate,decoration={brace,amplitude=5,raise=0.5}] (11.5,-2.9) --
      (11.5,-0.5) node[pos=0.5,left=5,black]{\scriptsize
        $\Img(\sigma)$};
      \draw[decorate,decoration={brace,mirror,amplitude=5,raise=0.5}] (11.5,-3) --
      (11.5,-0.2) node[pos=0.5,right=5,black]{\scriptsize $\Img(\tilde{\sigma})$};
    \end{tikzpicture}
  \end{center}
  \caption{Illustration of Lemma~\ref{lem:preserve-basic-pp-aux-1},
    Case~\ref{item:preserve-basic-pp-aux-1-i}.}\label{fig:preserve-basic-pp-aux-1}
\end{figure}

\begin{proof}
  Our assumption $x\notin (\inf h,\sup h)$ implies that either
  $h(r)\leq x$ for all $r\in\Q$ or $h(r)\geq x$ for all
  $r\in\Q$. Since $h$ is injective, $\Img(h)$ cannot have a greatest
  or least element. Thus, either $\Img(h)\subseteq (-\infty,x)$ or
  $\Img(h)\subseteq (x,+\infty)$. We only treat the former case which
  corresponds to~\ref{item:preserve-basic-pp-aux-1-i}.

  \textbf{(i).} We have to show that
  $\Img(\sigma)\subseteq (-\infty,f(y)]$ -- this follows directly from
  Lemma~\ref{lem:sandwich-reform} by noting
  $\iota^{-1}(-\infty,x]=g^{-1}(b^{-1}(h^{-1}(-\infty,x]))=\Q$.

  \textbf{(ii).} With the same argument as in~(i), observe
  $\tilde{\iota}^{-1}(-\infty,x]=\Q$ and
  $\tilde{\iota}^{-1}[x,+\infty)=\emptyset$. Thus, the statement
  follows by another application of Lemma~\ref{lem:sandwich-reform}.
\end{proof}
\begin{lemma}[see
  Figure~\ref{fig:preserve-basic-pp-aux-2}]\label{lem:preserve-basic-pp-aux-2}
  Let $\sigma,\tilde{\sigma},f,g,h,a,b$ such that $(*)$ holds. Let
  further $x,y\in\Q$ such that $a(x)=y$.

  Suppose that $x\in (\inf h,\sup h)$
  with\footnote{Note that this encompasses the case $x\in\Img(h)$ and
    in particular $x\in\Img(\iota)$.}  $r:=h^{\dagger}(x)\in\Q$, and
  set $p:=b^{-1}(r)$ as well as $I_{-}:=\iota^{-1}(-\infty,x]$ and
  $I_{+}:=\iota^{-1}[x,+\infty)$. Then the following holds:
  \begin{enumerate}[label=(\roman*)]
  \item $I_{-}$ and $I_{+}$ are rational intervals with
    $\sigma(I_{-})\subseteq (-\infty,f(y)]$ and
    $\sigma(I_{+})\subseteq [f(y),+\infty)$.
  \item If
    \begin{displaymath}
      \tilde{\sigma}\left(I_{-}\right)\subseteq
      (-\infty,f(y)] \quad\text{and}\quad       \tilde{\sigma}\left(I_{+}\right)\subseteq
      [f(y),+\infty),
    \end{displaymath}
    then for any $\tilde{b}\in\GG_{\Q}$ with
    $\tilde{b}(p)=b(p)\,(=r)$, the map $x\mapsto y$ preserves all
    basic formulas as a finite partial map from
    $\A(\tilde{\sigma},f,\tilde{\iota})$ to
    $\B(\tilde{\sigma},f,\tilde{\iota})$ where
    $\tilde{\iota}:=h\tilde{b}g$.
  \end{enumerate}
\end{lemma}
\begin{figure}[h]
  \begin{center}
    \begin{tikzpicture}
      \draw (0,-3) -- (0,3);
      \draw (2.5,-3) -- (2.5,3);
      \draw (5,-3) -- (5,3);
      \draw (7,-3) -- (7,3);
      \draw (9,-3) -- (9,3);
      \draw (11.5,-3) -- (11.5,3);
      \draw[->,font=\scriptsize] (0.2,3.1) -- node[above]{$g$} (2.3,3.1);
      \draw[->,font=\scriptsize] (2.7,3.1) -- node[above]{$b$} (4.8,3.1);
      \draw[->,dashed,font=\scriptsize,thick] (2.7,2.7) --
            node[below]{$\tilde{b}$} (4.8,2.7);
      \draw[->,font=\scriptsize] (5.2,3.1) -- node[above]{$h$} (6.8,3.1);
      \draw[->,font=\scriptsize] (7.2,3.1) -- node[above]{$a$} (8.8,3.1);
      \draw[->,font=\scriptsize] (9.2,3.1) -- node[above]{$f$}
      (11.2,3.1);
      \coordinate (x) at (7,0);
      \coordinate (y) at (9,0);
      \coordinate (fy) at (11.5,0);
      \filldraw (x) circle (1pt) node[below right]{\scriptsize $x$};
      \filldraw (y) circle (1pt) node[below right]{\scriptsize $y$};
      \filldraw (fy) circle (1pt) node[right]{\scriptsize $f(y)$};
      \draw[->] (x) -- (y);
      \draw[->] (y) -- (fy); 

      \coordinate (qq) at (0,-0.5);
      \coordinate (p) at (2.5,-0.5);
      \coordinate (r) at (5,-0.5);
      \coordinate (pl) at (2.5,-1);
      \coordinate (pu) at (2.5,0);
      \coordinate (xl) at (7,-0.5);
      \coordinate (xu) at (7,0.5);
      \filldraw (p) circle (1pt) node[below right] {\scriptsize $p$};
      \filldraw (r) circle (1pt) node[below right] {\scriptsize
        \!\!$r=b(p)=\tilde{b}(p)$};
      \filldraw (xl) circle (1pt) node[below right] {\scriptsize $h(r)$};
      \filldraw (xu) circle (1pt) node[right] {\scriptsize $\inf h(r,+\infty)$};
      \draw[->,thick,dotted] (qq) -- (pl);
      \draw[->,thick,dotted] (qq) -- (pu);
      \draw[->] (p) -- (r);
      \draw[->,dashed,thick] (p) to[bend left=10] (r);
      \draw[->] (r) -- (xl);
      \draw[->,thick,dotted] (r) -- (xu);

      \draw[decorate,decoration={brace,amplitude=5,raise=0.5}] (0,-3) --
      (qq) node[pos=0.5,left=5,black]{\scriptsize
        \parbox{1.5cm}{\hfill $I_{-}=$ \\ $\iota^{-1}(-\infty,x]$}};
      \draw[decorate,decoration={brace,mirror,amplitude=5,raise=0.5}] (0,-3) --
      (qq) node[pos=0.5,right=5,black]{\scriptsize
        $\tilde{\iota}^{-1}(-\infty,x]$};
      \draw[decorate,decoration={brace,amplitude=5,raise=0.5}] (qq) --
      (0,3) node[pos=0.5,left=5,black]{\scriptsize
        \parbox{1.5cm}{\hfill $I_{+}=$ \\ $\iota^{-1}[x,+\infty)$}};
      \draw[decorate,decoration={brace,mirror,amplitude=5,raise=0.5}] (qq) --
      (0,3) node[pos=0.5,right=5,black]{\scriptsize
        $\tilde{\iota}^{-1}[x,+\infty)$};
      \draw[decorate,decoration={brace,amplitude=5,raise=0.5}] (11.5,-2.8) --
      (11.5,-0.5) node[pos=0.5,left=5,black]{\scriptsize $\sigma(I_{-})$};
      \draw[decorate,decoration={brace,amplitude=5,raise=0.5}] (11.5,0.2) --
      (11.5,3) node[pos=0.5,left=5,black]{\scriptsize $\sigma(I_{+})$};
      \draw[decorate,decoration={brace,amplitude=5,mirror,raise=0.5}] (11.5,-3) --
      (11.5,-0.2) node[pos=0.5,right=5,black]{\scriptsize $\tilde{\sigma}(I_{-})$};
      \draw[decorate,decoration={brace,amplitude=5,mirror,raise=0.5}] (11.5,0.5) --
      (11.5,2.8) node[pos=0.5,right=5,black]{\scriptsize $\tilde{\sigma}(I_{+})$};
    \end{tikzpicture}
  \end{center}
  \caption{Illustration of Lemma~\ref{lem:preserve-basic-pp-aux-2}.}\label{fig:preserve-basic-pp-aux-2}
\end{figure}
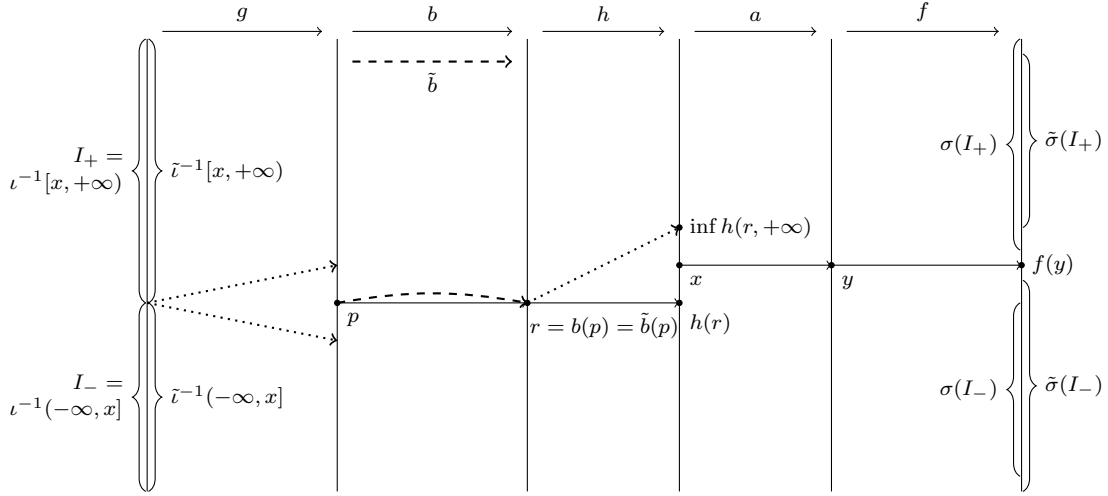
\begin{proof}
  $ $\par\nobreak\ignorespaces\textbf{(i).}  Since
  $r=h^{\dagger}(x)\in\Q$, we know that $h^{-1}(-\infty,x)$ and
  $h^{-1}(x,+\infty)$ are rational intervals, both with boundary point
  $r$. Hence, the intervals $h^{-1}(-\infty,x]$ and
  $h^{-1}[x,+\infty)$ also have boundary point $r$ (if $x\in\Img(h)$,
  the intervals become closed; otherwise, they do not change). By
  $b\in\GG_{\Q}$, the preimages $b^{-1}(h^{-1}(-\infty,x])$ and
  $b^{-1}(h^{-1}[x,+\infty))$ are rational intervals as well, both
  with boundary point $p=b^{-1}(r)$. Finally,
  $I_{-}=g^{-1}(b^{-1}(h^{-1}(-\infty,x]))$ and
  $I_{+}=g^{-1}(b^{-1}(h^{-1}[x,+\infty)))$ are rational intervals by
  applying Lemma~\ref{lem:easy-gen-inv}\ref{item:easy-gen-inv-iii}.

  The inclusions $\sigma(I_{-})\subseteq (-\infty,f(y)]$ and
  $\sigma(I_{+})\subseteq [f(y),+\infty)$ follow from
  Lemma~\ref{lem:sandwich-reform}.

  \textbf{(ii).} We claim that
  $\tilde{\iota}^{-1}(-\infty,x]=\iota^{-1}(-\infty,x]=I_{-}$ and
  $\tilde{\iota}^{-1}[x,+\infty)=\iota^{-1}[x,+\infty)=I_{+}$; the
  statement then follows by another application of
  Lemma~\ref{lem:sandwich-reform}. To this end, it suffices to note
  that $\tilde{b}^{-1}(h^{-1}(-\infty,x])$ coincides with
  $b^{-1}(h^{-1}(-\infty,x])$ since they have the same structure
  (open/closed) as $h^{-1}(-\infty,x]$ and the same boundary point,
  namely $\tilde{b}^{-1}(r)=p=b^{-1}(r)$. Analogously,
  $\tilde{b}^{-1}(h^{-1}[x,+\infty))$ coincides with
  $b^{-1}(h^{-1}[x,+\infty))$.
\end{proof}
\begin{lemma}[see Figures~\ref{fig:preserve-basic-pp-aux-3-case1}
  and~\ref{fig:preserve-basic-pp-aux-3-case2}]\label{lem:preserve-basic-pp-aux-3}
  Let $\sigma,\tilde{\sigma},f,g,h,a,b$ such that $(*)$ holds. Let
  further $x,y\in\Q$ such that $a(x)=y$.

  Suppose that $x\in (\inf h,\sup h)$ with
  $\gamma:=h^{\dagger}(x)\in\I$ and $q:=\iota^{\dagger}(x)\in\Q$. Then
  one of the following four cases occurs:
  \begin{enumerate}[label=(\arabic*)]
  \item\label{item:preserve-basic-pp-aux-3-case1}
    \begin{enumerate}[label=(\roman*)]
    \item $g(q)<\bar{b}^{-1}(\gamma)<\inf g(q,+\infty)$. Additionally,
      there
      exist $u,v\in\Q$ such that\\
      $g(q)<u<v<\inf g(q,+\infty)$ and $hb(u)<x<hb(v)$. Finally,
      $\sigma(q)\leq f(y)$ and
      $\sigma(q,+\infty)\subseteq [f(y),+\infty)$.
    \item If
      \begin{displaymath}
        \tilde{\sigma}(q)=\sigma(q) \quad\text{and}\quad \tilde{\sigma}(q,+\infty)\subseteq [f(y),+\infty),
      \end{displaymath}
      then for any $\tilde{b}\in\GG_{\Q}$ with $\tilde{b}(u)=b(u)$ and
      $\tilde{b}(v)=b(v)$, the map $x\mapsto y$ preserves all basic
      formulas as a finite partial map from
      $\A(\tilde{\sigma},f,\tilde{\iota})$ to
      $\B(\tilde{\sigma},f,\tilde{\iota})$ where
      $\tilde{\iota}:=h\tilde{b}g$.
    \end{enumerate}
  \item\label{item:preserve-basic-pp-aux-3-case2}
    \begin{enumerate}[label=(\roman*)]
    \item $g(q)<\bar{b}^{-1}(\gamma)=\inf g(q,+\infty)$. Additionally,
      there exists $u\in\Q$ such that\\
      $g(q)<u<\inf g(q,+\infty)$ and $hb(u)<x$. Finally,
      $\sigma(q)\leq f(y)$ and
      $\sigma(q,+\infty)\subseteq [f(y),+\infty)$ as well as
      $\bar{b}^{-1}(\gamma)\in(\R\setminus\Img(\bar{g}))\cap\I$ and
      $\gamma\in\Dc^{\I}(h)$.
    \item If
      \begin{displaymath}
        \tilde{\sigma}(q)=\sigma(q) \quad\text{and}\quad \tilde{\sigma}(q,+\infty)\subseteq [f(y),+\infty),
      \end{displaymath}
      then for any $\tilde{b}\in\GG_{\Q}$ with $\tilde{b}(u)=b(u)$ and
      $\bar{\tilde{b}}(\bar{b}^{-1}(\gamma))=\bar{b}(\bar{b}^{-1}(\gamma))=\gamma$
      (so $\bar{\tilde{b}}^{-1}(\gamma)=\bar{b}^{-1}(\gamma)$), the
      map $x\mapsto y$ preserves all basic formulas as a finite
      partial map from $\A(\tilde{\sigma},f,\tilde{\iota})$ to
      $\B(\tilde{\sigma},f,\tilde{\iota})$ where
      $\tilde{\iota}:=h\tilde{b}g$.
    \end{enumerate}
  \item\label{item:preserve-basic-pp-aux-3-case3}
    \begin{enumerate}[label=(\roman*)]
    \item $\sup g(-\infty,q)<\bar{b}^{-1}(\gamma)<g(q)$. Additionally,
      there
      exist $u,v\in\Q$ such that\\
      $\sup g(-\infty,q)<u<v<g(q)$ and $hb(u)<x<hb(v)$. Finally,
      $\sigma(-\infty,q)\subseteq (-\infty,f(y)]$ and
      $\sigma(q)\geq f(y)$.
    \item If
      \begin{displaymath}
        \sigma(-\infty,q)\subseteq (-\infty,f(y)]\quad\text{and}\quad\tilde{\sigma}(q)=\sigma(q),
      \end{displaymath}
      then for any $\tilde{b}\in\GG_{\Q}$ with $\tilde{b}(u)=b(u)$ and
      $\tilde{b}(v)=b(v)$, the map $x\mapsto y$ preserves all basic
      formulas as a finite partial map from
      $\A(\tilde{\sigma},f,\tilde{\iota})$ to
      $\B(\tilde{\sigma},f,\tilde{\iota})$ where
      $\tilde{\iota}:=h\tilde{b}g$.
    \end{enumerate}
  \item\label{item:preserve-basic-pp-aux-3-case4}
    \begin{enumerate}[label=(\roman*)]
    \item $\sup g(-\infty,q)=\bar{b}^{-1}(\gamma)<g(q)$. Additionally,
      there exists $v\in\Q$ such that\\
      $\sup g(-\infty,q)<v<g(q)$ and $x<hb(v)$. Finally,
      $\sigma(-\infty,q)\subseteq (-\infty,f(y)]$ and
      $\sigma(q)\geq f(y)$ as well as
      $\bar{b}^{-1}(\gamma)\in(\R\setminus\Img(\bar{g}))\cap\I$ and
      $\gamma\in\Dc^{\I}(h)$.
    \item If
      \begin{displaymath}
        \sigma(-\infty,q)\subseteq
        (-\infty,f(y)]\quad\text{and}\quad\tilde{\sigma}(q)=\sigma(q),
      \end{displaymath}
      then for any $\tilde{b}\in\GG_{\Q}$ with $\tilde{b}(v)=b(v)$ and
      $\bar{\tilde{b}}(\bar{b}^{-1}(\gamma))=\bar{b}(\bar{b}^{-1}(\gamma))=\gamma$
      (so $\bar{\tilde{b}}^{-1}(\gamma)=\bar{b}^{-1}(\gamma)$), the
      map $x\mapsto y$ preserves all basic formulas as a finite
      partial map from $\A(\tilde{\sigma},f,\tilde{\iota})$ to
      $\B(\tilde{\sigma},f,\tilde{\iota})$ where
      $\tilde{\iota}:=h\tilde{b}g$.
    \end{enumerate}
  \end{enumerate}
\end{lemma}
\begin{figure}[h]
  \begin{center}
    \begin{tikzpicture}
      \draw (0,-3) -- (0,3);
      \draw (2.5,-3) -- (2.5,3);
      \draw (5,-3) -- (5,3);
      \draw (7,-3) -- (7,3);
      \draw (9,-3) -- (9,3);
      \draw (11.5,-3) -- (11.5,3);
      \draw[->,font=\scriptsize] (0.2,3.1) -- node[above]{$g$} (2.3,3.1);
      \draw[->,font=\scriptsize] (2.7,3.1) -- node[above]{$b$} (4.8,3.1);
      \draw[->,dashed,font=\scriptsize,thick] (2.7,2.7) --
            node[below]{$\tilde{b}$} (4.8,2.7);
      \draw[->,font=\scriptsize] (5.2,3.1) -- node[above]{$h$} (6.8,3.1);
      \draw[->,font=\scriptsize] (7.2,3.1) -- node[above]{$a$} (8.8,3.1);
      \draw[->,font=\scriptsize] (9.2,3.1) -- node[above]{$f$}
      (11.2,3.1);
      \coordinate (x) at (7,0);
      \coordinate (y) at (9,0);
      \coordinate (fy) at (11.5,0);
      \filldraw (x) circle (1pt) node[below right]{\scriptsize $x$};
      \filldraw (y) circle (1pt) node[below right]{\scriptsize $y$};
      \filldraw (fy) circle (1pt) node[right]{\scriptsize $f(y)$};
      \draw[->] (x) -- (y);
      \draw[->] (y) -- (fy);

      \coordinate (q) at (0,-1);
      \coordinate (gql) at (2.5,-1.5);
      \coordinate (binvgamma) at (2.5,0);
      \coordinate (infg) at (2.5,1.5);
      \coordinate (u) at (2.5,-1);
      \coordinate (v) at (2.5,1);
      \coordinate (gamma) at (5,0);
      \coordinate (bu) at (5,-1);
      \coordinate (bv) at (5,1);
      \coordinate (xl) at (7,-0.5);
      \coordinate (xu) at (7,0.5);
      \coordinate (sigmaq) at (11.5,-1);
      \filldraw (q) circle (1pt) node[left]{\scriptsize $q$};
      \filldraw (gql) circle (1pt) node[right]{\scriptsize $g(q)$};
      \filldraw (binvgamma) circle (1pt) node[left]
      {\scriptsize $b^{-1}(\gamma)$};
      \filldraw (infg) circle (1pt) node[right] {\scriptsize
        $\inf g(q,+\infty)$};
      \filldraw (u) circle (1pt) node[below right]{\scriptsize $u$};
      \filldraw (v) circle (1pt) node[below right]{\scriptsize $v$};
      \filldraw (gamma) circle (1pt) node[below right]{\scriptsize
        $\gamma$};
      \filldraw (bu) circle (1pt) node[right]{\scriptsize
        $b(u)=\tilde{b}(u)$};
      \filldraw (bv) circle (1pt) node[right]{\scriptsize
        $b(v)=\tilde{b}(v)$};
     \filldraw (sigmaq) circle (1pt) node[below left]{\scriptsize
       $\sigma(q)$} node[below right]{\scriptsize $=\tilde{\sigma}(q)$};

      \draw[->] (q) -- (gql);
      \draw[->,dotted,thick] (q) -- (infg);
      \draw[->] (binvgamma) -- (gamma);
      \draw[->] (u) -- (bu);
      \draw[->,dashed,thick] (u) to[bend left=10] (bu);
      \draw[->] (v) -- (bv);
      \draw[->,dashed,thick] (v) to[bend left=10] (bv);
      \draw[->,dotted,thick] (gamma) -- (xl);
      \draw[->,dotted,thick] (gamma) -- (xu);

      \draw[decorate,decoration={brace,amplitude=5,raise=0.5}] (0,-3) --
      (q) node[pos=0.5,left=5,black]{\scriptsize
        $\iota^{-1}(-\infty,x]$};
      \draw[decorate,decoration={brace,mirror,amplitude=5,raise=0.5}] (0,-3) --
      (q) node[pos=0.5,right=5,black]{\scriptsize
        $\tilde{\iota}^{-1}(-\infty,x]$};
      \draw[decorate,decoration={brace,amplitude=5,raise=0.5}] (q) --
      (0,3) node[pos=0.5,left=5,black]{\scriptsize $\iota^{-1}[x,+\infty)$};
      \draw[decorate,decoration={brace,mirror,amplitude=5,raise=0.5}] (q) --
      (0,3) node[pos=0.5,right=5,black]{\scriptsize $\tilde{\iota}^{-1}[x,+\infty)$};

      \draw[decorate,decoration={brace,amplitude=5,raise=0.5}] (11.5,0.5) --
      (11.5,3) node[pos=0.5,left=5,black]{\scriptsize $\sigma(q,+\infty)$};      
      \draw[decorate,decoration={brace,amplitude=5,mirror,raise=0.5}] (11.5,0.2) --
      (11.5,3) node[pos=0.5,right=5,black]{\scriptsize $\tilde{\sigma}(q,+\infty)$};
    \end{tikzpicture}
  \end{center}
  \caption{Illustration of Lemma~\ref{lem:preserve-basic-pp-aux-3}, Case~\ref{item:preserve-basic-pp-aux-3-case1}.}\label{fig:preserve-basic-pp-aux-3-case1}
  \begin{center}
    \begin{tikzpicture}
      \draw (0,-3) -- (0,3);
      \draw (2.5,-3) -- (2.5,3);
      \draw (5,-3) -- (5,3);
      \draw (7,-3) -- (7,3);
      \draw (9,-3) -- (9,3);
      \draw (11.5,-3) -- (11.5,3);
      \draw[->,font=\scriptsize] (0.2,3.1) -- node[above]{$g$} (2.3,3.1);
      \draw[->,font=\scriptsize] (2.7,3.1) -- node[above]{$b$} (4.8,3.1);
      \draw[->,dashed,font=\scriptsize,thick] (2.7,2.7) --
            node[below]{$\tilde{b}$} (4.8,2.7);
      \draw[->,font=\scriptsize] (5.2,3.1) -- node[above]{$h$} (6.8,3.1);
      \draw[->,font=\scriptsize] (7.2,3.1) -- node[above]{$a$} (8.8,3.1);
      \draw[->,font=\scriptsize] (9.2,3.1) -- node[above]{$f$}
      (11.2,3.1);
      \coordinate (x) at (7,0);
      \coordinate (y) at (9,0);
      \coordinate (fy) at (11.5,0);
      \filldraw (x) circle (1pt) node[below right]{\scriptsize $x$};
      \filldraw (y) circle (1pt) node[below right]{\scriptsize $y$};
      \filldraw (fy) circle (1pt) node[right]{\scriptsize $f(y)$};
      \draw[->] (x) -- (y);
      \draw[->] (y) -- (fy);

      \coordinate (q) at (0,-1);
      \coordinate (gql) at (2.5,-1.5);
      \coordinate (binvgamma) at (2.5,0);
      \coordinate (u) at (2.5,-0.5);
      \coordinate (gamma) at (5,0);
      \coordinate (bu) at (5,-1);
      \coordinate (xl) at (7,-0.5);
      \coordinate (xu) at (7,0.5);
      \coordinate (sigmaq) at (11.5,-1);
      \filldraw (q) circle (1pt) node[left]{\scriptsize $q$};
      \filldraw (gql) circle (1pt) node[right]{\scriptsize $g(q)$};
      \filldraw (binvgamma) circle (1pt) node[above right]{\scriptsize\parbox{2.5cm}{$\inf g(q,+\infty)=$ \\ $b^{-1}(\gamma)=\tilde{b}^{-1}(\gamma)$}};
      \filldraw (u) circle (1pt) node[below right]{\scriptsize $u$};
      \filldraw (gamma) circle (1pt) node[below right]{\scriptsize
        $\gamma$};
      \filldraw (bu) circle (1pt) node[right]{\scriptsize
        $b(u)=\tilde{b}(u)$};
     \filldraw (sigmaq) circle (1pt) node[below left]{\scriptsize
       $\sigma(q)$} node[below right]{\scriptsize $=\tilde{\sigma}(q)$};

      \draw[->] (q) -- (gql);
      \draw[->,dotted,thick] (q) -- (binvgamma);
      \draw[->] (binvgamma) -- (gamma);
      \draw[->,dashed,thick] (binvgamma) to[bend right=10] (gamma);
      \draw[->] (u) -- (bu);
      \draw[->,dashed,thick] (u) to[bend left=10] (bu);
      \draw[->,dotted,thick] (gamma) -- (xl);
      \draw[->,dotted,thick] (gamma) -- (xu);

      \draw[decorate,decoration={brace,amplitude=5,raise=0.5}] (0,-3) --
      (q) node[pos=0.5,left=5,black]{\scriptsize
        $\iota^{-1}(-\infty,x]$};
      \draw[decorate,decoration={brace,mirror,amplitude=5,raise=0.5}] (0,-3) --
      (q) node[pos=0.5,right=5,black]{\scriptsize
        $\tilde{\iota}^{-1}(-\infty,x]$};
      \draw[decorate,decoration={brace,amplitude=5,raise=0.5}] (q) --
      (0,3) node[pos=0.5,left=5,black]{\scriptsize $\iota^{-1}[x,+\infty)$};
      \draw[decorate,decoration={brace,mirror,amplitude=5,raise=0.5}] (q) --
      (0,3) node[pos=0.5,right=5,black]{\scriptsize $\tilde{\iota}^{-1}[x,+\infty)$};

      \draw[decorate,decoration={brace,amplitude=5,raise=0.5}] (11.5,0.5) --
      (11.5,3) node[pos=0.5,left=5,black]{\scriptsize $\sigma(q,+\infty)$};      
      \draw[decorate,decoration={brace,amplitude=5,mirror,raise=0.5}] (11.5,0.2) --
      (11.5,3) node[pos=0.5,right=5,black]{\scriptsize $\tilde{\sigma}(q,+\infty)$};
    \end{tikzpicture}
  \end{center}
  \caption{Illustration of Lemma~\ref{lem:preserve-basic-pp-aux-3}, Case~\ref{item:preserve-basic-pp-aux-3-case2}.}\label{fig:preserve-basic-pp-aux-3-case2}
\end{figure}
\begin{proof}
  First of all, note that $\iota^{\dagger}(x)$ is welldefined since
  $x\in (\inf h,\sup h)$ which coincides with
  $(\inf\iota,\sup\iota)$ by the unboundedness on
  either side of both $g$ and $b$. Since $\gamma=h^{\dagger}(x)$ is
  irrational and $h$ is injective, $x$ cannot be contained in
  $\Img(h)$, in particular not in $\Img(\iota)$. In any case, we have
  $h^{-1}(-\infty,x]=h^{-1}(-\infty,x)=(-\infty,\gamma)$ and
  $h^{-1}[x,+\infty)=h^{-1}(x,+\infty)=(\gamma,+\infty)$. The
  cases~\ref{item:preserve-basic-pp-aux-3-case1}
  and~\ref{item:preserve-basic-pp-aux-3-case2} correspond to
  $\iota(q)<x$, in other words $\iota^{-1}(-\infty,x]=(-\infty,q]$ and
  $\iota^{-1}[x,+\infty)=(q,+\infty)$, while the
  cases~\ref{item:preserve-basic-pp-aux-3-case3}
  and~\ref{item:preserve-basic-pp-aux-3-case4} correspond to
  $\iota(q)>x$, in other words $\iota^{-1}(-\infty,x]=(-\infty,q)$ and
  $\iota^{-1}[x,+\infty)=[q,+\infty)$. We only treat the former
  cases. Since $\LP(h)\subseteq\I$, the point $x$ cannot be a
  limit point of $h$, so we have
  $\sup h(-\infty,\gamma)<x<\inf h(\gamma,+\infty)$ and
  $\gamma\in\Dc^{\I}(h)$. Additionally, $x<\inf\iota(q,+\infty)$ by
  the same argument. If $\iota(q)<x$, we obtain
  $bg(q)<\gamma\leq\inf bg(q,+\infty)$,
  i.e.~$g(q)<\bar{b}^{-1}(\gamma)\leq\inf g(q,+\infty)$. The first two
  cases are distinguished by checking whether the latter inequality is
  strict or not.
  \begin{enumerate}[label=(\arabic*)]
  \item $g(q)<\bar{b}^{-1}(\gamma)<\inf g(q,+\infty)$.

    \noindent\textbf{(i).} Take any $u,v\in\Q$ with
    $g(q)<u<\bar{b}^{-1}(\gamma)<v<\inf g(q,+\infty)$ to satisfy\\
    $g(q)<u<v<\inf g(q,+\infty)$ and $hb(u)<x<hb(v)$. The remaining
    statements follows from Lemma~\ref{lem:sandwich-reform}:
    $\sigma(q)\in\sigma(-\infty,q]=\sigma\left(\iota^{-1}(-\infty,x]\right)\subseteq
    (-\infty,f(y)]$ and
    $\sigma(q,+\infty)=\sigma\left(\iota^{-1}[x,+\infty)\right)\subseteq
    [f(y),+\infty)$.

    \noindent\textbf{(ii).} We use $\tilde{b}(u)=b(u)$ and $\tilde{b}(v)=b(v)$ to verify the
    conditions in Lemma~\ref{lem:sandwich-reform}. Note that
    $h^{-1}(-\infty,x]\subseteq(-\infty,b(v))$ and
    $h^{-1}[x,+\infty)\subseteq (b(u),+\infty)$, so that
    \begin{align*}
      \tilde{\iota}^{-1}(-\infty,x]&=g^{-1}(\tilde{b}^{-1}(h^{-1}(-\infty,x]))\subseteq
                                     g^{-1}(-\infty,v)=(-\infty,q]\qquad\text{and}\\
      \tilde{\iota}^{-1}[x,+\infty)&=g^{-1}(\tilde{b}^{-1}(h^{-1}[x,+\infty)))\subseteq
                                     g^{-1}(u,+\infty)=(q,+\infty)
    \end{align*}
    which yields
    \begin{align*}
      \tilde{\sigma}\left(\tilde{\iota}^{-1}(-\infty,x]\right)&\subseteq\tilde{\sigma}(-\infty,q]\subseteq
                                                                (-\infty,\tilde{\sigma}(q)]=(-\infty,\sigma(q)]\subseteq
                                                                (-\infty,f(y)]\quad\text{and}\\
      \tilde{\sigma}\left(\tilde{\iota}^{-1}[x,+\infty)\right)&\subseteq\tilde{\sigma}(q,+\infty)\subseteq [f(y),+\infty).
    \end{align*}
  \item $g(q)<\bar{b}^{-1}(\gamma)=\inf g(q,+\infty)$.

    \noindent\textbf{(i).} Take any $u\in\Q$ with $g(q)<u<\bar{b}^{-1}(\gamma)$ to
    satisfy $g(q)<u<\inf g(q,+\infty)$ and $hb(u)<x$. The statements
    $\sigma(q)\leq f(y)$ and
    $\sigma(q,+\infty)\subseteq [f(y),+\infty)$ follow just as
    in~\ref{item:preserve-basic-pp-aux-3-case1}. We have already
    argued that $\gamma\in\Dc^{\I}(h)$, so it remains to show
    $\bar{b}^{-1}(\gamma)\in(\R\setminus\Img(\bar{g}))\cap\I$. We know
    that $\gamma$ is irrational, so $\bar{b}^{-1}(\gamma)$ is as well
    by
    Lemma~\ref{lem:cont-lp-easy-facts}\ref{item:cont-lp-easy-facts-iii}. Additionally,
    $\bar{b}^{-1}(\gamma)=\inf g(q,+\infty)$ cannot be contained in
    $\Img(\bar{g})$ since $g$ is injective.

    \noindent\textbf{(ii).} Similarly to~\ref{item:preserve-basic-pp-aux-3-case1}, we
    use $\tilde{b}(u)=b(u)$ and
    $\bar{\tilde{b}}^{-1}(\gamma)=\bar{b}^{-1}(\gamma)$ to verify the
    conditions in Lemma~\ref{lem:sandwich-reform}. Observe
    $h^{-1}(-\infty,x]=(-\infty,\gamma)$ and
    $h^{-1}[x,+\infty)\subseteq (b(u),+\infty)$, so that
    \begin{align*}
      \tilde{\iota}^{-1}(-\infty,x]&=g^{-1}(-\infty,\bar{\tilde{b}}^{-1}(\gamma))=\iota^{-1}(-\infty,x]=(-\infty,q]\quad\text{and}\\
      \tilde{\iota}^{-1}[x,+\infty)&\subseteq
                                     g^{-1}(u,+\infty)=(q,+\infty),
    \end{align*}
    which yields
    $\tilde{\sigma}(\tilde{\iota}^{-1}(-\infty,x])\subseteq
    (-\infty,f(y)]$ and
    $\tilde{\sigma}(\tilde{\iota}^{-1}[x,+\infty))\subseteq
    [f(y),+\infty)$ as
    in~\ref{item:preserve-basic-pp-aux-3-case1}.
  \end{enumerate}
\end{proof}
In the remaining case that $x\in (\inf h,\sup h)$ and
both $h^{\dagger}(x)$ and $\iota^{\dagger}(x)$ are irrational, we take
a similar but somewhat more involved route in that the automorphisms
$\tilde{b}\in\GG_{\Q}$ we are picking do not simply mimic the
behaviour of $b$ on sufficiently many elements. Instead, we redefine
$\tilde{b}$ on certain crucial points which are tuned to the specific
$\tilde{\sigma}$ being considered. In doing so, we have to make sure
that our desired redefinition does not violate the condition for
$\tilde{b}$ on finitely many points given by \PPXb{} and the previous
auxiliary lemmas. We split our treatment of this problem into two
subcases.
\begin{lemma}[see
  Figure~\ref{fig:preserve-basic-pp-aux-4}]\label{lem:preserve-basic-pp-aux-4}
  Let $\sigma,\tilde{\sigma},f,g,h,a,b$ such that $(*)$ holds. Let
  further $x,y\in\Q$ such that $a(x)=y$.

  Suppose that $x\in (\inf h,\sup h)$ with
  $\gamma:=h^{\dagger}(x)\in\I$ and
  $\delta:=\iota^{\dagger}(x)\in\I$. Additionally, suppose that
  $f(y)\in\Img(\sigma)$. Let $\bar{z}$ and $\bar{w}$ be tuples in $\Q$
  and let $\bar{z}'$ and $\bar{w}'$ be tuples in
  $(\R\setminus\Img(\bar{g}))\cap\I$ and $\Dc^{\I}(h)$, respectively,
  such that $b(\bar{z})=\bar{w}$ and
  $\bar{b}(\bar{z}')=\bar{w}'$. Assume that $\bar{z}\cup\bar{z}'$
  contains both an element greater and less than
  $\bar{g}(\delta)$. Put $z_{-}$ and $z_{+}$ to be the greatest entry
  of $\bar{z}\cup\bar{z}'$ less than $\bar{g}(\delta)$ and the least
  entry of $\bar{z}\cup\bar{z}'$ greater than $\bar{g}(\delta)$,
  respectively, and put $w_{-}$ and $w_{+}$ to be the corresponding
  entries of $\bar{w}\cup\bar{w}'$. Then one of the following two
  cases occurs\footnote{It is possible that both cases occur
    simultaneously; it this happens, pick one of them arbitrarily.}:
  \begin{enumerate}[label=(\arabic*)]
  \item\label{item:preserve-basic-pp-aux-4-case1}
    \begin{enumerate}[label=(\roman*)]
    \item\label{item:preserve-basic-pp-aux-4-case1-i} There exist
      $q,q'\in\Q$ such that $q<q'<\delta$ and
      $z_{-}<g(q)<g(q')<\bar{g}(\delta)<z_{+}$ as well as
      $\sigma(q)=\sigma(q')=f(y)$; further,
      $\gamma=\bar{b}(\bar{g}(\delta))$ and $w_{-}<\gamma<w_{+}$.
    \item\label{item:preserve-basic-pp-aux-4-case1-ii} If
      \begin{displaymath}
        \tilde{\sigma}(q)=\sigma(q)=f(y)\quad\text{and}\quad \tilde{\sigma}(q')=\sigma(q')=f(y),
      \end{displaymath}
      and if $\tilde{u},\tilde{v},\hat{u},\hat{v}\in\Q$ satisfy
      $g(q)<\tilde{u}<\tilde{v}<g(q')$ as well as
      $w_{-}<\hat{u}<\gamma<\hat{v}<w_{+}$, then the finite partial
      map $\bar{z}\mapsto\bar{w}$, $\bar{z}'\mapsto\bar{w}'$ and
      $\tilde{u}\mapsto\hat{u},\tilde{v}\mapsto\hat{v}$ is strictly
      increasing. Additionally, for any $b\in\GG_{\Q}$ such that
      $\bar{b}(\bar{z})=\bar{w}$, $\bar{b}(\bar{z}')=\bar{w}'$ and
      $b(\tilde{u})=\hat{u}$, $b(\tilde{v})=\hat{v}$, the map
      $x\mapsto y$ preserves all basic formulas as a finite partial
      map from $\A(\tilde{\sigma},f,\tilde{\iota})$ to
      $\B(\tilde{\sigma},f,\tilde{\iota})$ where
      $\tilde{\iota}:=h\tilde{b}g$.
    \end{enumerate}
  \item\label{item:preserve-basic-pp-aux-4-case2}
    \begin{enumerate}[label=(\roman*)]
    \item\label{item:preserve-basic-pp-aux-4-case2-i} There exist
      $q,q'\in\Q$ such that $\delta<q<q'$ and
      $z_{-}<\bar{g}(\delta)<g(q)<g(q')<z_{+}$ as well as
      $\sigma(q)=\sigma(q')=f(y)$; further,
      $\gamma=\bar{b}(\bar{g}(\delta))$ and $w_{-}<\gamma<w_{+}$.
    \item\label{item:preserve-basic-pp-aux-4-case2-ii} If
      \begin{displaymath}
        \tilde{\sigma}(q)=\sigma(q)=f(y)\quad\text{and}\quad \tilde{\sigma}(q')=\sigma(q')=f(y),
      \end{displaymath}
      and if $\tilde{u},\tilde{v},\hat{u},\hat{v}\in\Q$ satisfy
      $g(q)<\tilde{u}<\tilde{v}<g(q')$ as well as
      $w_{-}<\hat{u}<\gamma<\hat{v}<w_{+}$, then the finite partial
      map $\bar{z}\mapsto\bar{w}$, $\bar{z}'\mapsto\bar{w}'$ and
      $\tilde{u}\mapsto\hat{u},\tilde{v}\mapsto\hat{v}$ is strictly
      increasing. Additionally, for any $b\in\GG_{\Q}$ such that
      $\bar{b}(\bar{z})=\bar{w}$, $\bar{b}(\bar{z}')=\bar{w}'$ and
      $b(\tilde{u})=\hat{u}$, $b(\tilde{v})=\hat{v}$, the map
      $x\mapsto y$ preserves all basic formulas as a finite partial
      map from $\A(\tilde{\sigma},f,\tilde{\iota})$ to
      $\B(\tilde{\sigma},f,\tilde{\iota})$ where
      $\tilde{\iota}:=h\tilde{b}g$.
    \end{enumerate}
  \end{enumerate}
\end{lemma}
\begin{figure}[h]
  \begin{center}
    \begin{tikzpicture}
      \draw (0,-3) -- (0,3);
      \draw (2.5,-3) -- (2.5,3);
      \draw (5,-3) -- (5,3);
      \draw (7,-3) -- (7,3);
      \draw (9,-3) -- (9,3);
      \draw (11.5,-3) -- (11.5,3);
      \draw[->,font=\scriptsize] (0.2,3.1) -- node[above]{$g$} (2.3,3.1);
      \draw[->,font=\scriptsize] (2.7,3.1) -- node[above]{$b$} (4.8,3.1);
      \draw[->,dashed,font=\scriptsize,thick] (2.7,2.7) --
            node[below]{$\tilde{b}$} (4.8,2.7);
      \draw[->,font=\scriptsize] (5.2,3.1) -- node[above]{$h$} (6.8,3.1);
      \draw[->,font=\scriptsize] (7.2,3.1) -- node[above]{$a$} (8.8,3.1);
      \draw[->,font=\scriptsize] (9.2,3.1) -- node[above]{$f$}
      (11.3,3.1);
      \coordinate (x) at (7,0);
      \coordinate (y) at (9,0);
      \coordinate (fy) at (11.5,0);
      \filldraw (x) circle (1pt) node[below right]{\scriptsize $x$};
      \filldraw (y) circle (1pt) node[below right]{\scriptsize $y$};
      \filldraw (fy) circle (1pt) node[above left]{\scriptsize
        $f(y)=\sigma(q)$} node[below left]{\scriptsize
        $\textcolor{white}{f(y)}=\sigma(q')$} node[above right]{\scriptsize
        $=\tilde{\sigma}(q)$} node[below right]{\scriptsize $=\tilde{\sigma}(q')$};
      \draw[->] (x) -- (y);
      \draw[->] (y) -- (fy);

      \coordinate (delta) at (0,0);
      \coordinate (q) at (0,-2);
      \coordinate (qprime) at (0,-1);
      \coordinate (zminus) at (2.5,-2.5);
      \coordinate (zplus) at (2.5,1.5);
      \coordinate (gq) at (2.5,-2);
      \coordinate (gqprime) at (2.5,-0.5);
      \coordinate (gdelta) at (2.5,0);
      \coordinate (utilde) at (2.5,-1.5);
      \coordinate (vtilde) at (2.5,-1);
      \coordinate (wminus) at (5,-2.5);
      \coordinate (wplus) at (5,1.5);
      \coordinate (gamma) at (5,0);
      \coordinate (uhat) at (5,-1.5);
      \coordinate (vhat) at (5,1);
      \coordinate (xl) at (7,-0.5);
      \coordinate (xu) at (7,0.5);
      \filldraw (delta) circle (1pt) node[left] {\scriptsize $\delta$};
      \filldraw (q) circle (1pt) node[left] {\scriptsize $q$};
      \filldraw (qprime) circle (1pt) node[left] {\scriptsize $q'$};
      \filldraw (zminus) circle (1pt) node[below right] {\scriptsize $z_{-}$};
      \filldraw (zplus) circle (1pt) node[below right] {\scriptsize $z_{+}$};
      \filldraw (gq) circle (1pt) node[right] {\scriptsize $g(q)$};
      \filldraw (gqprime) circle (1pt) node[right] {\scriptsize $g(q')$};
      \filldraw (gdelta) circle (1pt) node[above right] {\scriptsize $\bar{g}(\delta)$};
      \filldraw (utilde) circle (1pt) node[left] {\scriptsize $\tilde{u}$};
      \filldraw (vtilde) circle (1pt) node[left] {\scriptsize $\tilde{v}$};
      \filldraw (wminus) circle (1pt) node[right] {\scriptsize $w_{-}$};
      \filldraw (wplus) circle (1pt) node[right] {\scriptsize
        $w_{+}$};
      \filldraw (gamma) circle (1pt) node[below right] {\scriptsize $\gamma$};
      \filldraw (uhat) circle (1pt) node[right] {\scriptsize $\hat{u}$};
      \filldraw (vhat) circle (1pt) node[right] {\scriptsize $\hat{v}$};
      \draw[->] (delta) -- (gdelta);
      \draw[->] (q) -- (gq);
      \draw[->] (qprime) -- (gqprime);
      \draw[->] (zplus) -- (wplus);
      \draw[->] (zminus) -- (wminus);
      \draw[->] (gdelta) -- (gamma);
      \draw[->,dashed,thick] (utilde) -- (uhat);
      \draw[->,dashed,thick] (vtilde) to[bend right=20] (vhat);
      \draw[->,dotted,thick] (gamma) -- (xl);
      \draw[->,dotted,thick] (gamma) -- (xu);

      \draw[decorate,decoration={brace,amplitude=5,raise=0.5}] (0,-3) --
      (delta) node[pos=0.5,left=5,black]{\scriptsize $\iota^{-1}(-\infty,x]$};
      \draw[decorate,decoration={brace,amplitude=5,raise=0.5}] (delta) --
      (0,3) node[pos=0.5,left=5,black]{\scriptsize
        $\iota^{-1}[x,+\infty)$};
      \draw[decorate,decoration={brace,mirror,amplitude=10,raise=0.5}] (0,-3) --
      (0,-1.5) node[pos=0.5,right=10,black]{\scriptsize $\tilde{\iota}^{-1}(-\infty,x]$};
      \draw[decorate,decoration={brace,mirror,amplitude=10,raise=0.5}] (0,-1.5) --
      (0,3) node[pos=0.5,right=10,black]{\scriptsize $\tilde{\iota}^{-1}[x,+\infty)$};
      \draw[decorate,decoration={brace,amplitude=5,mirror,raise=0.5}] (11.5,0) --
      (11.5,2.9) node[pos=0.5,right=5,black]{\scriptsize
        $\tilde{\sigma}(\tilde{\iota}^{-1}[x,+\infty))$};
      \draw[decorate,decoration={brace,amplitude=5,mirror,raise=0.5}] (11.5,-2.9) --
      (11.5,0) node[pos=0.5,right=5,black]{\scriptsize $\tilde{\sigma}(\tilde{\iota}^{-1}(-\infty,x])$};
    \end{tikzpicture}
  \end{center}
\caption{Illustration of Lemma~\ref{lem:preserve-basic-pp-aux-4},
  Case~\ref{item:preserve-basic-pp-aux-4-case1}.}\label{fig:preserve-basic-pp-aux-4}
\end{figure}
\begin{proof}
  As in the proof of Lemma~\ref{lem:preserve-basic-pp-aux-3}, the
  generalised inverse $\iota^{\dagger}(x)$ is welldefined and $x$
  cannot be contained in $\Img(h)$, in particular
  $\Img(\iota)$. Additionally, we have
  $h^{-1}(-\infty,x]=h^{-1}(-\infty,x)=(-\infty,\gamma)$ and
  $h^{-1}[x,+\infty)=h^{-1}(x,+\infty)=(\gamma,+\infty)$ as well as
  $\iota^{-1}(-\infty,x]=\iota^{-1}(-\infty,x)=(-\infty,\delta)$ and
  $\iota^{-1}[x,+\infty)=\iota^{-1}(x,+\infty)=(\delta,+\infty)$. This
  also yields that $\bar{b}(\bar{g}(\delta))=\gamma$. By
  Lemma~\ref{lem:sandwich-reform}, we conclude\footnote{Note that we
    cannot express
    $\tilde{\sigma}(-\infty,\delta)\subseteq (-\infty,f(y)]$ and
    $\tilde{\sigma}(\delta,+\infty)\subseteq [f(y),+\infty)$ using the
    rich topology from Definition~\ref{def:rich-top} since $\delta$ is
    irrational. This is one of the reasons why we need to redefine
    $\tilde{b}$ instead of transferring the behaviour of $b$ at
    sufficiently many points.}
  \begin{align*}
    \sigma(-\infty,\delta)&=\sigma\left(\iota^{-1}(-\infty,x]\right)\subseteq
                            (-\infty,f(y)]\quad\text{and}\\
    \sigma(\delta,+\infty)&=\sigma\left(\iota^{-1}[x,+\infty)\right)\subseteq
                            [f(y),+\infty).
  \end{align*}
  Since $f(y)\in\Img(\sigma)$ and since $\delta$ is irrational, this
  is only possible if $\sigma$ is locally constant with value $f(y)$
  either below or above $\delta$ (or both). These two situations form
  the cases~\ref{item:preserve-basic-pp-aux-4-case1}
  and~\ref{item:preserve-basic-pp-aux-4-case2}, respectively. We only
  treat the former option.

  \textbf{(i).} By our preparatory reasoning, there exist $q,q'\in\Q$
  with $q<q'<\delta$ and $\sigma(q)=\sigma(q')=f(y)$. The number
  $\bar{g}(\delta)$ is irrational by
  Lemma~\ref{lem:cont-lp-easy-facts}\ref{item:cont-lp-easy-facts-ii},
  and obviously not an element of $\R\setminus\Img(\bar{g})$. Thus,
  $\bar{g}(\delta)$ cannot be contained in
  $\bar{z}\cup\bar{z}'$. Consequently,
  $\gamma=\bar{b}(\bar{g}(\delta))$ cannot be contained in
  $\bar{w}\cup\bar{w}'$, and we conclude $w_{-}<\gamma<w_{+}$ from
  $z_{-}<\bar{g}(\delta)<z_{+}$. Since $\delta\in\I=\Cont(g)$, we can
  pick $q,q'$ close enough to $\delta$ to ascertain
  $z_{-}<g(q)<g(q')<\bar{g}(\delta)<z_{+}$.

  \textbf{(ii).} By the definitions of $z_{-}$ and $z_{+}$ and the
  fact that $\bar{z}\mapsto\bar{w}$, $\bar{z}'\mapsto\bar{w}'$ is
  strictly increasing, the finite partial map $\bar{z}\mapsto\bar{w}$,
  $\bar{z}'\mapsto\bar{w}'$ and
  $\tilde{u}\mapsto\hat{u},\tilde{v}\mapsto\hat{v}$ is strictly
  increasing. For the second statement, we check the assumptions of
  Lemma~\ref{lem:sandwich-reform}. Note that
  \begin{align*}
    \tilde{\iota}^{-1}(-\infty,x]&=g^{-1}(\tilde{b}^{-1}(-\infty,\gamma))\subseteq
                                   g^{-1}(-\infty,\tilde{v})\subseteq
                                   (-\infty,q']\qquad\text{and}\\
    \tilde{\iota}^{-1}[x,+\infty)&=g^{-1}(\tilde{b}^{-1}(\gamma,+\infty))\subseteq
                                   g^{-1}(\tilde{u},+\infty)\subseteq [q,+\infty),
  \end{align*}
  so
  \begin{align*}
    \tilde{\sigma}\left(\tilde{\iota}^{-1}(-\infty,x]\right)&\subseteq\tilde{\sigma}(-\infty,q']\subseteq
                                                              (-\infty,\tilde{\sigma}(q')]=(-\infty,\sigma(q')]=(-\infty,f(y)]\quad\text{and}\\
    \tilde{\sigma}\left(\tilde{\iota}^{-1}[x,+\infty)\right)&\subseteq\tilde{\sigma}[q,+\infty)\subseteq
                                                              [\tilde{\sigma}(q),+\infty)=[\sigma(q),+\infty)=[f(y),+\infty).\qedhere
  \end{align*}
\end{proof}
Our final auxiliary lemma treats the second subcase of
$x\in (\inf h,\sup h)$ and
$h^{\dagger}(x),\iota^{\dagger}(x)\in\I$.
\begin{lemma}[see
  Figure~\ref{fig:preserve-basic-pp-aux-5}]\label{lem:preserve-basic-pp-aux-5}
  Let $\sigma,\tilde{\sigma},f,g,h,a,b$ such that $(*)$ holds. Let
  further $x,y\in\Q$ such that $a(x)=y$.

  Suppose that $x\in (\inf h,\sup h)$ with
  $\gamma:=h^{\dagger}(x)\in\I$ and
  $\delta:=\iota^{\dagger}(x)\in\I$. Additionally, suppose that
  $f(y)\notin\Img(\sigma)$. Let $\bar{z}$ and $\bar{w}$ be tuples in
  $\Q$ and let $\bar{z}'$ and $\bar{w}'$ be tuples in
  $(\R\setminus\Img(\bar{g}))\cap\I$ and $\Dc^{\I}(h)$, respectively,
  such that $b(\bar{z})=\bar{w}$ and
  $\bar{b}(\bar{z}')=\bar{w}'$. Assume that $\bar{z}\cup\bar{z}'$
  contains both an element greater and less than
  $\bar{g}(\delta)$. Put $z_{-}$ and $z_{+}$ to be the greatest entry
  of $\bar{z}\cup\bar{z}'$ less than $\bar{g}(\delta)$ and the least
  entry of $\bar{z}\cup\bar{z}'$ greater than $\bar{g}(\delta)$,
  respectively, and put $w_{-}$ and $w_{+}$ to be the corresponding
  entries of $\bar{w}\cup\bar{w}'$. Then the following holds:
  \begin{enumerate}[label=(\roman*)]
  \item\label{item:preserve-basic-pp-aux-5-i}
    $\sigma^{\dagger}(f(y))=\delta$ and
    $\delta\in\Dc^{\I}(\sigma)$. Additionally, $\gamma\in\Dc^{\I}(h)$
    and $\gamma=\bar{b}(\bar{g}(\delta))$. Further,
    $z_{-}<\bar{g}(\delta)<z_{+}$ as well as
    $w_{-}<\gamma<w_{+}$. Finally,
    $\sigma\left(g^{-1}(-\infty,z_{-}]\right)\subseteq (-\infty,f(y))$
    and
    $\sigma\left(g^{-1}[z_{+},+\infty)\right)\subseteq
    (f(y),+\infty)$.
    
    If $z_{\pm}$ and $w_{\pm}$ are rational, then the intervals
    $I_{-}:=g^{-1}(-\infty,z_{-}]$ and $I_{+}:=g^{-1}[z_{+},+\infty)$
    are rational as well.
  \item\label{item:preserve-basic-pp-aux-5-ii} If $z_{\pm}$ and
    $w_{\pm}$ are rational and if
    \begin{displaymath}
      \Img(\tilde{\sigma})\cap\{f(y)\}=\emptyset
      \quad\text{and}\quad\tilde{\sigma}(I_{-})\subseteq
      (-\infty,f(y))
      \quad\text{and}\quad\tilde{\sigma}(I_{+})\subseteq
      (f(y),+\infty),
    \end{displaymath}
    then -- setting
    $\tilde{H}_{-}:=\tilde{\sigma}^{-1}(-\infty,f(y)]=\tilde{\sigma}^{-1}(-\infty,f(y))$
    and
    $\tilde{H}_{+}:=\tilde{\sigma}^{-1}[f(y),+\infty)=\tilde{\sigma}^{-1}(f(y),+\infty)$
    -- we have $\sup g(\tilde{H}_{-})<z_{+}$ as well as
    $z_{-}<\inf g(\tilde{H}_{+})$. Further, there exists
    $\tilde{\rho}\in ((\R\setminus\Img(\bar{g}))\cap\I)\cup
    \bar{g}(\Dc^{\I}(\tilde{\sigma}))$ such that
    $z_{-}<\tilde{\rho}<z_{+}$ and
    $g(\tilde{H}_{-})\subseteq (-\infty,\tilde{\rho})$ as well as
    $g(\tilde{H}_{+})\subseteq (\tilde{\rho},+\infty)$. The finite
    partial map $\bar{z}\mapsto\bar{w}$, $\bar{z}'\mapsto\bar{w}'$ and
    $\tilde{\rho}\mapsto\gamma$ is strictly increasing and,
    additionally, for any $\tilde{b}\in\GG_{\Q}$ such that
    $\tilde{b}(\bar{z})=\bar{w}$, $\bar{\tilde{b}}(\bar{z}')=\bar{w}'$
    and $\bar{\tilde{b}}(\tilde{\rho})=\gamma$, the map $x\mapsto y$
    preserves all basic formulas as a finite partial map from
    $\A(\tilde{\sigma},f,\tilde{\iota})$ to
    $\B(\tilde{\sigma},f,\tilde{\iota})$ where
    $\tilde{\iota}:=h\tilde{b}g$.
  \end{enumerate}
\end{lemma}
\begin{figure}[h]
  \begin{center}
    \begin{tikzpicture}
      \draw (0,-3) -- (0,3);
      \draw (2.5,-3) -- (2.5,3);
      \draw (5,-3) -- (5,3);
      \draw (7,-3) -- (7,3);
      \draw (9,-3) -- (9,3);
      \draw (11.5,-3) -- (11.5,3);
      \draw[->,font=\scriptsize] (0.2,3.1) -- node[above]{$g$} (2.3,3.1);
      \draw[->,font=\scriptsize] (2.7,3.1) -- node[above]{$b$} (4.8,3.1);
      \draw[->,dashed,font=\scriptsize,thick] (2.7,2.7) --
            node[below]{$\tilde{b}$} (4.8,2.7);
      \draw[->,font=\scriptsize] (5.2,3.1) -- node[above]{$h$} (6.8,3.1);
      \draw[->,font=\scriptsize] (7.2,3.1) -- node[above]{$a$} (8.8,3.1);
      \draw[->,font=\scriptsize] (9.2,3.1) -- node[above]{$f$}
      (11.3,3.1);
      \coordinate (x) at (7,0);
      \coordinate (y) at (9,0);
      \coordinate (fy) at (11.5,0);
      \filldraw (x) circle (1pt) node[below right]{\scriptsize $x$};
      \filldraw (y) circle (1pt) node[below right]{\scriptsize $y$};
      \filldraw (fy) circle (1pt) node[right]{\scriptsize $f(y)$};
      \draw[->] (x) -- (y);
      \draw[->] (y) -- (fy);

      \coordinate (delta) at (0,0);
      \coordinate (ginvzminus) at (0,-1.75);
      \coordinate (ginvzplus) at (0,1.5);
      \coordinate (Hsplit) at (0,-0.75);
      \coordinate (zminus) at (2.5,-2.5);
      \coordinate (zplus) at (2.5,1.5);
      \coordinate (gdelta) at (2.5,0);
      \coordinate (rhotilde) at (2.5,-1);
      \coordinate (rhotildel) at (2.5,-1.5);
      \coordinate (rhotildeu) at (2.5,-0.5);
      \coordinate (wminus) at (5,-2.5);
      \coordinate (wplus) at (5,1.5);
      \coordinate (gamma) at (5,0);
      \coordinate (xl) at (7,-0.5);
      \coordinate (xu) at (7,0.5);
      \filldraw (delta) circle (1pt) node[left] {\scriptsize $\delta$};
      \filldraw (zminus) circle (1pt) node[below right] {\scriptsize $z_{-}$};
      \filldraw (zplus) circle (1pt) node[below right] {\scriptsize $z_{+}$};
      \filldraw (gdelta) circle (1pt) node[above right] {\scriptsize
        $\bar{g}(\delta)$};
      \filldraw (rhotilde) circle (1pt) node[below right] {\scriptsize $\tilde{\rho}$};
      \filldraw (wminus) circle (1pt) node[right] {\scriptsize $w_{-}$};
      \filldraw (wplus) circle (1pt) node[right] {\scriptsize
        $w_{+}$};
      \filldraw (gamma) circle (1pt) node[below right] {\scriptsize $\gamma$};

      \draw[->] (delta) -- (gdelta);
      \draw[->,dotted,thick] (ginvzminus) -- (zminus);
      \draw[->,dotted,thick] (ginvzplus) -- (zplus);
      \draw[->,dotted,thick] (Hsplit) -- (rhotildel);
      \draw[->,dotted,thick] (Hsplit) -- (rhotildeu);      
      \draw[->] (gdelta) -- (gamma);
      \draw[->,dashed,thick] (rhotilde) -- (gamma);
      \draw[->] (zminus) -- (wminus);
      \draw[->] (zplus) -- (wplus);
      \draw[->,dotted,thick] (gamma) -- (xl);
      \draw[->,dotted,thick] (gamma) -- (xu);

      \draw[decorate,decoration={brace,aspect=0.6,amplitude=5,raise=0.5}] (0,-3) --
      (0,0) node[pos=0.6,left=5,black]{\scriptsize
        \parbox{2cm}{\hfill $\iota^{-1}(-\infty,x]=$ \\ $\sigma^{-1}(-\infty,f(y))$}};
      \draw[decorate,decoration={brace,aspect=0.25,amplitude=5,raise=0.5}] (0,0) --
      (0,3) node[pos=0.25,left=5,black]{\scriptsize
        \parbox{2cm}{\hfill $\iota^{-1}[x,+\infty)=$ \\
          $\sigma^{-1}(f(y),+\infty)$}};
      \draw[decorate,decoration={brace,amplitude=15,raise=0.5}] (0,-3) --
      (ginvzminus) node[pos=0.5,left=15,black]{\scriptsize $I_{-}$};
      \draw[decorate,decoration={brace,amplitude=15,raise=0.5}] (ginvzplus) --
      (0,3) node[pos=0.5,left=15,black]{\scriptsize $I_{+}$};

      \draw[decorate,decoration={brace,mirror,aspect=0.6,amplitude=15,raise=0.5}] (0,-3) --
      (Hsplit) node[pos=0.6,right=15,black]{\scriptsize $\tilde{H}_{-}$};
      \draw[decorate,decoration={brace,mirror,amplitude=15,raise=0.5}] (Hsplit) --
      (0,3) node[pos=0.5,right=15,black]{\scriptsize $\tilde{H}_{+}$};
      
      \draw[decorate,decoration={brace,amplitude=5,raise=0.5}] (11.5,-2.8) --
      (11.5,-0.5) node[pos=0.5,left=5,black]{\scriptsize
        \parbox{1.9cm}{\hfill $\sigma(-\infty,\delta)=$ \\ $\sigma(\iota^{-1}(-\infty,x])$}};
      \draw[decorate,decoration={brace,amplitude=5,raise=0.5}] (11.5,0.2) --
      (11.5,3) node[pos=0.5,left=5,black]{\scriptsize \parbox{1.9cm}{\hfill $\sigma(\delta,+\infty)=$ \\ $\sigma(\iota^{-1}[x,+\infty))$}};
      \draw[decorate,decoration={brace,amplitude=5,mirror,raise=0.5}] (11.5,-3) --
      (11.5,-0.2) node[pos=0.5,right=5,black]{\scriptsize $\tilde{\sigma}(\tilde{H}_{-})$};
      \draw[decorate,decoration={brace,amplitude=5,mirror,raise=0.5}] (11.5,0.5) --
      (11.5,2.8) node[pos=0.5,right=5,black]{\scriptsize $\tilde{\sigma}(\tilde{H}_{+})$};
    \end{tikzpicture}
  \end{center}
\caption{Illustration of Lemma~\ref{lem:preserve-basic-pp-aux-5}.}\label{fig:preserve-basic-pp-aux-5}
\end{figure}
\begin{proof}
  As in the proof of Lemma~\ref{lem:preserve-basic-pp-aux-4}, the
  generalised inverse $\iota^{\dagger}(x)$ is welldefined and $x$
  cannot be contained in $\Img(h)$, in particular
  $\Img(\iota)$. Additionally,
  $h^{-1}(-\infty,x]=h^{-1}(-\infty,x)=(-\infty,\gamma)$ and
  $h^{-1}[x,+\infty)=h^{-1}(x,+\infty)=(\gamma,+\infty)$ as well as
  $\iota^{-1}(-\infty,x]=\iota^{-1}(-\infty,x)=(-\infty,\delta)$ and
  $\iota^{-1}[x,+\infty)=\iota^{-1}(x,+\infty)=(\delta,+\infty)$. We
  again conclude $\bar{b}(\bar{g}(\delta))=\gamma$. Combining
  Lemma~\ref{lem:sandwich-reform} with $f(y)\notin\Img(\sigma)$, we
  obtain $\sigma(-\infty,\delta)\subseteq (-\infty,f(y))$ as well as
  $\sigma(\delta,+\infty)\subseteq (f(y),+\infty)$ which yields
  $\delta=\sigma^{\dagger}(f(y))$ and
  \begin{align*}
    \sigma\left(g^{-1}(-\infty,z_{-}]\right)&\subseteq\sigma\left(g^{-1}(-\infty,\bar{g}(\delta))\right)\subseteq
                                              (-\infty,f(y))\quad\text{as
                                              well as}\\
    \sigma\left(g^{-1}[z_{+},+\infty)\right)&\subseteq\sigma\left(g^{-1}(\bar{g}(\delta),+\infty)\right)\subseteq
                                              (f(y),+\infty).
  \end{align*}

  \textbf{(i).} We know that $f(y)\in\Q$ cannot be a limit point of
  $\sigma$, so
  $\sup\sigma(-\infty,\delta)<f(y)<\inf\sigma(\delta,+\infty)$, in
  particular $\delta\in\Dc^{\I}(\sigma)$. Using $\LP(h)\subseteq\I$ in
  the same fashion, we obtain $\gamma\in\Dc^{\I}(h)$.  The remaining
  statements $z_{-}<\bar{g}(\delta)<z_{+}$ as well as
  $w_{-}<\gamma<w_{+}$ follow just as in the proof of
  Lemma~\ref{lem:preserve-basic-pp-aux-4}.

  Finally, if $z_{\pm}$ and $w_{\pm}$ are rational, then the intervals
  $I_{-}:=g^{-1}(-\infty,z_{-}]$ and $I_{+}:=g^{-1}[z_{+},+\infty)$
  are rational by
  Lemma~\ref{lem:easy-gen-inv}\ref{item:easy-gen-inv-iii}.

  \textbf{(ii).} We have $\Q=\tilde{H}_{-}\cupdot\tilde{H}_{+}$, so
  $\sup\tilde{H}_{-}=\inf\tilde{H}_{+}$. By our assumption on
  $\tilde{\sigma}$, we know that $I_{-}\cap\tilde{H}_{+}=\emptyset$,
  so $\inf g(\tilde{H}_{+})\geq z_{-}$. In fact, this inequality is
  strict since $\inf g(\tilde{H}_{+})$ is either contained in
  $g(\tilde{H}_{+})$ or irrational by $\LP(g)\subseteq\I$. One argues
  analogously to show $\sup g(\tilde{H}_{-})<z_{+}$. To find
  $\tilde{\rho}$, we distinguish whether
  $\sup\tilde{H}_{-}=\inf\tilde{H}_{+}$ is rational or irrational.

  \textit{Case~1}
  ($\tilde{q}:=\tilde{\sigma}^{\dagger}(f(y))=\sup\tilde{H}_{-}=\inf\tilde{H}_{+}\in\Q$):
  We conclude $\sup g(\tilde{H}_{-})<\inf g(\tilde{H}_{+})$ from
  Lemma~\ref{lem:cont-lp-easy-facts}\ref{item:cont-lp-easy-facts-ii}.
  Combined with $\sup g(\tilde{H}_{-})<z_{+}$ and
  $z_{-}<\inf g(\tilde{H}_{+})$, this implies
  $\max\left(\sup g(\tilde{H}_{-}),z_{-}\right)<\min\left(\inf
    g(\tilde{H}_{+}),z_{+}\right)$. Any irrational $\tilde{\rho}$
  between these two numbers satisfies the requirements -- note that
  $\tilde{\rho}$ is contained in $\R\setminus\Img(\bar{g})$ by
  injectivity of $g$.

  \textit{Case~2}
  ($\tilde{\delta}:=\tilde{\sigma}^{\dagger}(f(y))=\sup\tilde{H}_{-}=\inf\tilde{H}_{+}\in\I$):
  We obtain
  $\sup g(\tilde{H}_{-})=\bar{g}(\tilde{\delta})=\inf
  g(\tilde{H}_{+})$ since $\tilde{\delta}\in\Cont(g)$, so
  $z_{-}<\bar{g}(\tilde{\delta})<z_{+}$. Since
  $\tilde{\sigma}(-\infty,\tilde{\delta})\subseteq (-\infty,f(y))$ and
  $\tilde{\sigma}(\tilde{\delta},+\infty)\subseteq (f(y),+\infty)$ and
  since $f(y)\in\Q$ cannot be a limit point of $\tilde{\sigma}$, we
  conclude $\tilde{\delta}\in\Dc^{\I}(\tilde{\sigma})$. Hence, we set
  $\tilde{\rho}:=\bar{g}(\tilde{\delta})$.

  \smallskip By construction, the finite partial map
  $\bar{z}\mapsto\bar{w}$, $\bar{z}'\mapsto\bar{w}'$ and
  $\tilde{\rho}\mapsto\gamma$ is strictly increasing. For the
  preservation statement, we verify the conditions in
  Lemma~\ref{lem:sandwich-reform}. Note that
  \begin{align*}
    \tilde{\iota}^{-1}(-\infty,x]&=
                                   g^{-1}(-\infty,\tilde{\rho})=\Q\setminus
                                   g^{-1}(\tilde{\rho},+\infty)\subseteq \Q\setminus
                                   g^{-1}(g(\tilde{H}_{+}))=\tilde{H}_{-}
                                   \qquad\text{and}\\
    \tilde{\iota}^{-1}[x,+\infty)&=g^{-1}(\tilde{\rho},+\infty)=\Q\setminus
                                   g^{-1}(-\infty,\tilde{\rho})\subseteq\Q\setminus
                                   g^{-1}(g(\tilde{H}_{-}))=\tilde{H}_{+},
  \end{align*}
  so
  \begin{align*}
    \tilde{\sigma}\left(\tilde{\iota}^{-1}(-\infty,x]\right)&\subseteq\tilde{\sigma}(\tilde{H}_{-})\subseteq
                                                              (-\infty,f(y))\quad\text{and}\\
    \tilde{\sigma}\left(\tilde{\iota}^{-1}[x,+\infty)\right)&\subseteq\tilde{\sigma}(\tilde{H}_{+})\subseteq
                                                              (f(y),+\infty).\qedhere
  \end{align*}
\end{proof}
\begin{remark}\label{rem:preserve-basic-pp-aux-5}
  Examining the last proof more closely, one observes that we never
  used $a(x)=y$ other than via the inclusion of intervals from
  Lemma~\ref{lem:sandwich-reform}. Hence, we in fact proved the
  following slightly stronger statement which will be useful when
  amalgamating the auxiliary lemmas: \medskip

  \it Let $\sigma,\tilde{\sigma},f,g,h,a,b$ such that $(*)$ holds. Let
  further $x,y\in\Q$ such that
  \begin{displaymath}
    \sigma\left(\iota^{-1}(-\infty,x]\right)\subseteq
    (-\infty,f(y)]\quad\text{and}\quad \sigma\left(\iota^{-1}[x,+\infty)\right)\subseteq[f(y),+\infty).
  \end{displaymath}
  Suppose that $x\in (\inf h,\sup h)$ with
  $\gamma:=h^{\dagger}(x)\in\I$ and
  $\delta:=\iota^{\dagger}(x)\in\I$. Additionally, suppose that
  $f(y)\notin\Img(\sigma)$. Let $\bar{z}$ and $\bar{w}$ be tuples in
  $\Q$ and let $\bar{z}'$ and $\bar{w}'$ be tuples in
  $(\R\setminus\Img(\bar{g}))\cap\I$ and $\Dc^{\I}(h)$, respectively,
  such that $b(\bar{z})=\bar{w}$ and
  $\bar{b}(\bar{z}')=\bar{w}'$. Assume that $\bar{z}\cup\bar{z}'$
  contains both an element greater and less than
  $\bar{g}(\delta)$. Put $z_{-}$ and $z_{+}$ to be the greatest entry
  of $\bar{z}\cup\bar{z}'$ less than $\bar{g}(\delta)$ and the least
  entry of $\bar{z}\cup\bar{z}'$ greater than $\bar{g}(\delta)$,
  respectively, and put $w_{-}$ and $w_{+}$ to be the corresponding
  entries of $\bar{w}\cup\bar{w}'$. Then the following holds:
  \begin{enumerate}[label=(\roman*)]
  \item $\sigma^{\dagger}(f(y))=\delta$ and
    $\delta\in\Dc^{\I}(\sigma)$. Additionally, $\gamma\in\Dc^{\I}(h)$
    and $\gamma=\bar{b}(\bar{g}(\delta))$.  Further,
    $z_{-}<\bar{g}(\delta)<z_{+}$ as well as
    $w_{-}<\gamma<w_{+}$. Finally,
    $\sigma\left(g^{-1}(-\infty,z_{-}]\right)\subseteq (-\infty,f(y))$
    and
    $\sigma\left(g^{-1}[z_{+},+\infty)\right)\subseteq
    (f(y),+\infty)$.

    If $z_{\pm}$ and $w_{\pm}$ are rational, then the intervals
    $I_{-}:=g^{-1}(-\infty,z_{-}]$ and $I_{+}:=g^{-1}[z_{+},+\infty)$
    are rational as well.
  \item If $z_{\pm}$ and $w_{\pm}$ are rational and if
    \begin{displaymath}
      \Img(\tilde{\sigma})\cap\{f(y)\}=\emptyset
      \quad\text{and}\quad\tilde{\sigma}(I_{-})\subseteq
      (-\infty,f(y))
      \quad\text{and}\quad\tilde{\sigma}(I_{+})\subseteq
      (f(y),+\infty),
    \end{displaymath}
    then -- setting
    $\tilde{H}_{-}:=\tilde{\sigma}^{-1}(-\infty,f(y)]=\tilde{\sigma}^{-1}(-\infty,f(y))$
    and
    $\tilde{H}_{+}:=\tilde{\sigma}^{-1}[f(y),+\infty)=\tilde{\sigma}^{-1}(f(y),+\infty)$
    -- we have $\sup g(\tilde{H}_{-})<z_{+}$ as well as
    $z_{-}<\inf g(\tilde{H}_{+})$. Further, there exists
    $\tilde{\rho}\in ((\R\setminus\Img(\bar{g}))\cap\I)\cup
    \bar{g}(\Dc^{\I}(\tilde{\sigma}))$ such that
    $z_{-}<\tilde{\rho}<z_{+}$ and
    $g(\tilde{H}_{-})\subseteq (-\infty,\tilde{\rho})$ as well as
    $g(\tilde{H}_{+})\subseteq (\tilde{\rho},+\infty)$. The finite
    partial map $\bar{z}\mapsto\bar{w}$, $\bar{z}'\mapsto\bar{w}'$ and
    $\tilde{\rho}\mapsto\gamma$ is strictly increasing and,
    additionally, for any $\tilde{b}\in\GG_{\Q}$ such that
    $\tilde{b}(\bar{z})=\bar{w}$, $\bar{\tilde{b}}(\bar{z}')=\bar{w}'$
    and $\bar{\tilde{b}}(\tilde{\rho})=\gamma$, we have
    \begin{displaymath}
      \tilde{\sigma}\left(\tilde{\iota}^{-1}(-\infty,x]\right)\subseteq
      (-\infty,f(y)]\quad\text{and}\quad\tilde{\sigma}\left(\tilde{\iota}^{-1}[x,+\infty)\right)\subseteq[f(y),+\infty),
    \end{displaymath}
    where $\tilde{\iota}:=h\tilde{b}g$.
  \end{enumerate}
\end{remark}

\subsection{Proving the Variation Lemma~\ref{lem:variation}, full}
\label{sec:prov-var-lemma-full}
Finally, we amalgamate the special cases.
\begin{proof}[Proof (of the Variation Lemma~\ref{lem:variation}).]
  We construct $O$ as an intersection of $\TT_{rich}$-subbasic open
  sets, i.e.~of sets of the
  types~\ref{item:types-0},~\ref{item:types-1},~\ref{item:types-2},~\ref{item:types-3}.

  By adding to the intersection $O$ the condition that
  $\tilde{\sigma}$ has the same boundedness type as $h$
  (type~\ref{item:types-2}), we can ascertain that $(*)$
  holds. Considering that $\bar{x}\mapsto\bar{y}$ automatically
  preserves the formulas $z_{i}<z_{j}$ since $a(\bar{x})=\bar{y}$ and
  that all the other basic formulas are unary, it suffices to pick the
  automorphism $\tilde{b}\in\GG_{\Q}$ in such a way that the map
  $x\mapsto y$ preserves all basic formulas for each corresponding
  pair $x,y$ in $\bar{x},\bar{y}$.

  First, we treat those corresponding pairs $x,y$ in $\bar{x},\bar{y}$
  for which
  \begin{enumerate}[label=(\alph*)]
  \item\label{item:proof-variation-i} $x\notin (\inf h,\sup h)$\quad
    OR
  \item\label{item:proof-variation-ii} $x\in (\inf h,\sup h)$ with
    $h^{\dagger}(x)\in\Q$\quad OR
  \item\label{item:proof-variation-iii} $x\in (\inf h,\sup h)$ with
    $h^{\dagger}(x)\in\I$ and $\iota^{\dagger}(x)\in\Q$.
  \end{enumerate}
  Applying
  Lemmas~\ref{lem:preserve-basic-pp-aux-1},~\ref{lem:preserve-basic-pp-aux-2}
  and~\ref{lem:preserve-basic-pp-aux-3} each yields a finite
  intersection of sets of
  types~\ref{item:types-0},~\ref{item:types-1},~\ref{item:types-2},~\ref{item:types-3}
  and additional conditions of the form $\tilde{b}(z)=w=b(z)$ for
  $z,w\in\Q$ or $\bar{\tilde{b}}(z')=w'=\bar{b}(z')$ for
  $z'\in (\R\setminus\Img(\bar{g}))\cap\I$ and $w'\in\Dc^{\I}(h)$
  under which $x\mapsto y$ always preserves all basic formulas as a
  finite partial map from $\A(\tilde{\sigma},f,\tilde{\iota})$ to
  $\B(\tilde{\sigma},f,\tilde{\iota})$ where
  $\tilde{\iota}:=h\tilde{b}g$. We add the sets of
  types~\ref{item:types-0},~\ref{item:types-1},~\ref{item:types-2},~\ref{item:types-3}
  to the intersection $O$, we add the points $z$ and $w$ to
  $\bar{z}^{*}$ and $\bar{w}^{*}$, respectively, and we add the points
  $z'$ and $w'$ to $\bar{z}'$ and $\bar{w}'$,
  respectively. Summarising, we obtain that if $\tilde{\sigma}$ is
  contained in the set $O$ constructed thus far and if
  $\tilde{b}(\bar{z}^{*})=\bar{w}^{*}$ and
  $\tilde{b}(\bar{z}')=\bar{w}'$, then $x\mapsto y$ preserves all
  basic formulas for each corresponding pair $x,y$ with one of the
  three
  properties~\ref{item:proof-variation-i}\nobreakdash-\ref{item:proof-variation-iii}.
  
  It remains to consider those corresponding pairs $x,y$ in
  $\bar{x},\bar{y}$ for which
  \begin{enumerate}[label=(\alph*),start=4]
  \item\label{item:proof-variation-iv} $x\in (\inf h,\sup h)$ with
    $\gamma:=h^{\dagger}(x)\in\I$ and
    $\delta:=\iota^{\dagger}(x)\in\I$.
  \end{enumerate}
  Put $z_{-}$ and $z_{+}$ to be the greatest entry of
  $\bar{z}\cup\bar{z}^{*}\cup\bar{z}'$ less than $\bar{g}(\delta)$ and
  the least entry of $\bar{z}\cup\bar{z}^{*}\cup\bar{z}'$ greater than
  $\bar{g}(\delta)$, respectively, and put $w_{-}$ and $w_{+}$ to be
  the corresponding entries of
  $\bar{w}\cup\bar{w}^{*}\cup\bar{w}'$. As a first step,
  Lemmas~\ref{lem:preserve-basic-pp-aux-4}
  and~\ref{lem:preserve-basic-pp-aux-5} yield that
  $z_{-}<\bar{g}(\delta)<z_{+}$ (as well as $w_{-}<\gamma<w_{+}$) and
  $\gamma=\bar{b}(\bar{g}(\delta))$. Hence, we can find
  \emph{rationals} $\hat{z}_{-},\hat{z}_{+}\in\Q$ such that
  $z_{-}<\hat{z}_{-}<\bar{g}(\delta)<\hat{z}_{+}<z_{+}$.  We add
  $\hat{z}_{\pm}$ to $\bar{z}^{*}$ and
  $\hat{w}_{\pm}:=b(\hat{z}_{\pm})$ to $\bar{w}^{*}$. In this way, we
  can assume that $z_{\pm}$ and $w_{\pm}$ are always rational for each
  corresponding pair $x,y$.

  If $x_{1},y_{1}$ and $x_{2},y_{2}$ are two such pairs (without loss
  of generality, let $x_{1}<x_{2}$) and if
  \begin{displaymath}
    \gamma_{1}:=h^{\dagger}(x_{1})<\gamma_{2}:=h^{\dagger}(x_{2}),
  \end{displaymath}
  we enrich $\bar{z}^{*}$ and $\bar{w}^{*}$ even further: putting
  $\delta_{1}:=\iota^{\dagger}(x_{1})$ and
  $\delta_{2}:=\iota^{\dagger}(x_{2})$, we know that
  \begin{displaymath}
    \bar{b}(\bar{g}(\delta_{1}))=\gamma_{1}<\gamma_{2}=\bar{b}(\bar{g}(\delta_{2}))
  \end{displaymath}
  by Lemmas~\ref{lem:preserve-basic-pp-aux-4}
  and~\ref{lem:preserve-basic-pp-aux-5}, and hence
  $\bar{g}(\delta_{1})<\bar{g}(\delta_{2})$. If we pick
  $\tilde{z},\tilde{w}\in\Q$ such that
  \begin{equation}\label{eq:proof-variation-1}
    \bar{g}(\delta_{1})<\tilde{z}<\bar{g}(\delta_{2})\text{ and }\tilde{w}:=b(\tilde{z}),
  \end{equation}
  then $\gamma_{1}<\tilde{w}<\gamma_{2}$. We add $\tilde{z}$ to
  $\bar{z}^{*}$ and $\tilde{w}$ to $\bar{w}^{*}$. If
  $z_{\pm,1},w_{\pm,1},z_{\pm,2},w_{\pm,2}$ denote the
  values\footnote{They are necessarily rational!}  $z_{\pm},w_{\pm}$
  for $x_{1},y_{1}$ and $x_{2},y_{2}$, respectively, we obtain
  $z_{+,1}\leq\tilde{z}\leq z_{-,2}$ and
  $w_{+,1}\leq\tilde{w}\leq w_{-,2}$. Distinguishing cases, we
  conclude that whichever combination of
  Lemmas~\ref{lem:preserve-basic-pp-aux-4}
  and~\ref{lem:preserve-basic-pp-aux-5} applies to $x_{1},y_{1}$ and
  $x_{2},y_{2}$, the resulting conditions on $\tilde{b}$ will be
  compatible, i.e.~strictly increasing. By way of example, consider
  the case that $x_{1},y_{1}$ fall into the scope of
  Lemma~\ref{lem:preserve-basic-pp-aux-4} and $x_{2},y_{2}$ fall into
  the scope of Lemma~\ref{lem:preserve-basic-pp-aux-5}. Then we are
  required to pick
  \begin{displaymath}
    \tilde{u},\tilde{v},\hat{u},\hat{v}\in\Q \quad\text{and}\quad
    \tilde{\rho}\in ((\R\setminus\Img(\bar{g}))\cap\I)\cup
    \bar{g}(\Dc^{\I}(\tilde{\sigma}))
  \end{displaymath}
  with (in particular)
  \begin{displaymath}
    z_{-,1}<\tilde{u}<\tilde{v}<z_{+,1}\leq
    z_{-,2}<\tilde{\rho}<z_{+,2}\quad\text{and}\quad
    w_{-,1}<\hat{u}<\hat{v}<w_{+,1}\leq w_{-,2}<\gamma_{2}<w_{+,2}.
  \end{displaymath}
  Thus, for any $\tilde{u},\tilde{v},\hat{u},\hat{v},\tilde{\rho}$ we
  could pick, the finite partial map $\bar{z}\mapsto\bar{w}$,
  $\bar{z}^{*}\mapsto\bar{w}^{*}$, $\bar{z}'\mapsto\bar{w}'$,
  $\tilde{u}\mapsto\hat{u}$, $\tilde{v}\mapsto\hat{v}$ and
  $\tilde{\rho}\mapsto\gamma_{2}$ is automatically strictly
  increasing.
  
  Finally, we treat the possibility that
  \begin{displaymath}
    h^{\dagger}(x_{1})=h^{\dagger}(x_{2}).
  \end{displaymath}
  We will show that we can reduce to a single application of
  Lemma~\ref{lem:preserve-basic-pp-aux-4},
  Lemma~\ref{lem:preserve-basic-pp-aux-5} or
  Remark~\ref{rem:preserve-basic-pp-aux-5}. First, let $x_{1},y_{1}$
  and $x_{2},y_{2}$ and $x_{3},y_{3}$ be three corresponding pairs
  with $x_{1}<x_{2}<x_{3}$ (and consequently $y_{1}<y_{2}<y_{3}$) but
  $h^{\dagger}(x_{1})=h^{\dagger}(x_{2})=h^{\dagger}(x_{3})$. Then
  \begin{displaymath}
    h^{-1}(-\infty,x_{1}]=h^{-1}(-\infty,x_{2}]=h^{-1}(-\infty,x_{3}]\text{
      and } h^{-1}[x_{1},+\infty)=h^{-1}[x_{2},+\infty)=h^{-1}[x_{3},+\infty),
  \end{displaymath}
  so
  \begin{displaymath}
    \tilde{\iota}^{-1}(-\infty,x_{1}]=\tilde{\iota}^{-1}(-\infty,x_{2}]=\tilde{\iota}^{-1}(-\infty,x_{3}]\text{
      and }
    \tilde{\iota}^{-1}[x_{1},+\infty)=\tilde{\iota}^{-1}[x_{2},+\infty)=\tilde{\iota}^{-1}[x_{3},+\infty)
  \end{displaymath}
  for \emph{all} $\tilde{\iota}=h\tilde{b}g$ we could pick in the
  sequel. It is immediate from Lemma~\ref{lem:sandwich-reform} that we
  can drop $x_{2},y_{2}$ from $\bar{x},\bar{y}$; more precisely: if
  $x_{1}\mapsto y_{1}$ and $x_{3}\mapsto y_{3}$ preserve all basic
  formulas, then so does $x_{2}\mapsto y_{2}$. Hence, we can assume
  that $\bar{x},\bar{y}$ contains only two corresponding pairs
  $x_{1},y_{1}$ and $x_{2},y_{2}$ with
  $h^{\dagger}(x_{1})=h^{\dagger}(x_{2})$. If additionally
  $f(y_{1})=f(y_{2})$, we can drop one of the pairs from
  $\bar{x},\bar{y}$ and apply Lemma~\ref{lem:preserve-basic-pp-aux-4}
  or~\ref{lem:preserve-basic-pp-aux-5} to the remaining one. If on the
  other hand $f(y_{1})<f(y_{2})$, we apply
  Lemma~\ref{lem:sandwich-reform} to $x_{1}\mapsto y_{1}$ and
  $x_{2}\mapsto y_{2}$ as finite partial maps from $\A$ to $\B$ to
  obtain
  \begin{displaymath}
    \sigma\left(\iota^{-1}(-\infty,x_{1}]\right)\subseteq
    (-\infty,f(y_{1})]\quad\text{and}\quad\sigma\left(\iota^{-1}[x_{1},+\infty)\right)\subseteq [f(y_{2}),+\infty).
  \end{displaymath}
  Since $\iota^{-1}(-\infty,x_{1}]$ and $\iota^{-1}[x_{1},+\infty)$
  partition the whole of $\Q$ (note that $x_{1}\notin\Img(\iota)$),
  this implies $\Img(\sigma)\cap (f(y_{1}),f(y_{2}))=\emptyset$. We
  add the condition
  \begin{displaymath}
    \Img(\tilde{\sigma})\cap (f(y_{1}),f(y_{2}))=\emptyset\quad \text{(type~\ref{item:types-3})}
  \end{displaymath}
  to the intersection $O$ and pick $\hat{y}\in\Q$ such that
  $f(y_{1})<f(\hat{y})<f(y_{2})$. Then
  \begin{displaymath}
    \sigma\left(\iota^{-1}(-\infty,x_{1}]\right)\subseteq
    (-\infty,f(\hat{y})]\quad\text{and}\quad\sigma\left(\iota^{-1}[x_{1},+\infty)\right)\subseteq [f(\hat{y}),+\infty).
  \end{displaymath}
  Applying Remark~\ref{rem:preserve-basic-pp-aux-5} to the pair
  $x_{1},\hat{y}$ one obtains
  \begin{displaymath}
    \tilde{\sigma}\left(\tilde{\iota}^{-1}(-\infty,x_{1}]\right)\subseteq
    (-\infty,f(\hat{y})]\quad\text{and}\quad\tilde{\sigma}\left(\tilde{\iota}^{-1}[x_{1},+\infty)\right)\subseteq [f(\hat{y}),+\infty)
  \end{displaymath}
  under suitable conditions on $\tilde{\sigma}$ and $\tilde{b}$ (see
  below). By our choice of $\hat{y}$, since
  $\tilde{\iota}^{-1}[x_{2},+\infty)=\tilde{\iota}^{-1}[x_{1},+\infty)$
  and since $\Img(\tilde{\sigma})\cap (f(y_{1}),f(y_{2}))=\emptyset$,
  Lemma~\ref{lem:sandwich-reform} yields that this is equivalent to
  $x_{1}\mapsto y_{1}$ and $x_{2}\mapsto y_{2}$ both preserving all
  basic formulas.

  To complete the proof, we apply either
  Lemma~\ref{lem:preserve-basic-pp-aux-4},
  Lemma~\ref{lem:preserve-basic-pp-aux-5} or
  Remark~\ref{rem:preserve-basic-pp-aux-5} (the latter only if we use
  the reduction from two instances of
  Lemmas~\ref{lem:preserve-basic-pp-aux-4}
  or~\ref{lem:preserve-basic-pp-aux-5} to a single instance of
  Remark~\ref{rem:preserve-basic-pp-aux-5} as derived above) to each
  corresponding pair $x,y$ in $\bar{x},\bar{y}$
  satisfying~\ref{item:proof-variation-iv}. This yields additional
  sets of
  types~\ref{item:types-0},~\ref{item:types-1},~\ref{item:types-2},~\ref{item:types-3}
  and additional tuples $\bar{\zeta}^{*},\bar{\eta}^{*}$ in $\Q$,
  $\bar{\zeta}'$ in $(\R\setminus\Img(\bar{g}))\cap\I$,
  $\bar{\zeta}''$ in $\bar{g}(\Dc^{\I}(\tilde{\sigma}))$ and
  $\bar{\eta}',\bar{\eta}''$ in $\Dc^{\I}(h)$ such that for all these
  pairs $x,y$, the map $x\mapsto y$ preserves all basic formulas as a
  finite partial map from $\A(\tilde{\sigma},f,\tilde{\iota})$ to
  $\B(\tilde{\sigma},f,\tilde{\iota})$ where
  $\tilde{\iota}:=h\tilde{b}g$, whenever $\tilde{\sigma}$ is contained
  in the the additional sets and $\tilde{b}\in\GG_{\Q}$ satisfies
  $\tilde{b}(\bar{z})=\bar{w}$, $\tilde{b}(\bar{z}^{*})=\bar{w}^{*}$,
  $\bar{\tilde{b}}(\bar{z}')=\bar{w}'$ as well as
  $\tilde{b}(\bar{\zeta}^{*})=\bar{\eta}^{*}$,
  $\bar{\tilde{b}}(\bar{\zeta}')=\bar{\eta}'$,
  $\bar{\tilde{b}}(\bar{\zeta}'')=\bar{\eta}''$. We add the additional
  sets to the intersection $O$ and add the tuples
  $\bar{\zeta}^{*},\bar{\eta}^{*}$ to $\bar{z}^{*},\bar{w}^{*}$, the
  tuples $\bar{\zeta}',\bar{\eta}'$ to $\bar{z}',\bar{w}'$ and the
  tuples $\bar{\zeta}'',\bar{\eta}''$ to $\bar{z}'',\bar{w}''$,
  respectively. Note that the resulting finite partial map
  $\bar{z}\mapsto\bar{w}$, $\bar{z}^{*}\mapsto\bar{w}^{*}$,
  $\bar{z}'\mapsto\bar{w}'$, $\bar{z}''\mapsto\bar{w}''$ is strictly
  increasing: different entries of the new tuples
  $\bar{\zeta}^{*},\bar{\zeta}',\bar{\zeta}'',\bar{\eta}^{*},\bar{\eta}',\bar{\eta}''$
  cannot interfere with each other since the generalised inverses
  $h^{\dagger}(x)$ are pairwise distinct and since we added the
  elements $\tilde{z}$ and $\tilde{w}$
  from~\eqref{eq:proof-variation-1} to $\bar{z}^{*}$ and
  $\bar{w}^{*}$.

  If $\tilde{\sigma}\in O$ and if $\tilde{b}\in\GG_{\Q}$ satisfies
  $\tilde{b}(\bar{z})=\bar{w}$, $\tilde{b}(\bar{z}^{*})=\bar{w}^{*}$,
  $\bar{\tilde{b}}(\bar{z}')=\bar{w}'$,
  $\bar{\tilde{b}}(\bar{z}'')=\bar{w}''$, then by our previous
  construction of $\bar{z}^{*},\bar{z}',\bar{w}^{*},\bar{w}'$, the
  finite partial map $x\mapsto y$ preserves all basic formulas as a
  map from $\A(\tilde{\sigma},f,\tilde{\iota})$ to
  $\B(\tilde{\sigma},f,\tilde{\iota})$ not only for for each pair
  $x,y$ with property~\ref{item:proof-variation-iv} but also for each
  pair $x,y$ with one of the
  properties~\ref{item:proof-variation-i}\nobreakdash-\ref{item:proof-variation-iii}
  -- thus completing the proof.
\end{proof}

\section{Reduction of the rich to the pointwise topology}
\label{sec:reduct-pointw-topol}
The aim of this section is to prove
Proposition~\ref{prop:reduct-pointw-topol}. We will argue in several
steps, each having the following general form:
\begin{notation}
  If $\TT_{a}$ and $\TT_{b}$ are topologies on $\MM_{\Q}$ with
  $\TT_{pw}\subseteq\TT_{a},\TT_{b}$, then
  $\TT_{a}\rightsquigarrow\TT_{b}$ shall denote the following
  statement\footnote{In many (but not all!) applications of this
    notation, we will have $\TT_{b}\subseteq\TT_{a}$.}:
  \begin{quote}
    Let $\TT$ be a Polish semigroup topology on $\MM_{\Q}$ such that
    $\TT_{pw}\subseteq\TT\subseteq\TT_{a}$.

    \noindent Then $\TT\subseteq\TT_{b}$.
  \end{quote}
\end{notation}

We will require an additional auxiliary type of subsets of $\MM_{\Q}$
which encompasses type~\ref{item:types-3} (see
Definition~\ref{def:types}).
\begin{definition}
  \hspace{0mm}
  \begin{enumerate}[label=(\arabic*),ref=\arabic*,start=4]
  \item\label{item:types-4}
    $O_{A}^{(4)}:=\set{s\in\MM_{\Q}}{\Img(s)\subseteq A}$ \quad for $A\subseteq\Q$\hfill
    (restricting)
  \end{enumerate}
\end{definition}
The proof will proceed along the following route:
\begin{displaymath}
  \TT_{rich}=\TT_{0123}\stackrel{\ref{lem:0123-squig-01cls23opn}}{\rightsquigarrow}\TT_{01^{cls}23^{opn}}\stackrel{\ref{lem:01cls23opn-squig-024}}{\rightsquigarrow}\TT_{024}\stackrel{\ref{lem:024-squig-023opn}}{\rightsquigarrow}\TT_{023^{opn}}\stackrel{\ref{lem:023opn-squig-03opn}}{\rightsquigarrow}\TT_{03^{opn}}\stackrel{\ref{lem:03opn-squig-pw}}{\rightsquigarrow}\TT_{0}=\TT_{pw}
\end{displaymath}
\begin{proof}[Proof (of Proposition~\ref{prop:reduct-pointw-topol}
  given
  Lemmas~\ref{lem:0123-squig-01cls23opn},~\ref{lem:01cls23opn-squig-024},~\ref{lem:024-squig-023opn},~\ref{lem:023opn-squig-03opn}
  and~\ref{lem:03opn-squig-pw}).]
  Let $\TT$ be a Polish semigroup topology on $\MM_{\Q}$ with
  $\TT_{pw}\subseteq\TT\subseteq\TT_{rich}=\TT_{0123}$. By
  Lemma~\ref{lem:0123-squig-01cls23opn}, we know that
  $\TT_{pw}\subseteq\TT\subseteq\TT_{01^{cls}23^{opn}}$. Analogously,
  we apply
  Lemmas~\ref{lem:01cls23opn-squig-024},~\ref{lem:024-squig-023opn},~\ref{lem:023opn-squig-03opn}
  and~\ref{lem:03opn-squig-pw} in sequence to finally conclude
  $\TT_{pw}\subseteq\TT\subseteq\TT_{pw}$, i.e.~$\TT=\TT_{pw}$ as
  claimed.
\end{proof}

\subsection{Reductions
  $\TT_{0123}\rightsquigarrow\TT_{01^{cls}23^{opn}}\rightsquigarrow\TT_{024}$}\label{sec:0123-squig-01cls23opn-squig-024}

For the first two reductions, we will need to determine the image of
$\TT_{0123}$-basic open (and in particular
$\TT_{01^{cls}23^{opn}}$-basic open) sets under a suitable left
translation $\lambda_{f}$. This requires a canonical representation of
basic open sets in $\TT_{0123}$ and $\TT_{01^{cls}23^{opn}}$ which
will later be also applied to $\TT_{023^{opn}}$-basic open sets.
\begin{definition}\label{def:stratified-repr}
  Let $O\neq\emptyset$ be a $\TT_{0123}$-basic open (or
  $\TT_{01^{cls}23^{opn}}$-basic open or $\TT_{023^{opn}}$-basic open)
  set, i.e.
  \begin{align}\label{eq:stratified-repr}
    O=\bigcap_{i=1}^{n}O_{x_{i},y_{i}}^{(0)}\cap\bigcap_{j=1}^{m}O_{I_{j},J_{j}}^{(1)}\cap\bigcap_{k=1}^{\widetilde{m}}O_{\tilde{I}_{k},\tilde{J}_{k}}^{(1)}\cap O_{LU}^{(2)}\cap\bigcap_{\ell=1}^{N}O_{K_{\ell}}^{(3)}
  \end{align}
  where $I_{j}=(-\infty,p_{j})$ and
  $\tilde{I}_{k}=(\tilde{p}_{k},+\infty)$ for
  $p_{j},\tilde{p}_{k}\in\Q$.  We call the
  representation~\eqref{eq:stratified-repr} \emph{stratified} if
  \begin{enumerate}[label=(S\arabic*)]
  \item\label{item:stratified-repr-i}
    $\forall i=1,\dots,n-1\colon x_{i}<x_{i+1}$ (then automatically,
    $y_{i}\leq y_{i+1}$ since $O\neq\emptyset$)
  \item\label{item:stratified-repr-ii}
    $\forall j=1,\dots,m-1\colon p_{j}<p_{j+1}\text{ and
    }J_{j}\subseteq J_{j+1}$
  \item\label{item:stratified-repr-iii}
    $\forall k=1,\dots,\widetilde{m}-1\colon
    \tilde{p}_{k}<\tilde{p}_{k+1}\text{ and }\tilde{J}_{k}\supseteq
    \tilde{J}_{k+1}$
  \item\label{item:stratified-repr-iv}
    $\forall\ell=1,\dots,N-1\colon \sup K_{\ell}\leq \inf
    K_{\ell+1}\text{ and }\left(\inf K_{\ell},\sup
      K_{\ell+1}\right)\setminus(K_{\ell}\cup
    K_{\ell+1})\neq\emptyset$
  \item\label{item:stratified-repr-v}
    $\forall j=1,\dots,m\,\forall i=1,\dots,n\colon p_{j}\leq
    x_{i}\Rightarrow y_{i}\notin J_{j}$
  \item\label{item:stratified-repr-vi}
    $\forall k=1,\dots,\widetilde{m}\,\forall i=1,\dots,n\colon
    \tilde{p}_{k}\geq x_{i}\Rightarrow y_{i}\notin\tilde{J}_{k}$
  \item\label{item:stratified-repr-vii}
    $\forall j=1,\dots,m\,\forall\ell=1,\dots,N\colon (J_{j}\cap
    K_{\ell}\neq\emptyset\Rightarrow\exists t\in J_{j}\colon
    t>K_{\ell})$
  \item\label{item:stratified-repr-viii}
    $\forall k=1,\dots,\widetilde{m}\,\forall\ell=1,\dots,N\colon
    (\tilde{J}_{k}\cap K_{\ell}\neq\emptyset\Rightarrow\exists t\in
    \tilde{J}_{k}\colon t<K_{\ell})$
  \end{enumerate}
\end{definition}
\begin{lemma}\label{lem:stratified-repr}
  Any $\TT_{0123}$-basic open set $O$ has a stratified representation.

  The same holds for a $\TT_{01^{cls}23^{opn}}$-basic open set, where
  the resulting representation again consists of sets of
  types~\ref{item:types-0},~\ref{item:types-1cls},~\ref{item:types-2}
  and~\ref{item:types-3opn}.

  The same holds for a $\TT_{023^{opn}}$-basic open set, where the
  resulting representation again consists of sets of
  types~\ref{item:types-0},~\ref{item:types-2}
  and~\ref{item:types-3opn}.
\end{lemma}
\begin{proof}
  We start with any representation
  \begin{displaymath}
    O=\bigcap_{i=1}^{n}O_{x_{i},y_{i}}^{(0)}\cap\bigcap_{j=1}^{m}O_{I_{j},J_{j}}^{(1)}\cap\bigcap_{k=1}^{\widetilde{m}}O_{\tilde{I}_{k},\tilde{J}_{k}}^{(1)}\cap
    O_{LU}^{(2)}\cap\bigcap_{\ell=1}^{N}O_{K_{\ell}}^{(3)}
  \end{displaymath}
  and turn it into a stratified one in several steps, one for each
  item in Definition~\ref{def:stratified-repr}.

  \textbf{(S1).} Rearrange the $x_{i}$ in increasing order.

  \textbf{(S2).} Rearrange the $p_{j}$ in increasing order; if
  $p_{j}=p_{j+1}$, drop the larger set of $J_{j}$ and $J_{j+1}$. If
  $J_{j}$ is not a subset of $J_{j+1}$, then $J_{j+1}\subseteq J_{j}$
  and $O_{I_{j},J_{j}}^{(1)}\cap O_{I_{j+1},J_{j+1}}^{(1)}$ can be
  replaced by $O_{I_{j+1},J_{j+1}}^{(1)}$.

  \textbf{(S3).} Analogously to~(S2).

  \textbf{(S4).} Rearrange the $K_{\ell}$ by increasing order of
  $\inf K_{\ell}$.

  \noindent If $\sup K_{\ell}>\inf K_{\ell+1}$ or if
  $\left(\inf K_{\ell},\sup K_{\ell+1}\right)\subseteq K_{\ell}\cup
  K_{\ell+1}$, then $K_{\ell}\cup K_{\ell+1}$ is again an interval and
  $O_{K_{\ell}}^{(3)}\cap O_{K_{\ell+1}}^{(3)}$ can be replaced by
  $O_{K_{\ell}\cup K_{\ell+1}}^{(3)}$.

  \textbf{(S5).} If $p_{j}\leq x_{i}$ and $y_{i}\in J_{j}$, then
  $O_{x_{i},y_{i}}^{(0)}\cap O_{I_{j},J_{j}}^{(1)}$ can be replaced by
  $O_{x_{i},y_{i}}^{(0)}$.

  \textbf{(S6).} Analogously to~(S5).

  \textbf{(S7).} If $J_{j}\cap K_{\ell}\neq\emptyset$ but no element
  of $J_{j}$ is greater than $K_{\ell}$, then
  $J_{j}\setminus K_{\ell}$ is again a rational interval and
  $O_{I_{j},J_{j}}^{(1)}\cap O_{K_{\ell}}^{(3)}$ can be replaced by
  $O_{I_{j},J_{j}\setminus K_{\ell}}^{(1)}\cap O_{K_{\ell}}^{(3)}$.

  \textbf{(S8).} Analogously to~(S7).

  \medskip If all intervals $J_{j}$ and $\tilde{J}_{k}$ are closed and
  all intervals $K_{\ell}$ are open, then so are the respective
  intervals in the representation we obtain after going
  through~(S1)-(S8), proving the second statement.

  If additionally no sets of type~\ref{item:types-1cls} occur, then
  the representation will only contains sets of
  types~\ref{item:types-0},~\ref{item:types-2}
  and~\ref{item:types-3opn} since the above procedure never generates
  sets of type~\ref{item:types-1} if there are none in the original
  set.
\end{proof}
Now we can provide the result about images of basic open sets under
certain left translations (which can be seen as a generalisation of
Lemma~\ref{lem:left-translation-preimage}). Its proof is somewhat
technical, but the main idea is very straightforward: if
$f,s'\in\MM_{\Q}$ and if $s'$ satisfies
$s'(-\infty,p)\subseteq (-\infty,q)$, then one might be led to believe
that necessarily $fs'(-\infty,p)\subseteq (-\infty,f(q))$. However, in
general only the conclusion $fs'(-\infty,p)\subseteq (-\infty,f(q)]$
is true, namely if the preimage $f^{-1}\{f(q)\}$ contains not only $q$
but also elements less than $q$. This can be ensured by requiring that
$f^{-1}\{f(q)\}$ is an irrational interval. Indeed, if $f$ is also
surjective, then \emph{any} $s$ with
$s(-\infty,p)\subseteq (-\infty,f(q)]$ can be rewritten as $s=fs'$
where $s'(-\infty,p)\subseteq (-\infty,q)$. An analogous fact holds
for sets of type~\ref{item:types-3} -- if $s'$ avoids $[u,v]$, then
$fs'$ in general only avoids $(f(u),f(v))$ and, conversely, if $f$ is
surjective with $f^{-1}\{f(u)\}$, $f^{-1}\{f(v)\}$ irrational and if
$s$ avoids $(f(u),f(v))$, then $s$ can be rewritten as $s=fs'$ where
$s'$ avoids $[u,v]$. Combining these facts for the building blocks of
(stratified representations of) basic open sets requires thorough
bookkeeping.
\begin{lemma}\label{lem:translate-stratified-basic-open}
  Let $O\neq\emptyset$ be a nonempty $\TT_{0123}$-basic open set with
  stratified representation
  \begin{displaymath}
    O=\bigcap_{i=1}^{n}O_{x_{i},y_{i}}^{(0)}\cap\bigcap_{j=1}^{m}O_{(-\infty,p_{j}),J_{j}}^{(1)}\cap\bigcap_{k=1}^{\widetilde{m}}O_{(\tilde{p}_{k},+\infty),\tilde{J}_{k}}^{(1)}\cap
    O_{LU}^{(2)}\cap\bigcap_{\ell=1}^{N}O_{K_{\ell}}^{(3)}.
  \end{displaymath}
  Define
  \begin{displaymath}
    q_{j}:=\sup J_{j},\quad \tilde{q}_{k}:=\inf\tilde{J}_{k},\quad
    u_{\ell}:=\inf K_{\ell}\quad\text{and}\quad v_{\ell}:=\sup K_{\ell}.
  \end{displaymath}
  Let further $f\in\MM_{\Q}$ be unbounded-unbounded such that for all
  $w\in\Img(f)$, the preimage $f^{-1}\{w\}$ is an irrational
  interval. Then (putting $f(\pm\infty):=\pm\infty$) we have
  \begin{multline*}
    \lambda_{f}(O)=\set{s}{\Img(s)\subseteq\Img(f)}\cap\bigcap_{i=1}^{n}O_{x_{i},f(y_{i})}^{(0)}\cap\\
    \bigcap_{j=1}^{m}O_{(-\infty,p_{j}),(-\infty,f(q_{j})]}^{(1)}\cap\bigcap_{k=1}^{\widetilde{m}}O_{(\tilde{p}_{k},+\infty),[f(\tilde{q}_{k}),+\infty)}^{(1)}\cap
    O_{LU}^{(2)}\cap\bigcap_{\ell=1}^{N}O_{(f(u_{\ell}),f(v_{\ell}))}^{(3)}.
  \end{multline*}
\end{lemma}
\begin{proof}
  The inclusion ``$\subseteq$'' is immediate, so we deal only with
  ``$\supseteq$''.

  Take $s\in\MM_{\Q}$ such that $\Img(s)\subseteq\Img(f)$ and
  \begin{equation}\label{eq:proof-translate-stratified-basic-open-i}
    s\in\bigcap_{i=1}^{n}O_{x_{i},f(y_{i})}^{(0)}\cap\bigcap_{j=1}^{m}O_{(-\infty,p_{j}),(-\infty,f(q_{j})]}^{(1)}\cap\bigcap_{k=1}^{\widetilde{m}}O_{(\tilde{p}_{k},+\infty),[f(\tilde{q}_{k}),+\infty)}^{(1)}\cap
    O_{LU}^{(2)}\cap\bigcap_{\ell=1}^{N}O_{(f(u_{\ell}),f(v_{\ell}))}^{(3)}.
  \end{equation}
  We want to find $s'\in O$ such that $s=fs'$. The latter statement is
  equivalent to $s'(s^{-1}\{w\})\subseteq f^{-1}\{w\}$ for all
  $w\in\Img(s)$. Since $\Img(s)\subseteq\Img(f)$, we have
  $f^{-1}\{w\}\neq\emptyset$ for all $w\in\Img(s)$. Note that if one
  takes $s'\vert_{s^{-1}\{w\}}$ to be an increasing map
  $s^{-1}\{w\}\to f^{-1}\{w\}$ independently for each $w\in\Img(s)$,
  their union will be increasing as well since
  $s^{-1}\{w_{1}\}<s^{-1}\{w_{2}\}$ and
  $f^{-1}\{w_{1}\}<f^{-1}\{w_{2}\}$ for all
  $w_{1}<w_{2}$. Additionally requiring $s'\in O$ amounts to the
  following properties:
  \begin{enumerate}[label=(\roman*)]
  \item\label{item:proof-translate-stratified-basic-open-i}
    $\forall i=1,\dots,n\colon s'(x_{i})=y_{i}$
  \item\label{item:proof-translate-stratified-basic-open-ii}
    $s'\in O_{LU}^{(2)}$
  \item\label{item:proof-translate-stratified-basic-open-iii}
    $\forall j=1,\dots,m\,\forall w\in\Img(s)\colon
    s'\left(s^{-1}\{w\}\cap (-\infty,p_{j})\right)\subseteq J_{j}\cap
    f^{-1}\{w\}$
  \item\label{item:proof-translate-stratified-basic-open-iv}
    $\forall k=1,\dots,\widetilde{m}\,\forall w\in\Img(s)\colon
    s'\left(s^{-1}\{w\}\cap (\tilde{p}_{k},+\infty)\right)\subseteq
    \tilde{J}_{k}\cap f^{-1}\{w\}$
  \item\label{item:proof-translate-stratified-basic-open-v}
    $\forall\ell=1,\dots,N\,\forall w\in\Img(s)\colon
    s'\left(s^{-1}\{w\}\right)\cap K_{\ell}=\emptyset$
  \end{enumerate}
  To simplify the proof, we replace
  \ref{item:proof-translate-stratified-basic-open-i} by:
  \begin{enumerate}[label=(\roman*),start=6]
  \item\label{item:proof-translate-stratified-basic-open-vi}
    $\forall i=1,\dots,n\,\forall w\in\Img(s)\colon
    s'\left(s^{-1}\{w\}\cap (-\infty,x_{i})\right)\subseteq
    (-\infty,y_{i}]\cap f^{-1}\{w\}$
  \item\label{item:proof-translate-stratified-basic-open-vii}
    $\forall i=1,\dots,n\,\forall w\in\Img(s)\colon
    s'\left(s^{-1}\{w\}\cap (x_{i},+\infty)\right)\subseteq
    [y_{i},+\infty)\cap f^{-1}\{w\}$
  \end{enumerate}
  If we find $s'$ satisfying
  \ref{item:proof-translate-stratified-basic-open-ii}-\ref{item:proof-translate-stratified-basic-open-vii},
  then we can redefine $s'(x_{i}):=y_{i}$ to obtain $s'\in O$ --
  by~\ref{item:proof-translate-stratified-basic-open-vi}
  and~\ref{item:proof-translate-stratified-basic-open-vii}, the
  resulting map will still be an element of $\MM_{\Q}$; and since
  $O\neq\emptyset$, mapping $x_{i}\mapsto y_{i}$ cannot
  contradict~\ref{item:proof-translate-stratified-basic-open-ii}-\ref{item:proof-translate-stratified-basic-open-v}.

  As a first step, we show that the statements
  in~\ref{item:proof-translate-stratified-basic-open-ii}-\ref{item:proof-translate-stratified-basic-open-vii}
  are already implied by $s'(s^{-1}\{w\})\subseteq f^{-1}\{w\}$ for
  many values $w$ (and therefore automatically satisfied). If
  $w<f(q_{j})$, then $f^{-1}\{w\}\subseteq J_{j}$,
  so~\ref{item:proof-translate-stratified-basic-open-iii}
  automatically holds for $w<f(q_{j})$
  and~\ref{item:proof-translate-stratified-basic-open-vi} for
  $w<f(y_{i})$; a dual argument
  yields~\ref{item:proof-translate-stratified-basic-open-iv} for
  $w>f(\tilde{q}_{k})$
  and~\ref{item:proof-translate-stratified-basic-open-vii} for
  $w>f(y_{i})$. Finally, if $w<f(u_{\ell})$ or $w>f(v_{\ell})$,
  then~\ref{item:proof-translate-stratified-basic-open-v} is
  automatically satisfied as well since
  $f^{-1}\{w\}\cap K_{\ell}=\emptyset$.

  Using~\eqref{eq:proof-translate-stratified-basic-open-i}, we obtain
  that~\ref{item:proof-translate-stratified-basic-open-ii}-\ref{item:proof-translate-stratified-basic-open-vii}
  hold for many more values $w$. For
  instance,~\ref{item:proof-translate-stratified-basic-open-iii} holds
  for $w>f(q_{j})$: since
  $s(-\infty,p_{j})\subseteq (-\infty,f(q_{j})]$, we have
  $s^{-1}\{w\}\cap
  (-\infty,p_{j})=\emptyset$. Similarly,~\ref{item:proof-translate-stratified-basic-open-vi}
  holds for
  $w>f(y_{i})$,~\ref{item:proof-translate-stratified-basic-open-iv}
  holds for $w<f(\tilde{q}_{k})$
  and~\ref{item:proof-translate-stratified-basic-open-vii} holds for
  $w<f(y_{i})$. In~\ref{item:proof-translate-stratified-basic-open-v},
  we do not have to consider $f(u_{\ell})<w<f(v_{\ell})$ since
  $\Img(s)\cap (f(u_{\ell}),f(v_{\ell}))=\emptyset$. Finally, since
  $f$ is unbounded-unbounded, any function $s'$ with $s=fs'$ has the
  same boundedness type as $s$,
  i.e.~\ref{item:proof-translate-stratified-basic-open-ii} is
  automatically satisfied as well.

  Collecting the previous arguments and additionally
  reformulating~\ref{item:proof-translate-stratified-basic-open-v}, it
  suffices to ascertain the following properties (instead
  of~\ref{item:proof-translate-stratified-basic-open-ii}-\ref{item:proof-translate-stratified-basic-open-vii}):
  \begin{enumerate}[label=(\roman*'),start=3]
  \item\label{item:proof-translate-stratified-basic-open-iii'}
    $\forall j=1,\dots,m\colon s'\left(s^{-1}\{f(q_{j})\}\cap
      (-\infty,p_{j})\right)\subseteq J_{j}\cap f^{-1}\{f(q_{j})\}$
  \item\label{item:proof-translate-stratified-basic-open-iv'}
    $\forall k=1,\dots,\widetilde{m}\colon
    s'\left(s^{-1}\{f(\tilde{q}_{k})\}\cap
      (\tilde{p}_{k},+\infty)\right)\subseteq \tilde{J}_{k}\cap
    f^{-1}\{f(\tilde{q}_{k})\}$
  \item\label{item:proof-translate-stratified-basic-open-v'}
    $\forall\ell=1,\dots,N\colon
    s'\left(s^{-1}\{f(u_{\ell})\}\right)\subseteq
    f^{-1}\{f(u_{\ell})\}\setminus K_{\ell}\text{ and
    }s'\left(s^{-1}\{f(v_{\ell})\}\right)\subseteq
    f^{-1}\{f(v_{\ell})\}\setminus K_{\ell}$
  \item\label{item:proof-translate-stratified-basic-open-vi'}
    $\forall i=1,\dots,n\colon s'\left(s^{-1}\{f(y_{i})\}\cap
      (-\infty,x_{i})\right)\subseteq (-\infty,y_{i}]\cap
    f^{-1}\{f(y_{i})\}$
  \item\label{item:proof-translate-stratified-basic-open-vii'}
    $\forall i=1,\dots,n\colon s'\left(s^{-1}\{f(y_{i})\}\cap
      (x_{i},+\infty)\right)\subseteq [y_{i},+\infty)\cap
    f^{-1}\{f(y_{i})\}$
  \end{enumerate}
  We replace $O$ by
  \begin{equation}\label{eq:proof-translate-stratified-basic-open-ii}
    \bigcap_{i=1}^{n}O_{(-\infty,x_{i}),(-\infty,y_{i}]}^{(1)}\cap O_{(x_{i},+\infty),[y_{i},+\infty)}^{(1)}\cap\bigcap_{j=1}^{m}O_{(-\infty,p_{j}),J_{j}}^{(1)}\cap\bigcap_{k=1}^{\widetilde{m}}O_{(\tilde{p}_{k},+\infty),\tilde{J}_{k}}^{(1)}\cap
    O_{LU}^{(2)}\cap\bigcap_{\ell=1}^{N}O_{K_{\ell}}^{(3)},
  \end{equation}
  observing that this representation is still stratified (up to adding
  the elements $x_{i}$ to $\{p_{1},\dots,p_{m}\}$ as well as
  $\{\tilde{p}_{1},\dots,\tilde{p}_{\widetilde{m}}\}$ and
  rearranging).

  Since we have
  \begin{multline*}
    s\in\set{s'}{\Img(s')\subseteq\Img(f)}\cap\bigcap_{i=1}^{n}O_{(-\infty,x_{i}),(-\infty,f(y_{i})]}^{(1)}\cap
    O_{(x_{i},+\infty),[f(y_{i}),+\infty)}^{(1)}\cap\\
    \bigcap_{j=1}^{m}O_{(-\infty,p_{j}),(-\infty,f(q_{j})]}^{(1)}\cap\bigcap_{k=1}^{\widetilde{m}}O_{(\tilde{p}_{k},+\infty),[f(\tilde{q}_{k}),+\infty)}^{(1)}\cap
    O_{LU}^{(2)}\cap\bigcap_{\ell=1}^{N}O_{(f(u_{\ell}),f(v_{\ell}))}^{(3)},
  \end{multline*}
  we can, without loss of generality, subsume
  \ref{item:proof-translate-stratified-basic-open-vi'} and
  \ref{item:proof-translate-stratified-basic-open-vii'} in
  \ref{item:proof-translate-stratified-basic-open-iii'} and
  \ref{item:proof-translate-stratified-basic-open-iv'} so that we only
  have to deal with
  \ref{item:proof-translate-stratified-basic-open-iii'}-\ref{item:proof-translate-stratified-basic-open-v'}.

  If we can find an increasing map satisfying
  \ref{item:proof-translate-stratified-basic-open-iii'}-\ref{item:proof-translate-stratified-basic-open-v'},
  then any extension $s'$ of that map satisfying
  $s'(s^{-1}\{w\})\subseteq f^{-1}\{w\}$ for
  $w\neq f(q_{j}),f(\tilde{q}_{k}),f(u_{\ell}),f(v_{\ell})$
  ($j=1,\dots,m$; $k=1,\dots,\widetilde{m}$; $\ell=1,\dots,N$) will be
  an element of $\MM_{\Q}$ for which
  \ref{item:proof-translate-stratified-basic-open-ii}-\ref{item:proof-translate-stratified-basic-open-vii}
  hold -- thus completing the proof.

  Since the $f$-preimages of single elements are assumed to be
  irrational intervals, the right hand sides in
  \ref{item:proof-translate-stratified-basic-open-iii'}-\ref{item:proof-translate-stratified-basic-open-v'}
  are nonempty\footnote{This is the first half of the main observation
    behind the lemma!}: the preimage $f^{-1}\{f(q_{j})\}$ in
  \ref{item:proof-translate-stratified-basic-open-iii'} is an
  irrational interval which contains $q_{j}$, so $q_{j}$ must be
  contained in the interior of $f^{-1}\{f(q_{j})\}$; we conclude
  $J_{j}\cap f^{-1}\{f(q_{j})\}\neq\emptyset$. For the other items, we
  argue analogously, noting in
  \ref{item:proof-translate-stratified-basic-open-v'} that $u_{\ell}$
  and $v_{\ell}$ are limit points not only of $K_{\ell}$ but also of
  $\Q\setminus K_{\ell}$.

  In the remainder of the proof, we will use that the representation
  of $O$ is stratified to show that combinations of
  \ref{item:proof-translate-stratified-basic-open-iii'}-\ref{item:proof-translate-stratified-basic-open-v'}
  are not contradictory, either. This could only happen if they are
  making statements about the same $s$-preimage, i.e. if the
  $f$-images of some of the points
  $q_{j}, \tilde{q}_{k}, u_{\ell}, v_{\ell}$ coincide. For each
  \begin{displaymath}
    w\in\{f(q_{1}),\dots,f(q_{m}),f(\tilde{q}_{1}),\dots,f(\tilde{q}_{\widetilde{m}}),f(u_{1}),\dots,f(u_{N}),f(v_{1}),\dots,f(v_{N})\},
  \end{displaymath}
  we will define $s'$ on $s^{-1}\{w\}$.

  We first show that we can find an image $s'(z)$ satisfying
  \ref{item:proof-translate-stratified-basic-open-iii'}-\ref{item:proof-translate-stratified-basic-open-v'}
  for each indivdual $z\in s^{-1}\{w\}$. If $f(q_{j})=f(q_{j+1})$,
  then
  $J_{j}\cap f^{-1}\{f(q_{j})\}\subseteq J_{j+1}\cap
  f^{-1}\{f(q_{j+1})\}$ by \ref{item:stratified-repr-ii}. Therefore,
  it suffices to consider the \emph{least} $j$ such that
  $z\in (-\infty,p_{j})$ (if such a $j$ exists) and fulfil
  $s'(z)\in J_{j}$ -- the other conditions of the same type will then
  be automatically satisfied. Analogously, it is enough to consider
  the \emph{greatest} $k$ such that $z\in (\tilde{p}_{k},+\infty)$ (if
  it exists). For given $z$, we can thus reduce
  \ref{item:proof-translate-stratified-basic-open-iii'} and
  \ref{item:proof-translate-stratified-basic-open-iv'} to a
  \emph{single} condition of the respective types (if they occur at
  all). Therefore, we need to map $z$ to the intersection of
  $f^{-1}\{w\}$ and a combination of $J_{j}$ and $\tilde{J}_{k}$ and
  $\bigcap_{\ell=\ell_{1}}^{\ell_{2}}\Q\setminus K_{\ell}$ which
  respectively occur if $z\in (-\infty,p_{j})\land f(q_{j})=w$ and
  $z\in (\tilde{p}_{k},+\infty)\land f(\tilde{q}_{k})=w$ and
  $f(v_{\ell_{1}})=f(u_{\ell_{1}+1})=f(v_{\ell_{1}+1})=\dots,f(u_{\ell_{2}})=w$
  (and possibly $f(u_{\ell_{1}})=w$ or $f(v_{\ell_{2}})=w$ as well --
  we do not know whether the chain
  $f^{-1}\{w\}\cap(\set{u_{\ell}}{\ell=1,\dots,N}\cup\set{v_{\ell}}{\ell=1,\dots,N})$
  begins and ends with an element $u_{*}$ or $v_{*}$!). We distinguish
  cases by the types of sets actually occurring and show that this
  intersection is always nonempty\footnote{This is the second half of
    the main observation behind the lemma!}:
  \begin{itemize}
  \item We have already argued that
    $f^{-1}\{w\}\cap J_{j}\neq\emptyset$,
    $f^{-1}\{w\}\cap \tilde{J}_{k}\neq\emptyset$ and
    $f^{-1}\{w\}\cap\Q\setminus K_{\ell}\neq\emptyset$.
  \item If the sets $\Q\setminus K_{\ell}$ occur for
    $\ell=\ell_{1},\dots,\ell_{2}$, then
    $f^{-1}\{w\}\cap\bigcap_{\ell=\ell_{1}}^{\ell_{2}}\Q\setminus
    K_{\ell}\neq\emptyset$, since
    $f^{-1}\{w\}\subseteq\bigcup_{\ell=\ell_{1}}^{\ell_{2}}K_{\ell}$
    combined with~\ref{item:stratified-repr-iv} would yield
    $f^{-1}\{w\}\subseteq K_{\ell}$ for some $\ell$, contradicting the
    previous item.
  \item If both $J_{j}$ and $\tilde{J}_{k}$ occur, then
    $J_{j}\cap\tilde{J}_{k}\neq\emptyset$ since
    $s(z)\in J_{j}\cap\tilde{J}_{k}$ by $s\in O$. Therefore,
    $\tilde{q}_{k}\leq q_{j}$ and
    $\sup (J_{j}\cap\tilde{J}_{k})=q_{j}$ as well as
    $\inf (J_{j}\cap\tilde{J}_{k})=\tilde{q}_{k}$. We know that both
    $q_{j}$ and $\tilde{q}_{k}$ are contained in the interval
    $f^{-1}\{w\}$, whence
    $f^{-1}\{w\}\cap
    J_{j}\cap\tilde{J}_{k}=J_{j}\cap\tilde{J}_{k}\neq\emptyset$.
  \item If $J_{j}$ and
    $\bigcap_{\ell=\ell_{1}}^{\ell_{2}}\Q\setminus K_{\ell}$ occur,
    then
    $f^{-1}\{w\}\cap
    J_{j}\cap\bigcap_{\ell=\ell_{1}}^{\ell_{2}}\Q\setminus
    K_{\ell}=\emptyset$ would imply
    $f^{-1}\{w\}\cap J_{j}\subseteq K_{\ell}$ for some $\ell$, again
    via~\ref{item:stratified-repr-iv}. We pick any
    $r\in f^{-1}\{w\}\cap J_{j}\neq\emptyset$, and thus
    $r\in K_{\ell}$. By~\ref{item:stratified-repr-vii}, there exists
    $t\in J_{j}$ such that $t>K_{\ell}$. In particular,
    $t\in [r,q_{j}]\subseteq f^{-1}\{w\}$ which yields the
    contradiction $t\in f^{-1}\{w\}\cap J_{j}$ but $t\notin K_{\ell}$.
  \item If $\tilde{J}_{k}$ and
    $\bigcap_{\ell=\ell_{1}}^{\ell_{2}}\Q\setminus K_{\ell}$ occur,
    one argues analogously.
  \item If both $J_{j}$ and $\tilde{J}_{k}$ as well as
    $\bigcap_{\ell=\ell_{1}}^{\ell_{2}}\Q\setminus K_{\ell}$ occur, we
    again derive a contradiction from
    $f^{-1}\{w\}\cap
    J_{j}\cap\tilde{J}_{k}\cap\bigcap_{\ell=\ell_{1}}^{\ell_{2}}\Q\setminus
    K_{\ell}=\emptyset$. As in the previous
    cases,~\ref{item:stratified-repr-iv} yields
    $\emptyset\neq f^{-1}\{w\}\cap J_{j}\cap\tilde{J}_{k}\subseteq
    K_{\ell}$ for some $\ell$. By \ref{item:stratified-repr-vii},
    there exists $t\in J_{j}$ such that $t>K_{\ell}$. Increasing $t$
    if necessary, one obtains the contradiction
    $t\in J_{j}\cap\tilde{J}_{k}=f^{-1}\{w\}\cap
    J_{j}\cap\tilde{J}_{k}$ but $t\notin K_{\ell}$.
  \end{itemize}
  Finally, we combine our arguments for each individual
  $z\in s^{-1}\{w\}$ to a definition of $s'$ on the whole of
  $s^{-1}\{w\}$. We define an equivalence relation $\sim$ on
  $s^{-1}\{w\}$ by putting $z\sim z'$ if and only if $z$ and $z'$ are
  contained in the same intervals of the shapes $(-\infty,p_{j})$ and
  $(\tilde{p}_{k},+\infty)$. Clearly, there are only finitely many
  $\sim$-equivalence classes. Let $z_{1},\dots,z_{M}$ be a system of
  representatives of these equivalence classes which we assume to be
  arranged in increasing order. By the previous part of our proof, we
  can pick images $s'(z_{1}),\dots,s'(z_{M})\in f^{-1}\{w\}$ such that
  \ref{item:proof-translate-stratified-basic-open-iii'}-\ref{item:proof-translate-stratified-basic-open-v'}
  hold, where $s'(z_{1}),\dots,s'(z_{M})$ are in increasing order --
  the latter is possible by~\ref{item:stratified-repr-ii}
  and~\ref{item:stratified-repr-iii}. Defining $s'$ on the equivalence
  class represented by $z_{h}$ to be the constant function with value
  $s'(z_{h})$, we obtain an increasing function
  $s'\colon s^{-1}\{w\}\to f^{-1}\{w\}$ such that
  \ref{item:proof-translate-stratified-basic-open-iii'}-\ref{item:proof-translate-stratified-basic-open-v'}
  hold.
\end{proof}
Now we can prove the first two reductions.
\begin{lemma}\label{lem:0123-squig-01cls23opn}
  It holds that $\TT_{0123}\rightsquigarrow\TT_{01^{cls}23^{opn}}$.
\end{lemma}
\begin{proof}
  Let $O\in\TT\subseteq\TT_{0123}$. We show that $O$ is a
  $\TT_{01^{cls}23^{opn}}$-neighbourhood of every element of $O$.

  Take $s\in O$ and, using Lemma~\ref{lem:existence-generic}, pick a
  generic surjection $f$. Since $\Img(s)\subseteq\Q=\Img(f)$, there
  exists $s'\in\MM_{\Q}$ such that $s=fs'=\lambda_{f}(s')$ by
  Lemma~\ref{lem:left-translation-preimage}\ref{item:left-translation-preimage-i}. Therefore,
  $s'\in\lambda_{f}^{-1}(O)$ where this set is $\TT$-open by
  continuity of $\lambda_{f}$. In particular,
  $\lambda_{f}^{-1}(O)\in\TT_{0123}$, so there exists
  \begin{equation}\label{eq:proof-0123-squig-01cls23opn}
    O'=O_{\bar{x},\bar{y}}^{(0)}\cap\bigcap_{j=1}^{m}O_{I_{j},J_{j}}^{(1)}\cap\bigcap_{k=1}^{\tilde{m}}O_{\tilde{I}_{k},\tilde{J}_{k}}^{(1)}\cap
    O_{LU}^{(2)}\cap\bigcap_{\ell=1}^{N}O_{K_{\ell}}^{(3)}
  \end{equation}
  such that $s'\in O'\subseteq\lambda_{f}^{-1}(O)$. We conclude
  $s=\lambda_{f}(s')\in\lambda_{f}(O')\subseteq O$. By
  Lemma~\ref{lem:stratified-repr}, we can assume the
  representation~\eqref{eq:proof-0123-squig-01cls23opn} to be
  stratified. Since $\Img(f)=\Q$,
  Lemma~\ref{lem:translate-stratified-basic-open} asserts that
  $\lambda_{f}(O')$ is a $\TT_{01^{cls}23^{opn}}$-basic open set, so
  $O$ is indeed a $\TT_{01^{cls}23^{opn}}$-neighbourhood of $s$.
\end{proof}
\begin{remark}\label{rem:0123-squig-01cls23opn}
  We can reformulate the proof of
  Lemma~\ref{lem:0123-squig-01cls23opn} as follows: We show that
  $(\MM_{\Q},\TT_{01^{cls}23^{opn}})$ has Property~\textbf{X} with
  respect to $(\MM_{\Q},\TT_{0123})$, using the decomposition
  $s=fs'\id_{\Q}$ for a fixed generic surjection $f$, the fixed map
  $\id_{\Q}$ and varying $s'$. Applying
  Proposition~\ref{prop:ppxb-auto-cont-lifting}\ref{item:ppxb-auto-cont-lifting-i}
  to the map $\id\colon (\MM_{\Q},\TT_{0123})\to (\MM_{\Q},\TT)$ --
  which is continuous since $\TT\subseteq\TT_{0123}$, note also that
  $(\MM_{\Q},\TT)$ is a topological semigroup -- yields the continuity
  of $\id\colon (\MM_{\Q},\TT_{01^{cls}23^{opn}})\to (\MM_{\Q},\TT)$,
  so $\TT\subseteq\TT_{01^{cls}23^{opn}}$.
\end{remark}
The second reduction is a slightly more involved application of
Lemma~\ref{lem:translate-stratified-basic-open}, picking both $f$ and
$s'$ in a more thoughtful way (tuned to the specific $s$ being
considered) by the following construction.
\begin{lemma}\label{lem:01cls23opn-squig-024-aux}
  Let $s,f\in\MM_{\Q}$ with
  $\Img(f)=(-\infty,\inf s)\cup\Img(s)\cup (\sup s,+\infty)$
  and such that the preimages $f^{-1}\{w\}$ are irrational intervals,
  i.e.~$f^{-1}\{w\}=(r_{w},t_{w})$ for all $w\in\Img(f)$, where
  $r_{w},t_{w}\in\I$. Then there exists $s'\in\MM_{\Q}$ such that
  $s=fs'$ and the following hold for all $p\in\Q$:
  \begin{enumerate}[label=(\roman*)]
  \item\label{item:01cls23opn-squig-024-aux-i} If
    $\sup s'(-\infty,p)<s'(p)$ then $\sup s'(-\infty,p)=r_{s(p)}$.
  \item\label{item:01cls23opn-squig-024-aux-ii} If
    $\inf s'(p,+\infty)>s'(p)$ then $\inf s'(p,+\infty)=t_{s(p)}$.
  \end{enumerate}
\end{lemma}
\begin{proof}
  Defining $s'$ as the union of order isomorphisms between
  $s^{-1}\{w\}$ and either $[z,z']$ or $(r_{w},z']$ or $[z,t_{w})$ or
  $(r_{w},t_{w})$ where $z$ and $z'$ are fixed elements of
  $f^{-1}\{w\}$ -- depending on the order type of $s^{-1}\{w\}$ -- we
  obtain a map with the following properties:
  \begin{enumerate}[label=(\alph*)]
  \item\label{item:proof-01cls23opn-squig-024-aux-a} $s=fs'$
  \item\label{item:proof-01cls23opn-squig-024-aux-b}
    $\forall w\in\Img(s)\colon\left(s^{-1}\{w\}\text{ has no greatest
        element }\Rightarrow \sup s'(s^{-1}\{w\})=t_{w}\right)$ \quad
    and\footnote{In other words: $s'$ exhausts $f^{-1}\{w\}$ whenever
      possible.}

    \noindent
    $\forall w\in\Img(s)\colon\left(s^{-1}\{w\}\text{ has no least
        element }\Rightarrow \inf s'(s^{-1}\{w\})=r_{w}\right)$
  \item\label{item:proof-01cls23opn-squig-024-aux-c}
    $\forall w\in\Img(s)\colon s'\vert_{s^{-1}\{w\}}$ is
    continuous\footnote{Note: If $s^{-1}\{w\}$ has e.g.~a greatest
      element, this does \emph{not} mean that $s'$ is continuous at
      that point but rather that $s'$ is left-continuous there.}
  \end{enumerate}
  We only show~\ref{item:01cls23opn-squig-024-aux-i}, the second
  assertion follows analogously. Assuming $\sup s'(-\infty,p)<s'(p)$,
  we distinguish two cases:

  \textit{Case~1} ($s(-\infty,p)$ has a greatest element): We set
  $w:=\max s(-\infty,p)$. Then there exists $p_{0}<p$ such that
  $s\vert_{(p_{0},p)}\equiv w$. Observe first that $s(p)>w$, i.e.~$p$
  is the supremum of $s^{-1}\{w\}$ but not a greatest element -- for
  otherwise $(p_{0},p]\subseteq s^{-1}\{w\}$,
  so~\ref{item:proof-01cls23opn-squig-024-aux-c} would yield
  $\sup
  s'(-\infty,p)=s'(p)$. By~\ref{item:proof-01cls23opn-squig-024-aux-b},
  we have $\sup s'(s^{-1}\{w\})=t_{w}$. Since
  $\sup s'(-\infty,p)=\sup s'(s^{-1}\{w\})$, it remains to show
  $t_{w}=r_{s(p)}$, equivalently $(w,s(p))\cap\Img(f)=\emptyset$. It
  suffices to note that $(w,s(p))\cap\Img(s)=\emptyset$ and that
  $\Img(f)\setminus\Img(s)$ and the convex hull of $\Img(s)$ are
  disjoint by choice of $\Img(f)$.

  \textit{Case~2} ($s(-\infty,p)$ does not have a greatest element):
  For each $p'<p$, there exists $p''$ such that $p'<p''<p$ and
  $s(p')<s(p'')\leq s(p)$. We have $s'(p'')\in f^{-1}\{s(p'')\}$ and
  thus
  \begin{displaymath}
    \sup s'(-\infty,p)\geq s'(p'')\geq\inf
    f^{-1}\{s(p'')\}=r_{s(p'')}\geq t_{s(p')}.
  \end{displaymath}
  Hence, $\sup s'(-\infty,p)\geq\sup_{p'<p}t_{s(p')}$. We claim that
  $\sup_{p'<p}t_{s(p')}\geq r_{s(p)}$. The opposite would yield
  $f(q)\in \big(\sup s(-\infty,p),s(p)\big)$ for any
  $q\in\big(\sup_{p'<p}t_{s(p')},r_{s(p)}\big)$. However,
  $\big(\sup s(-\infty,p),s(p)\big)\cap\Img(f)=\emptyset$ with the
  same reasoning as in Case~1.

  On the other hand, $s'(p')\leq t_{s(p')}<r_{s(p)}$ for each $p'<p$
  since $s(p')<s(p)$, so $\sup s'(-\infty,p)\leq r_{s(p)}$.
\end{proof}
For our reduction, we take into account the following two
observations: on the one hand, if
$s'\in O_{(-\infty,p),(-\infty,q]}^{(1)}$ and $s'(p)\leq q$, then
$O_{(-\infty,p),(-\infty,q]}^{(1)}$ can be replaced by
$O_{p,s'(p)}^{(0)}$; compare with~\ref{item:stratified-repr-v}. On the
other hand, if $r:=\sup s'(-\infty,p)\in\I$, then no set of the form
$O_{(-\infty,p),(-\infty,q]}^{(1)}$ with $q\in\Q$ (!) containing $s'$
can prohibit that $\sup \tilde{s}'(-\infty,p)>r$ for some
$\tilde{s}'\in O_{(-\infty,p),(-\infty,q]}^{(1)}$.
\begin{lemma}\label{lem:01cls23opn-squig-024}
  It holds that $\TT_{01^{cls}23^{opn}}\rightsquigarrow\TT_{024}$.
\end{lemma}
\begin{proof}
  Let $O\in\TT\subseteq\TT_{01^{cls}23^{opn}}$. We show that $O$ is a
  $\TT_{024}$-neighbourhood of every element of $O$.

  Take $s\in O$ and, using Lemma~\ref{lem:existence-generic}, pick
  $f\in\MM_{\Q}$ such that
  \begin{displaymath}
    \Img(f)=(-\infty,\inf s)\cup\Img(s)\cup (\sup s,+\infty)
  \end{displaymath}
  and all the preimages $f^{-1}\{w\}$ are irrational intervals,
  i.e.~$f^{-1}\{w\}=(r_{w},t_{w})$ for all $w\in\Img(f)$, where
  $r_{w},t_{w}\in\I$ (note that $r_{w}=-\infty$ or $t_{w}=+\infty$ is
  impossible since $f$ is unbounded-unbounded). By
  Lemma~\ref{lem:01cls23opn-squig-024-aux}, there exists
  $s'\in\MM_{\Q}$ satisfying $s=fs'$ and the following for all
  $p\in\Q$:
  \begin{enumerate}[label=(\roman*)]
  \item\label{item:proof-01cls23opn-squig-024-i} If
    $\sup s'(-\infty,p)<s'(p)$ then $\sup s'(-\infty,p)=r_{s(p)}$.
  \item\label{item:proof-01cls23opn-squig-024-ii} If
    $\inf s'(p,+\infty)>s'(p)$ then $\inf s'(p,+\infty)=t_{s(p)}$.
  \end{enumerate}
  Similarly to the proof of Lemma~\ref{lem:0123-squig-01cls23opn}, we
  use $s=fs'=\lambda_{f}(s')$, the $\TT$-continuity of $\lambda_{f}$
  and the assumption $\TT\subseteq\TT_{01^{cls}23^{opn}}$ to obtain a
  $\TT_{01^{cls}23^{opn}}$-basic open set
  \begin{equation}\label{eq:proof-01cls23opn-squig-024-iii}
    O'=O_{\bar{x},\bar{y}}^{(0)}\cap\bigcap_{j=1}^{m}O_{(-\infty,p_{j}),(-\infty,q_{j}]}^{(1)}\cap\bigcap_{k=1}^{\tilde{m}}O_{(\tilde{p}_{k},+\infty),[\tilde{q}_{k},+\infty)}^{(1)}\cap
    O_{LU}^{(2)}\cap\bigcap_{\ell=1}^{N}O_{(u_{\ell},v_{\ell})}^{(3)}
  \end{equation}
  such that $s\in\lambda_{f}(O')\subseteq O$. We additionally use
  Lemma~\ref{lem:stratified-repr} and assume that the
  representation~\eqref{eq:proof-01cls23opn-squig-024-iii} is
  stratified. If we have $s'(p_{j})\leq q_{j}$ for some
  $j\in\{1,\dots,m\}$, then
  $s'\in O_{p_{j},s'(p_{j})}^{(0)}\subseteq
  O_{(-\infty,p_{j}),(-\infty,q_{j}]}^{(1)}$ and we replace
  $O_{(-\infty,p_{j}),(-\infty,q_{j}]}^{(1)}$
  in~\eqref{eq:proof-01cls23opn-squig-024-iii} by
  $O_{p_{j},s'(p_{j})}^{(0)}$. We proceed analogously if
  $s'(\tilde{p}_{k})\geq\tilde{q}_{k}$. By rerunning the
  stratification procedure from Lemma~\ref{lem:stratified-repr}, we
  again obtain a strati\-fied representation. In our situation of
  \ref{item:stratified-repr-vii} and \ref{item:stratified-repr-viii}
  already holding, the proof of Lemma~\ref{lem:stratified-repr} never
  adds new sets of type~\ref{item:types-1}. Hence, we can assume that
  \begin{equation}\label{eq:proof-01cls23opn-squig-024-iv}
    s'(p_{j})>q_{j}\quad\text{for all
    }j\qquad\text{and}\qquad
    s'(\tilde{p}_{k})<\tilde{q}_{k}\quad\text{for all }k.
  \end{equation}
  Lemma~\ref{lem:translate-stratified-basic-open} yields
  \begin{multline}\label{eq:proof-01cls23opn-squig-024-v}
    \lambda_{f}(O')=\set{\tilde{s}}{\Img(\tilde{s})\subseteq\Img(f)}\cap
    O_{\bar{x},f(\bar{y})}^{(0)}\cap\\
    \bigcap_{j=1}^{m}O_{(-\infty,p_{j}),(-\infty,f(q_{j})]}^{(1)}\cap\bigcap_{k=1}^{\widetilde{m}}O_{(\tilde{p}_{k},+\infty),[f(\tilde{q}_{k}),+\infty)}^{(1)}\cap
    O_{LU}^{(2)}\cap\bigcap_{\ell=1}^{N}O_{(f(u_{\ell}),f(v_{\ell}))}^{(3)}.
  \end{multline}
  From~\eqref{eq:proof-01cls23opn-squig-024-iv} and $s'\in O'$, we
  obtain $\sup s'(-\infty,p_{j})\leq q_{j}<s'(p_{j})$ for all $j$ as
  well as
  $\inf s'(\tilde{p}_{k},+\infty)\geq \tilde{q}_{k}>s'(\tilde{p}_{k})$
  for all $k$. By \ref{item:proof-01cls23opn-squig-024-i} and
  \ref{item:proof-01cls23opn-squig-024-ii}, we conclude
  $\sup s'(-\infty,p_{j})=r_{s(p_{j})}$ for all $j$ and
  $\inf s'(\tilde{p}_{k},+\infty)=t_{s(\tilde{p}_{k})}$ for all
  $k$. Therefore, $q_{j}\geq r_{s(p_{j})}$ for all $j$ and
  $\tilde{q}_{k}\leq t_{s(\tilde{p}_{k})}$ for all $k$. Since the left
  hand sides of these inequalities are rational numbers while the
  right hand sides are irrational, we even obtain $q_{j}>r_{s(p_{j})}$
  for all $j$ and $\tilde{q}_{k}<t_{s(\tilde{p}_{k})}$ for all $k$. In
  other words, we have $f(q_{j})\geq s(p_{j})$ for all $j$ and
  $f(\tilde{q}_{k})\leq s(\tilde{p}_{k})$ for all $k$. Consequently,
  we can replace the sets of type~\ref{item:types-1}
  in~\eqref{eq:proof-01cls23opn-squig-024-v} by sets of
  type~\ref{item:types-0}, similarly to the above: we set
  \begin{displaymath}
    P:=\set{\tilde{s}}{\Img(\tilde{s})\subseteq\Img(f)}\cap
    O_{\bar{x},f(\bar{y})}^{(0)}\cap O_{\bar{p},s(\bar{p})}^{(0)}\cap O_{\overline{\widetilde{p}},s\left(\overline{\widetilde{p}}\right)}^{(0)}\cap
    O_{LU}^{(2)}\cap\bigcap_{\ell=1}^{N}O_{(f(u_{\ell}),f(v_{\ell}))}^{(3)}
  \end{displaymath}
  where
  $\bar{p}=(p_{1},\dots,p_{m}),\overline{\widetilde{p}}=(\widetilde{p}_{1},\dots,\widetilde{p}_{\widetilde{m}})$
  to obtain $s\in P\subseteq\lambda_{f}(O')\subseteq O$.

  Putting
  \begin{displaymath}
    A:=\Img(f)\cap\Q\setminus\bigcup_{\ell=1}^{N}(f(u_{\ell}),f(v_{\ell})),
  \end{displaymath}
  we see that
  \begin{displaymath}
    P=O_{\bar{x},f(\bar{y})}^{(0)}\cap O_{\bar{p},s(\bar{p})}^{(0)}\cap O_{\overline{\widetilde{p}},s\left(\overline{\widetilde{p}}\right)}^{(0)}\cap
    O_{LU}^{(2)}\cap O_{A}^{(4)}
  \end{displaymath}
  is a $\TT_{024}$-(basic) open set. Hence, $O$ is indeed a
  $\TT_{024}$-neighbourhood of $s$, as claimed.
\end{proof}
\begin{remark}\label{rem:01cls23opn-squig-024}
  We can reformulate the proof of Lemma~\ref{lem:01cls23opn-squig-024}
  as follows: We show that $(\MM_{\Q},\TT_{024})$ has
  Property~\textbf{X} with respect to
  $(\MM_{\Q},\TT_{01^{cls}23^{opn}})$, again using the decomposition
  $s=fs'\id_{\Q}$ -- this time for a fixed map $f$ with
  $\Img(f)=(-\infty,\inf s)\cup\Img(s)\cup (\sup s,+\infty)$ whose
  preimages of single points are irrational intervals, the fixed map
  $\id_{\Q}$ and varying $s'$. As in
  Remark~\ref{rem:0123-squig-01cls23opn}, we apply
  Proposition~\ref{prop:ppxb-auto-cont-lifting}\ref{item:ppxb-auto-cont-lifting-i}
  to the continous map
  $\id\colon (\MM_{\Q},\TT_{01^{cls}23^{opn}})\to (\MM_{\Q},\TT)$ to
  obtain $\TT\subseteq\TT_{024}$.
\end{remark}
\subsection{Reduction
  $\TT_{024}\rightsquigarrow\TT_{023^{opn}}$}\label{sec:024-squig-023opn}
For the next statement, we again aim at showing that a $\TT$-open set
is a $\TT_{023^{opn}}$-neighbourhood of its elements. Instead of
directly doing this for \emph{all} elements, we start by restricting
to \emph{injective} elements -- this special case contains the bulk of
the work. The main observation behind it is an analysis of
``products'' of the form $O_{\bar{z},\bar{z}}^{(0)}\circ O_{A}^{(4)}$
if $\bar{z}=(z_{1},\dots,z_{n})$ is a tuple in $\Q$ and $A$ is densely
ordered (which is connected to $g$ being injective). Clearly, if
$(z_{i},z_{j})\cap A=\emptyset$, then no element of
$O_{\bar{z},\bar{z}}^{(0)}\circ O_{A}^{(4)}$ can hit $(z_{i},z_{j})$
-- yielding a condition of type~\ref{item:types-3opn} instead
of~\ref{item:types-4}. As it turns out, this is the only obstruction
to points in the image.
\begin{lemma}\label{lem:024-squig-023opn-inj}
  Let $\TT$ be a semigroup topology on $\MM_{\Q}$ such that
  $\TT\subseteq\TT_{024}$. Then any injective endomorphism
  $g\in\MM_{\Q}$ has a neighbourhood basis consisting of
  $\TT_{023^{opn}}$-open sets.
\end{lemma}
\begin{proof}
  Given an injective endomorphism $g$, let $O$ be any $\TT$-open
  neighbourhood of $g$.

  We use continuity of the composition map. Since
  $g=\id_{\Q}\circ g\in O$, there exist $\TT$-neighbourhoods $V_{1}$
  of $\id_{\Q}$ and $V_{2}$ of $g$ such that
  $V_{1}\circ V_{2}\subseteq O$. By assumption,
  $\TT\subseteq\TT_{024}$, hence there exist $\TT_{024}$-basic open
  sets $U_{1},U_{2}$ such that $\id_{\Q}\in U_{1}\subseteq V_{1}$ and
  $g\in U_{2}\subseteq V_{2}$. Note that a $\TT_{024}$-basic open set
  containing $\id_{\Q}$ has the form
  $O_{\bar{z},\bar{z}}^{(0)}\cap O_{-\infty,+\infty}^{(2)}$ for a
  tuple $\bar{z}$ in $\Q$ -- sets of type~\ref{item:types-4} cannot
  occur. We can assume that $U_{2}$ has the form
  $U_{2}=O_{\bar{x},\bar{y}}^{(0)}\cap O_{LU}^{(2)}\cap O_{A}^{(4)}$,
  where $A$ is a densely ordered set (for otherwise, replace $A$ by
  $\Img(g)$).

  We obtain
  \begin{equation}\label{eq:proof-024-squig-023opn-inj-i}
    g\in \left(O_{\bar{z},\bar{z}}^{(0)}\cap O_{-\infty,+\infty}^{(2)}\right)\circ
    \left(O_{\bar{x},\bar{y}}^{(0)}\cap O_{LU}^{(2)}\cap
      O_{A}^{(4)}\right)\subseteq V_{1}\circ V_{2}\subseteq O.
  \end{equation}
  The lemma will be proved once we find a $\TT_{023^{opn}}$-open set
  $P$ with
  \begin{displaymath}
    g\in P\subseteq \left(O_{\bar{z},\bar{z}}^{(0)}\cap O_{-\infty,+\infty}^{(2)}\right)\circ
    \left(O_{\bar{x},\bar{y}}^{(0)}\cap O_{LU}^{(2)}\cap
      O_{A}^{(4)}\right).
  \end{displaymath}
  Since~\eqref{eq:proof-024-squig-023opn-inj-i} remains valid if we
  expand the tuple $\bar{z}$, we can assume that the elements listed
  in $\bar{y}$ are contained in $\bar{z}$.  We write
  $\bar{z}=(z_{1},\dots,z_{n})$ where the elements $z_{i}$ shall be
  sorted in ascending order. Adding additional elements $z_{\pm}$ to
  $\bar{z}$ if necessary, we can assume that $z_{1}=z_{-}<\inf A$ if
  $A$ is bounded below and that $z_{n}=z_{+}>\sup A$ if $A$ is bounded
  above. To simplify notation, we set $z_{0}:=-\infty$ as well as
  $z_{n+1}:=+\infty$. Further, we define
  \begin{align*}
    \MM_{0}&:=\set{(i,j)\in\{0,\dots,n+1\}^{2}}{i<j,\, (z_{i},z_{j})\cap A=\emptyset},\\
    \MM_{1}&:=\set{(i,j)\in\{0,\dots,n+1\}^{2}}{i<j,\,\abs{(z_{i},z_{j})\cap A}=1}.
  \end{align*}
  Note that $\set{(z_{i},z_{j})}{(i,j)\in\MM_{0}}$ always contains
  $(-\infty,z_{-})$ if $A$ is bounded below and $(z_{+},+\infty)$ if
  $A$ is bounded above. For $(i,j)\in\MM_{1}$, define $w_{i,j}$ such
  that $(z_{i},z_{j})\cap A=\{w_{i,j}\}$. Expanding the tuple
  $\bar{z}$ once more, we can assume that the elements $w_{i,j}$ are
  also contained in $\bar{z}$. Defining $\MM_{0}$ and $\MM_{1}$ from
  this expanded tuple, we obtain $\MM_{1}=\emptyset$. For each pair
  $(i,j)$, the set $(z_{i},z_{j})\cap A$ is thus either empty or it
  contains at least two elements -- in which case it contains an
  infinite densely ordered set. We claim that
  \begin{equation}\label{eq:proof-024-squig-023opn-inj-ii}
    g\in P:=O_{\bar{x},\bar{y}}^{(0)}\cap
    O_{LU}^{(2)}\cap\bigcap_{(i,j)\in\MM_{0}}O_{(z_{i},z_{j})}^{(3)}\subseteq \left(O_{\bar{z},\bar{z}}^{(0)}\cap O_{-\infty,+\infty}^{(2)}\right)\circ
    \left(O_{\bar{x},\bar{y}}^{(0)}\cap O_{LU}^{(2)}\cap
      O_{A}^{(4)}\right);
  \end{equation}
  note that
  $\Img(g)\cap (z_{i},z_{j})\subseteq A\cap (z_{i},z_{j})=\emptyset$
  for all $(i,j)\in\MM_{0}$.

  To prove the set inclusion
  in~\eqref{eq:proof-024-squig-023opn-inj-ii}, the crucial step is to
  find $f\in O_{\bar{z},\bar{z}}^{(0)}\cap O_{-\infty,+\infty}^{(2)}$
  such that
  \begin{equation}\label{eq:proof-024-squig-023opn-inj-iii}
    \forall q\in\Q\setminus\bigcup_{(i,j)\in\MM_{0}}(z_{i},z_{j})\colon
    f^{-1}\{q\}\cap A\neq\emptyset.
  \end{equation}
  This will be accomplished via a Back\&Forth strategy, distinguishing
  whether $A$ is bounded or unbounded above and below. In the
  following, we will consider the case that $A$ is bounded below and
  unbounded above; the other cases are treated analogously. We will
  first find an increasing map
  $\varphi\colon [z_{-},+\infty)\to [z_{-},+\infty)$ such
  that~\eqref{eq:proof-024-squig-023opn-inj-iii} holds with $\varphi$
  in place of $f$ (note that
  $\Q\setminus\bigcup_{(i,j)\in\MM_{0}}(z_{i},z_{j})=[z_{-},+\infty)\setminus\bigcup_{(i,j)\in\MM_{0}}(z_{i},z_{j})$). To
  this end, we consider the following property of a finite partial
  increasing map $m$ from $[z_{-},+\infty)$ to $[z_{-},+\infty)$:
  \begin{itemize}
  \item[$(+)$] For any
    $q\in \Q\setminus\bigcup_{(i,j)\in\MM_{0}}(z_{i},z_{j})$ and all
    $u,u'\in\Dom(m)$ with $u<u'$, if\\
    $m(u)<q<m(u')$, then $(u,u')\cap A$ is an infinite densely ordered
    set\footnote{equivalently: This intersection contains at least two
      elements.}.
  \end{itemize}
  Setting $C:=\Q\setminus\bigcup_{(i,j)\in\MM_{0}}(z_{i},z_{j})$, we
  claim that the system of all finite partial increasing maps $m$ from
  $[z_{-},+\infty)$ to $[z_{-},+\infty)$ satisfying~$(+)$ is an
  $(A,C)$-Back\&Forth system (see
  Definition~\ref{def:back-and-forth-variant}). In order to simplify
  notation, we formally add the elements $+\infty,-\infty$ to both
  $\Dom(m)$ and $\Img(m)$.

  $\mathbf{(A,C)}$\textbf{-Back:} Given $q\in C$, we set
  $u_{-}:=\max\set{u\in\Dom(m)}{m(u)<q}$ and further
  $u_{+}:=\min\set{u\in\Dom(m)}{q<m(u)}$; since we added $\pm\infty$
  to the domain and image of $m$, these elements are
  welldefined\footnote{Note: There might exist elements $u$ of
    $\Dom(m)$ in between $u_{-}$ and $u_{+}$, necessarily with
    $m(u)=q$. As can be seen from the proof, these elements can be
    neglected since they cannot prevent that an extension of $m$ with
    image $q$ is increasing.}. We claim that $(u_{-},u_{+})\cap A$ is
  an infinite densely ordered set. Note that $u_{-}\neq -\infty$ since
  $q\in C\subseteq [z_{-},+\infty)$. If $u_{+}$ is finite as well, our
  claim follows from condition~$(+)$, and if $u_{+}=+\infty$, it
  follows from $A$ being unbounded above.  Taking
  $p\in (u_{-},u_{+})\cap A$ such that both $(u_{-},p)\cap A$ and
  $(p,u_{+})\cap A$ are infinite densely ordered sets, we obtain that
  the extension $m'$ of $m$ by $m'(p):=q$ is an increasing map which
  still satisfies condition~$(+)$.

  \textbf{Forth:} Given $p\in\Q\setminus\Dom(m)$, we set
  $u_{-}:=\max\set{u\in\Dom(m)}{u<p}$ and\footnote{Here, there can
    never exist elements $u$ of $\Dom(m)$ in between $u_{-}$ and
    $u_{+}$!}  $u_{+}:=\min\set{u\in\Dom(m)}{p<u}$. We distinguish
  cases:

  \textit{Case~1} (both $(u_{-},p)\cap A$ and $(p,u_{+})\cap A$ are
  infinite densely ordered): Pick \emph{any} $q\in\Q$ with
  $m(u_{-})\leq q\leq m(u_{+})$.

  \textit{Case~2} ($(u_{-},p)\cap A$ is infinite densely ordered, but
  $(p,u_{+})\cap A$ is not): Pick $q:=m(u_{+})$.

  \textit{Case~3} ($(u_{-},p)\cap A$ is not infinite densely ordered,
  but $(p,u_{+})\cap A$ is): Pick $q:=m(u_{-})$.

  \textit{Case~4} (neither $(u_{-},p)\cap A$ nor $(p,u_{+})\cap A$ are
  infinite densely ordered): Pick \emph{any} $q\in\Q$ with
  $m(u_{-})\leq q\leq m(u_{+})$.

  \noindent Observe that the extension $m'$ of $m$ by $m'(p):=q$ is an
  increasing map which still satisfies condition~$(+)$; for Case~4, we
  see that $(u_{-},u_{+})\cap A$ is not infinite densely ordered --
  since~$(+)$ holds, $(u_{-},u_{+})\cap A$ is never considered as a
  set of the form $(u,u')\cap A$ in condition~$(+)$, even less so
  $(u_{-},p)\cap A$ and $(p,u_{+})\cap A$.

  \smallskip Since $\MM_{1}=\emptyset$, we know that $m$ defined by
  $\bar{z}\mapsto\bar{z}$ satisfies~$(+)$. Thus,
  Lemma~\ref{lem:back-and-forth-variant} yields
  $\varphi\colon [z_{-},+\infty)\to [z_{-},+\infty)$ with
  $\varphi(\bar{z})=\bar{z}$ and
  \begin{displaymath}
    \forall q\in C=\Q\setminus\bigcup_{(i,j)\in\MM_{0}}(z_{i},z_{j})\colon
    \varphi^{-1}\{q\}\cap A\neq\emptyset.
  \end{displaymath}
  Extending $\varphi$ to a total map $f$ by setting $f(q):=q$ for
  $q\in (-\infty,z_{-})$, this finishes the definition of $f$; by
  design, $f\in O_{\bar{z},\bar{z}}^{(0)}$. Since $C$ is unbounded
  above, $f$ must be as well. Moreover, $f$ is obviously unbounded
  below, yielding
  $f\in O_{\bar{z},\bar{z}}^{(0)}\cap O_{-\infty,+\infty}^{(2)}$ as
  desired.

  Using $f$, we can finally
  prove~\eqref{eq:proof-024-squig-023opn-inj-ii}. Let
  $s\in O_{\bar{x},\bar{y}}^{(0)}\cap
  O_{LU}^{(2)}\cap\bigcap_{(i,j)\in\MM_{0}}O_{(z_{i},z_{j})}^{(3)}$. We
  will prove
  $s\in \left(O_{\bar{z},\bar{z}}^{(0)}\cap
    O_{-\infty,+\infty}^{(2)}\right)\circ
  \left(O_{\bar{x},\bar{y}}^{(0)}\cap O_{LU}^{(2)}\cap
    O_{A}^{(4)}\right)$ by finding
  $h\in O_{\bar{x},\bar{y}}^{(0)}\cap O_{LU}^{(2)}\cap O_{A}^{(4)}$
  such that $s=fh$. The latter equality can be certainly satisfied by
  picking, for each
  \begin{displaymath}
    q\in\Img(s)\subseteq\Q\setminus\bigcup_{(i,j)\in\MM_{0}}(z_{i},z_{j})\subseteq\Img(f),
  \end{displaymath}
  \emph{any} element $p_{q}\in f^{-1}\{q\}\neq\emptyset$ and defining
  $h$ by $h(c):=p_{s(c)}$. Because
  of~\eqref{eq:proof-024-squig-023opn-inj-iii}, the elements $p_{q}$
  can be chosen in $A$, thus yielding $\Img(h)\subseteq A$. For the
  entries $y_{i}$ of $\bar{y}$, we can pick $p_{y_{i}}=y_{i}$ since
  $f(\bar{y})=\bar{y}$ -- note that $\bar{y}$ has been added to
  $\bar{z}$, that the entries are pairwise different since $g$ is
  injective and that $y_{i}\in A$ (by $g(\bar{x})=\bar{y}$). Thus,
  $s(\bar{x})=\bar{y}$ implies $h\in O_{\bar{x},\bar{y}}^{(0)}$. Since
  $f\in O_{-\infty,+\infty}^{(2)}$, the boundedness type of $h$ is the
  same as the boundedness type of $s$ which in turn is the same as the
  boundedness type of $g$. Hence, $h\in O_{LU}^{(2)}$ and we conclude
  $h\in O_{\bar{x},\bar{y}}^{(0)}\cap O_{LU}^{(2)}\cap O_{A}^{(4)}$.
  Therefore
  \begin{displaymath}
    s=fh\in \left(O_{\bar{z},\bar{z}}^{(0)}\cap
      O_{-\infty,+\infty}^{(2)}\right)\circ
    \left(O_{\bar{x},\bar{y}}^{(0)}\cap O_{LU}^{(2)}\cap
      O_{A}^{(4)}\right),
  \end{displaymath}
  thus proving~\eqref{eq:proof-024-squig-023opn-inj-ii} and,
  consequently, the lemma.
\end{proof}
\begin{lemma}\label{lem:024-squig-023opn}
  It holds that $\TT_{024}\rightsquigarrow\TT_{023^{opn}}$.
\end{lemma}
\begin{proof}
  Let $O\in\TT$. We show that $O$ is a
  $\TT_{023^{open}}$-neighbourhood of every element of $O$.

  Take $s\in O$. We claim that for any generic surjection
  $f\in\MM_{\Q}$ (which exists by Lemma~\ref{lem:existence-generic}),
  there is some injective $g\in\MM_{\Q}$ such that $s=fg$: Since
  $\Img(s)\subseteq\Q=\Img(f)$ and since the preimages $f^{-1}\{w\}$
  are irrational intervals,
  Lemma~\ref{lem:left-translation-preimage}\ref{item:left-translation-preimage-ii}
  applies and yields an injective $g\in\MM_{\Q}$ as desired.

  We use continuity of the translation map $\lambda_{f}$. Since
  $\lambda_{f}(g)=s\in O$, there exists a $\TT$-neighbourhood $V$ of
  $g$ such that $\lambda_{f}(V)\subseteq O$. By
  Lemma~\ref{lem:024-squig-023opn-inj}, there exists a
  $\TT_{023^{opn}}$-basic open set $U$ such that $g\in U\subseteq V$;
  we assume $U$ to be stratified via
  Lemma~\ref{lem:stratified-repr}. Hence,
  $s\in\lambda_{f}(U)\subseteq O$. Using
  Lemma~\ref{lem:translate-stratified-basic-open}, we obtain that
  $\lambda_{f}(U)$ is a $\TT_{023^{opn}}$-basic open set which proves
  the lemma.
\end{proof}
\begin{remark}\label{rem:024-squig-023opn}
  We can combine Lemmas~\ref{lem:024-squig-023opn-inj}
  and~\ref{lem:024-squig-023opn} and reformulate the proof of
  $\TT_{024}\rightsquigarrow\TT_{023^{opn}}$ as follows: We show that
  $(\MM_{\Q},\TT_{023^{opn}})$ has Property~$\mathbf{\overline{X}}$ of
  length~2 with respect to $(\MM_{\Q},\TT_{024})$, using the
  decomposition $s=f\id_{\Q}\id_{\Q}g\id_{\Q}$ where the first, third
  and fifth position are fixed and the second and fourth position are
  varying, subsequently yielding
  $\tilde{s}=f\tilde{f}\id_{\Q}\tilde{h}\id_{\Q}$. As in
  Remarks~\ref{rem:0123-squig-01cls23opn}
  and~\ref{rem:01cls23opn-squig-024}, we apply
  Proposition~\ref{prop:ppxb-auto-cont-lifting}\ref{item:ppxb-auto-cont-lifting-i}
  to the continous map
  $\id\colon (\MM_{\Q},\TT_{024})\to (\MM_{\Q},\TT)$ to obtain
  $\TT\subseteq\TT_{023^{opn}}$.
\end{remark}
\subsection{Reduction $\TT_{023^{opn}}\rightsquigarrow\TT_{03^{opn}}$}
\label{sec:023opn-squig-03opn}
In our next reduction, we eliminate the sets of
type~\ref{item:types-2}, i.e.~the boundedness types, from the upper
bound. Compared to our previous reductions, this requires a different
approach; we use the regularity of the given topology $\TT$ in a
crucial way. The main observation is the following: if $O$ is
$\TT$-open and $s\in O$, there exists a $\TT$-open set $P$ such that
$s\in P\subseteq\cl{P}^{\TT}\subseteq O$, where $\cl{P}^{\TT}$ denotes
the topological closure of $P$ with respect to $\TT$. Our proof
essentially amounts to showing that taking this topological closure
eliminates the sets $O_{LU}^{(2)}$ from $P$ -- this corresponds to
$O_{LU}^{(2)}$ being topologically dense. It is easy to see that
$O_{LU}^{(2)}$ is dense with respect to the pointwise topology;
however, this set is obviously not dense with respect to
$\TT_{023^{opn}}$. Hence, independently of the above sketch, it can
also be seen as an important step in showing $\TT=\TT_{pw}$ that
indeed $O_{LU}^{(2)}$ is dense with respect to $\TT$ as well. This
will depend on the Polishness of $\TT$. We start with a variant of
Lemma~\ref{lem:01cls23opn-squig-024-aux}.
\begin{lemma}\label{lem:023opn-squig-03opn-aux}
  Let $s,f\in\MM_{\Q}$ and $q\in\Q\setminus\Img(s)$ such that
  $\Img(f)=\Img(s)\cupdot\{q\}$ where the preimages $f^{-1}\{w\}$ are
  irrational intervals, i.e.~$f^{-1}\{w\}=(r_{w},t_{w})$ for all
  $w\in\Img(f)$, where $r_{w},t_{w}\in\I\cup\{\pm\infty\}$. Then the
  following hold:
  \begin{enumerate}[label=(\arabic*)]
  \item\label{item:023opn-squig-03opn-aux-i} Suppose there is $p\in\Q$
    such that $\sup s(-\infty,p)=\max s(-\infty,p)<q<s(p)$. \\Then
    there exists $s'\in\MM_{\Q}$ such that $s=fs'$ and
    $\sup s'(-\infty,p)=r_{q}\in\I$.
  \item\label{item:023opn-squig-03opn-aux-ii} Suppose that
    $\sup s=\max s<q$. \\Then there exists $s'\in\MM_{\Q}$ such that
    $s=fs'$ and $\sup s'=r_{q}\in\I$.
  \item\label{item:023opn-squig-03opn-aux-iii} Suppose that
    $q<\min s=\inf s$. \\Then there exists $s'\in\MM_{\Q}$ such that
    $s=fs'$ and $\inf s'=t_{q}\in\I$.
  \end{enumerate}
\end{lemma}
\begin{proof}[Proof (of Lemma~\ref{lem:023opn-squig-03opn-aux}).]
  One picks $s'\in\MM_{\Q}$ with
  \begin{enumerate}[label=(\alph*)]
  \item\label{item:proof-023opn-squig-03opn-aux-a} $s=fs'$
  \item\label{item:proof-023opn-squig-03opn-aux-b}
    $\forall w\in\Img(s)\colon\left(s^{-1}\{w\}\text{ has no greatest
        element }\Rightarrow \sup s'(s^{-1}\{w\})=t_{w}\right)$ \quad
    and

    \noindent
    $\forall w\in\Img(s)\colon\left(s^{-1}\{w\}\text{ has no least
        element }\Rightarrow \inf s'(s^{-1}\{w\})=r_{w}\right)$
  \item\label{item:proof-023opn-squig-03opn-aux-c}
    $\forall w\in\Img(s)\colon s'\vert_{s^{-1}\{w\}}$ is continuous
  \end{enumerate}
  and argues as in~Case~1 of the proof of
  Lemma~\ref{lem:01cls23opn-squig-024-aux} with $q$ in place of
  $s(p)$. The boundary points $r_{q}$
  (for~\ref{item:023opn-squig-03opn-aux-i},\ref{item:023opn-squig-03opn-aux-ii})
  and $t_{q}$ (for~\ref{item:023opn-squig-03opn-aux-iii}) are finite
  since there exist elements in $\Img(f)$ which are below $q$ and
  above $q$, respectively.
\end{proof}
Next, we show that the set $\Surj(\Q)$ of all surjective elements of
$\MM_{\Q}$ is dense with respect to our given Polish semigroup
topology $\TT$ with $\TT_{pw}\subseteq\TT\subseteq\TT_{023^{opn}}$ --
this uses Polishness in an essential way and is another step in
matching $\TT$ to $\TT_{pw}$.
\begin{lemma}\label{lem:023opn-density-surj}
  Let $\TT$ be a Polish semigroup topology on $\MM_{\Q}$ such that
  $\TT_{pw}\subseteq\TT\subseteq\TT_{023^{opn}}$.

  \begin{enumerate}[label=(\roman*)]
  \item\label{item:023opn-density-surj-i} For each $q\in\Q$, the set
    $M_{q}:=\set{s\in\MM_{\Q}}{q\in\Img(s)}$ is $\TT$-dense.
  \item\label{item:023opn-density-surj-ii} The set $\Surj(\Q)$ of
    surjective endomorphisms on $\Q$ is $\TT$-dense.
  \end{enumerate}
\end{lemma}
\begin{proof}
  $ $\par\nobreak\ignorespaces\textbf{(i).} Let $O\in\TT$ be open and
  nonempty; we have to show $O\cap M_{q}\neq\emptyset$. Since
  $\TT\subseteq\TT_{023^{opn}}$, the set $O$ contains a nonempty
  $\TT_{023^{opn}}$-basic open set; we write
  \begin{displaymath}
    \emptyset\neq O_{\bar{x},\bar{y}}^{(0)}\cap
    O_{LU}^{(2)}\cap\bigcap_{\ell=1}^{N}
    O_{(u_{\ell},v_{\ell})}^{(3)}\subseteq O
  \end{displaymath}
  which we assume to be a stratified representation, see
  Lemma~\ref{lem:stratified-repr}. If $q$ is contained in $\bar{y}$,
  then \emph{any} $s\in O$ has $q$ in its image, so we assume the
  contrary. Distinguishing by the position of $q$ relative to
  $\bar{y}=(y_{1},\dots,y_{n})$ and $(u_{1},v_{1},\dots,u_{N},v_{N})$
  and by the required boundedness type $O_{LU}^{(2)}$, one easily
  constructs (by a piecewise definition) a map $s\in O$ possibly
  together with a rational $p\in\Q$ such that
  $\sup s(-\infty,p)=\max s(-\infty,p)<q<s(p)$ (if
  $\Q\setminus\bigcup_{\ell=1}^{N}(u_{\ell},v_{\ell})$ contains
  elements less and elements greater than $q$) or $\sup s=\max s<q$
  (if $\Q\setminus\bigcup_{\ell=1}^{N}(u_{\ell},v_{\ell})$ contains
  only elements less than $q$) or $q<\min s=\inf s$ (if
  $\Q\setminus\bigcup_{\ell=1}^{N}(u_{\ell},v_{\ell})$ contains only
  elements greater than $q$). We use Lemma~\ref{lem:existence-generic}
  to find $f\in\MM_{\Q}$ with $\Img(f)=\Img(s)\cupdot\{q\}$
  and~$f^{-1}\{w\}=(r_{w},t_{w})$ for all $w\in\Img(f)$, where
  $r_{w},t_{w}\in\I\cup\{\pm\infty\}$. By
  Lemma~\ref{lem:023opn-squig-03opn-aux}, there exists $s'\in\MM_{\Q}$
  such that $s=fs'$ and $\sup s'(-\infty,p)=r_{q}\in\I$ or
  $\sup s'=r_{q}\in\I$ or $\inf s'=t_{q}\in\I$. Applying continuity of
  the translation map $\lambda_{f}$ at $s'$ as well as
  $\TT\subseteq\TT_{023^{opn}}$, we obtain a $\TT_{023^{opn}}$-basic
  open set
  \begin{displaymath}
    O'=O_{\bar{x}',\bar{y}'}^{(0)}\cap O_{L'U'}^{(2)}\cap\bigcap_{\ell=1}^{N'}O_{(u'_{\ell},v'_{\ell})}^{(3)}
  \end{displaymath}
  such that $s'\in O'$ and $s\in\lambda_{f}(O')\subseteq O$. In
  particular,
  $\Img(s')\subseteq\Q\setminus\bigcup_{\ell=1}^{N}(u'_{\ell},v'_{\ell})=:A'$,
  so either $r_{q}$ or $t_{q}$ is a limit point of $A'$. Since $r_{q}$
  and $t_{q}$ are irrational while the boundary points of $A'$ are
  rational, either $r_{q}$ or $t_{q}$ must in fact be contained in the
  interior of $A'$. Thus,
  $A'\cap f^{-1}\{q\}=A'\cap (r_{q},t_{q})\neq\emptyset$; we pick $z'$
  in this intersection.

  Similarly to our construction of $s$, we distinguish by the
  positition of $z'$ relative to $\bar{y}'=(y'_{1},\dots,y'_{n'})$ and
  $(u'_{1},v'_{1},\dots,u'_{N'},v'_{N'})$ and by the required
  boundedness type $O_{L'U'}^{(2)}$ to find a map $\tilde{s}'\in O'$
  with $z'\in\Img(\tilde{s}')$. We obtain
  $\tilde{s}:=\lambda_{f}(\tilde{s}')=f\tilde{s}'\in O$ and
  $q\in\Img(\tilde{s})$, i.e.~$\tilde{s}\in O\cap M_{q}\neq\emptyset$.

  \textbf{(ii).} For each $q\in\Q$, the set
  $M_{q}=\set{s\in\MM_{\Q}}{q\in\Img(s)}$ is $\TT$-open since
  $\TT_{pw}\subseteq\TT$. By~\ref{item:023opn-density-surj-i}, it is
  also $\TT$-dense. Since $\TT$ is a Polish topology, Baire's Category
  Theorem applies and yields the $\TT$-density of
  $\Surj(\Q)=\bigcap_{q\in\Q}M_{q}$.
\end{proof}
By definition, any $\TT_{023^{opn}}$-open set can be represented as a
union of sets of the form
\begin{displaymath}
  O_{\bar{x},\bar{y}}^{(0)}\cap O_{LU}^{(2)}\cap\bigcap_{\ell=1}^{N}
  O_{(u_{\ell},v_{\ell})}^{(3)}.
\end{displaymath}
If we rearrange to separate the $\TT_{02}$-interior from the
``proper'' type~\ref{item:types-3opn} portion, we obtain the following
alternative notation which will prove to be very helpful:
\begin{notation}\label{not:023opn-alternative}
  Setting
  \begin{displaymath}
    A:=O_{-\infty,+\infty}^{(2)}\quad \quad B:=O_{-\infty,\R}^{(2)}\quad \quad
    C:=O_{\R,+\infty}^{(2)}\quad \quad D:=O_{\R,\R}^{(2)},
  \end{displaymath}
  we can rewrite any $\TT_{023^{opn}}$-open set $O$ as
  \begin{displaymath}
    O=(O_{A}\cap A)\cup (O_{B}\cap B)\cup (O_{C}\cap C)\cup
    (O_{D}\cap D)\cup
    \bigcup_{i\in
      I}\left(O_{\bar{x}^{(i)},\bar{y}^{(i)}}^{(0)}\cap
      O_{L^{(i)},U^{(i)}}^{(2)}\cap\bigcap_{\ell=1}^{N^{(i)}}O_{(u_{\ell}^{(i)},v_{\ell}^{(i)})}^{(3)}\right)
  \end{displaymath}
  where $O_{A},O_{B},O_{C},O_{D}\in\TT_{pw}$,
  $\bar{x}^{(i)},\bar{y}^{(i)}$ are tuples in $\Q$, $N^{(i)}\geq 1$
  and $u_{\ell}^{(i)},v_{\ell}^{(i)}\in\Q$.
\end{notation}
Note that the sets $O_{A},O_{B},O_{C},O_{D}$ could in general be empty
even if $O$ is nonempty. However, for
$O\in\TT\subseteq\TT_{023^{opn}}$, one uses the previous lemma to
prove:
\begin{lemma}\label{lem:023opn-open-set-A-part-nonempty}
  Let $\TT$ be a Polish semigroup topology on $\MM_{\Q}$ such that
  $\TT_{pw}\subseteq\TT\subseteq\TT_{023^{opn}}$, and let $O\in\TT$ be
  nonempty. Then $O\subseteq \cl{O_{A}}^{\TT}$. In particular,
  $O_{A}\neq\emptyset$.
\end{lemma}
\begin{proof}
  Aiming for a contradiction, we assume
  $O\nsubseteq \cl{O_{A}}^{\TT}$. Thus, denoting the complement of
  $\cl{O_{A}}^{\TT}$ by $\compl{\left(\cl{O_{A}}^{\TT}\right)}$, we
  know that $O\cap \compl{\left(\cl{O_{A}}^{\TT}\right)}$ is a
  nonempty $\TT$-open set. However,
  $O_{(u_{\ell}^{(i)},v_{\ell}^{(i)})}^{(3)}\cap\Surj(\Q)=\emptyset$
  and $(B\cup C\cup D)\cap\Surj(\Q)=\emptyset$ imply
  \begin{displaymath}
    O\cap
    \compl{\left(\cl{O_{A}}^{\TT}\right)}\cap\Surj(\Q)=O_{A}\cap
    A\cap\compl{\left(\cl{O_{A}}^{\TT}\right)}\cap\Surj(\Q)\subseteq O_{A}\cap\compl{\left(\cl{O_{A}}^{\TT}\right)}=\emptyset,
  \end{displaymath}
  which contradicts
  Lemma~\ref{lem:023opn-density-surj}\ref{item:023opn-density-surj-ii}.
\end{proof}
With this result, we can attain an important intermediate step already
hinted at in our proof outline in the introductory remarks to
Subsection~\ref{sec:023opn-squig-03opn}.
\begin{lemma}\label{lem:023opn-open-set-nonempty-pw-interior}
  Let $\TT$ be a Polish semigroup topology on $\MM_{\Q}$ such that
  $\TT_{pw}\subseteq\TT\subseteq\TT_{023^{opn}}$. Then any nonempty
  $O\in\TT$ has nonempty $\TT_{pw}$-interior. Consequently, a subset of $\MM_{\Q}$ is
  $\TT$-dense if and only if it is $\TT_{pw}$-dense; in particular,
  every boundedness type $O_{LU}^{(2)}$ is $\TT$-dense.
\end{lemma}
\begin{proof}
  By regularity, there exists a nonempty $P\in\TT$ such that
  $\cl{P}^{\TT}\subseteq O$ and thus
  $\cl{P_{A}\cap A}^{\TT}\subseteq O$. Since $A\supseteq\Surj(\Q)$ is
  $\TT$-dense by
  Lemma~\ref{lem:023opn-density-surj}\ref{item:023opn-density-surj-ii}
  and $P_{A}$ is $\TT$-open, we obtain
  $\cl{P_{A}\cap A}^{\TT}=\cl{P_{A}}^{\TT}$ from elementary
  topology. Therefore,
  $P_{A}\subseteq\cl{P_{A}\cap A}^{\TT}\subseteq O$ and the
  $\TT_{pw}$-interior of $O$ contains the set $P_{A}$ which is
  nonempty by Lemma~\ref{lem:023opn-open-set-A-part-nonempty}.

  That any $\TT$-dense set is $\TT_{pw}$-dense follows from
  $\TT_{pw}\subseteq\TT$. For the converse, assume that $M$ is
  $\TT_{pw}$-dense and let $O$ be nonempty and $\TT$-open. Since the
  $\TT_{pw}$-interior of $O$ is nonempty, it has nonempty intersection
  with $M$, in particular $M\cap O\neq\emptyset$.
\end{proof}
Next, we use the previous results to show that taking the topological
closure with respect to $\TT$ eliminates the boundedness types from
open sets. This is the crucial technical step in the proof of our
reduction.
\begin{lemma}\label{lem:023opn-closure}
  Let $\TT$ be a Polish semigroup topology on $\MM_{\Q}$ such that
  $\TT_{pw}\subseteq\TT\subseteq\TT_{023^{opn}}$. Let further
  $O_{\bar{x},\bar{y}}^{(0)}\cap
  O_{LU}^{(2)}\cap\bigcap_{\ell=1}^{N}O_{(u_{\ell},v_{\ell})}^{(3)}\neq\emptyset$
  be a nonempty $\TT_{023^{opn}}$-basic open set. Then
  \begin{displaymath}
    \cl{O_{\bar{x},\bar{y}}^{(0)}\cap
      O_{LU}^{(2)}\cap\bigcap_{\ell=1}^{N}O_{(u_{\ell},v_{\ell})}^{(3)}}^{\TT}=O_{\bar{x},\bar{y}}^{(0)}\cap\bigcap_{\ell=1}^{N}O_{(u_{\ell},v_{\ell})}^{(3)}. 
  \end{displaymath}
\end{lemma}
\begin{proof}
  The inclusion ``$\subseteq$'' follows from
  $O_{\bar{x},\bar{y}}^{(0)}\cap\bigcap_{\ell=1}^{N}O_{(u_{\ell},v_{\ell})}^{(3)}$
  being $\TT_{pw}$-closed, in particular $\TT$-closed.

  For the other inclusion ``$\supseteq$'', take
  $s\in
  O_{\bar{x},\bar{y}}^{(0)}\cap\bigcap_{\ell=1}^{N}O_{(u_{\ell},v_{\ell})}^{(3)}$
  and consider a $\TT$-open set $O$ containing $s$. We have to show
  $O_{\bar{x},\bar{y}}^{(0)}\cap
  O_{LU}^{(2)}\cap\bigcap_{\ell=1}^{N}O_{(u_{\ell},v_{\ell})}^{(3)}\cap
  O\neq\emptyset$. Pick $f\in\MM_{\Q}$ such that
  $\Img(f)=\Q\setminus\left(\bigcup_{\ell=1}^{N}(u_{\ell},v_{\ell})\right)\supseteq\Img(s)$. By
  continuity of the translation map $\lambda_{f}$, the preimage
  $\lambda_{f}^{-1}(O)$ is $\TT$-open.
  Lemma~\ref{lem:left-translation-preimage}\ref{item:left-translation-preimage-i}
  yields a map $s'\in\MM_{\Q}$ such that $s=fs'$. We conclude from
  $\TT_{pw}\subseteq\TT$ that the intersection
  $\emptyset\neq\lambda_{f}^{-1}(O)\cap
  O_{\bar{x},s'(\bar{x})}^{(0)}\ni s'$ is $\TT$-open. By
  Lemma~\ref{lem:023opn-open-set-nonempty-pw-interior}, the
  boundedness type $O_{LU}^{(2)}$ is $\TT$-dense, therefore there
  exists
  $\tilde{s}'\in\lambda_{f}^{-1}(O)\cap
  O_{\bar{x},s'(\bar{x})}^{(0)}\cap O_{LU}^{(2)}$. We define
  $\tilde{s}:=f\tilde{s}'=\lambda_{f}(\tilde{s}')$ and claim
  \begin{align*}
    \tilde{s}\in O_{\bar{x},\bar{y}}^{(0)}\cap
    O_{LU}^{(2)}\cap\bigcap_{\ell=1}^{N}O_{(u_{\ell},v_{\ell})}^{(3)}\cap
    O
  \end{align*}
  which will complete the proof. We only argue
  $\tilde{s}\in O_{LU}^{(2)}$, the rest is straightforward. If
  $-\infty$ occurs among the $u_{\ell}$, then $L=\R$ since
  $O_{\bar{x},\bar{y}}^{(0)}\cap
  O_{LU}^{(2)}\cap\bigcap_{\ell=1}^{N}O_{(u_{\ell},v_{\ell})}^{(3)}\neq\emptyset$. Further,
  $\tilde{s}$ is bounded below since
  $\tilde{s}\in\bigcap_{\ell=1}^{N}O_{(u_{\ell},v_{\ell})}^{(3)}$. If
  on the other hand $-\infty$ is not contained among the $u_{\ell}$,
  then $f$ is unbounded below, so $\tilde{s}$ is unbounded below if
  and only if $\tilde{s}'$ is unbounded below which occurs if and only
  if $L=\{-\infty\}$. Arguing analogously for upper bounds, we
  conclude $\tilde{s}\in O_{LU}^{(2)}$.
\end{proof}
Lemmas~\ref{lem:023opn-open-set-nonempty-pw-interior}
and~\ref{lem:023opn-closure} finally enable us to show our reduction:
\begin{lemma}\label{lem:023opn-squig-03opn}
  It holds that $\TT_{023^{opn}}\rightsquigarrow\TT_{03^{opn}}$.
\end{lemma}
\begin{proof}
  Let $O\in\TT$. We show that $O$ is a $\TT_{03^{opn}}$-neighbourhood
  of every element of $O$.

  Take $s\in O$. By regularity, there exists $P\in\TT$ such that
  $s\in P\subseteq\cl{P}^{\TT}\subseteq O$. Since
  $\TT\subseteq\TT_{023^{opn}}$, there exists a
  $\TT_{023^{opn}}$-basic open set
  \begin{displaymath}
    U=O_{\bar{x},\bar{y}}^{(0)}\cap
    O_{LU}^{(2)}\cap\bigcap_{\ell=1}^{N}O_{(u_{\ell},v_{\ell})}^{(3)}
  \end{displaymath}
  such that $s\in U\subseteq P$, in particular
  $s\in\cl{U}^{\TT}\subseteq O$. By Lemma~\ref{lem:023opn-closure},
  the $\TT$-closure of $U$ is
  $O_{\bar{x},\bar{y}}^{(0)}\cap\bigcap_{\ell=1}^{N}O_{(u_{\ell}^{(i)},v_{\ell}^{(i)})}^{(3)}$. Hence,
  $O$ is indeed a $\TT_{03^{opn}}$-neighbourhood of $s$.
\end{proof}
\begin{remark}\label{rem:023opn-squig-03opn}
  As already stated in the concluding remarks of
  Section~\ref{sec:finest-topol-strat}, the reduction
  $\TT_{023^{opn}}\rightsquigarrow\TT_{03^{opn}}$ is the only one
  whose proof cannot be reformulated as a
  (Pseudo\nobreakdash-)Property~$\mathbf{\overline{X}}$-type
  statement. Starting from Proposition~\ref{prop:rich-top-ppxb} and
  applying Proposition~\ref{prop:ppxb-auto-cont-lifting} along the
  route
  $\TT_{rich}=\TT_{0123}\rightsquigarrow\TT_{01^{cls}23^{opn}}\rightsquigarrow\TT_{024}\rightsquigarrow\TT_{023^{opn}}$,
  we obtain that $(\MM_{\Q},\TT_{023^{opn}})$ has automatic continuity
  with respect to the class of second countable topological
  semigroups. However, we cannot continue on to
  $\TT_{03^{opn}}$. Thus, the reduction
  $\TT_{023^{opn}}\rightsquigarrow\TT_{03^{opn}}$ is indeed
  fundamentally different.
\end{remark}
\subsection{Reduction
  $\TT_{03^{opn}}\rightsquigarrow\TT_{0}=\TT_{pw}$}
\label{sec:03opn-squig-pw}
In our final reduction, we eliminate the sets of
type~\ref{item:types-3opn}. The technique resembles those of
Subsections~\ref{sec:0123-squig-01cls23opn-squig-024}
and~\ref{sec:024-squig-023opn}, albeit with crucial involvement of the
$\TT$-density of $\Surj(\Q)$ shown in
Subsection~\ref{sec:023opn-squig-03opn}. The following easy
observation gives an idea as to why this is important.
\begin{lemma}\label{lem:03opn-squiq-pw-aux}
  Let $\TT$ be a Polish semigroup topology on $\MM_{\Q}$ such that
  $\TT\subseteq\TT_{03^{opn}}$. Let further $O\in\TT$ and let $f\in O$
  be surjective. Then $O$ is a $\TT_{pw}$-neighbourhood of $f$, in
  other words, there exists $P\in\TT_{pw}$ such that
  $f\in P\subseteq O$.
\end{lemma}
\begin{proof}
  With the same spirit as in Notation~\ref{not:023opn-alternative}, we
  can write
  \begin{displaymath}
    O=O_{pw}\cup\bigcup_{i\in I}\left(O_{\bar{x}^{(i)},\bar{y}^{(i)}}^{(0)}\cap\bigcap_{\ell=1}^{N^{(i)}}O_{(u_{\ell}^{(i)},v_{\ell}^{(i)})}^{(3)}\right),
  \end{displaymath}
  where $O_{pw}\in\TT_{pw}$, $\bar{x}^{(i)},\bar{y}^{(i)}$ are tuples
  in $\Q$, $N^{(i)}\geq 1$ and
  $u_{\ell}^{(i)},v_{\ell}^{(i)}\in\Q$. Since none of the sets
  $O_{\bar{x}^{(i)},\bar{y}^{(i)}}^{(0)}\cap\bigcap_{\ell=1}^{N^{(i)}}O_{(u_{\ell}^{(i)},v_{\ell}^{(i)})}^{(3)}$
  can contain surjective functions, $f$ has to be contained in
  $P:=O_{pw}$.
\end{proof}

\begin{lemma}\label{lem:03opn-squig-pw}
  It holds that $\TT_{03^{opn}}\rightsquigarrow\TT_{0}=\TT_{pw}$.
\end{lemma}
\begin{proof}
  Let $O\in\TT$. We show that $O$ is a $\TT_{pw}$-neighbourhood of
  every element of $O$.

  Take $s\in O$. By $\TT$-continuity of the composition map $\circ$
  and since $s\circ\id_{\Q}\in O$, there exist $\TT$-open sets $U$ and
  $V$ such that $s\in U$, $\id_{\Q}\in V$ and $U\circ V\subseteq
  O$. Using Lemma~\ref{lem:03opn-squiq-pw-aux}, we can shrink $V$ and
  assume that $V$ is $\TT_{pw}$-open; shrinking further we can even
  take $V$ to be $\TT_{pw}$-basic open, so
  $V=O_{\bar{x},\bar{x}}^{(0)}$. The set
  $U\cap O_{\bar{x},s(\bar{x})}^{(0)}$ is a nonempty $\TT$-open
  set. By
  Lemma~\ref{lem:023opn-density-surj}\ref{item:023opn-density-surj-ii},
  the surjective functions form a $\TT$-dense set, so there exists
  $f\in U\cap O_{\bar{x},s(\bar{x})}^{(0)}\cap\Surj(\Q)$.

  We claim that
  $f\circ O_{\bar{x},\bar{x}}^{(0)}=O_{\bar{x},f(\bar{x})}^{(0)}\
  (=O_{\bar{x},s(\bar{x})}^{(0)})$. The inclusion ``$\subseteq$'' is
  clear; for the converse inclusion ``$\supseteq$'', we argue as
  follows: given $\tilde{s}\in O_{\bar{x},f(\bar{x})}^{(0)}$, the
  finite partial map $m$ defined by $\bar{x}\mapsto\bar{x}$ satisfies
  $\tilde{s}(p)=fm(p)$ for all $p\in\Dom(m)$. Since $f$ is surjective,
  we can apply
  Lemma~\ref{lem:left-translation-preimage}\ref{item:left-translation-preimage-i}
  to find $\tilde{s}'\in\MM_{\Q}$ such that
  $\tilde{s}'(\bar{x})=\bar{x}$ and $\tilde{s}=f\tilde{s}'$, thus
  proving the claim.

  We obtain
  \begin{displaymath}
    s\in O_{\bar{x},s(\bar{x})}^{(0)}=f\circ O_{\bar{x},\bar{x}}^{(0)}\subseteq U\circ
    V\subseteq O,
  \end{displaymath}
  showing that $O$ is indeed a $\TT_{pw}$-neighbourhood of $s$, as
  desired.
\end{proof}
\begin{remark}\label{rem:03opn-squig-pw}
  We can reformulate the proof of Lemma~\ref{lem:03opn-squig-pw} as
  follows: We show that $(\MM_{\Q},\TT_{pw})$ has
  Property~$\mathbf{\overline{X}}$ of length~2 with respect to
  $(\MM_{\Q},\TT_{03^{opn}})$, using the decomposition
  $s=\id_{\Q}s\id_{\Q}\id_{\Q}\id_{\Q}$ where the first, third and
  fifth position are fixed and the second and fourth position are
  varying, subsequently yielding
  $\tilde{s}=\id_{\Q}f\id_{\Q}\tilde{s}'\id_{\Q}$. As in
  Remarks~\ref{rem:0123-squig-01cls23opn},~\ref{rem:01cls23opn-squig-024}
  and~\ref{rem:024-squig-023opn}, we apply
  Proposition~\ref{prop:ppxb-auto-cont-lifting}\ref{item:ppxb-auto-cont-lifting-i}
  to the continous map
  $\id\colon (\MM_{\Q},\TT_{03^{opn}})\to (\MM_{\Q},\TT)$ to obtain
  $\TT\subseteq\TT_{pw}$. Observe that the existence of $f$ requires
  the density statements from the previous reduction (which were shown
  using Polishness).
\end{remark}

\bibliographystyle{alpha}
\bibliography{global.bib}

\newcommand{\etalchar}[1]{$^{#1}$}
\def\cprime{$'$} \def\cprime{$'$} \def\cprime{$'$}
\begin{thebibliography}{BTVG17}

\bibitem[BEKP18]{BodirskyEvansKompatscherPinsker}
Manuel Bodirsky, David Evans, Michael Kompatscher, and Michael Pinsker.
\newblock A counterexample to the reconstruction of $\omega$-categorical
  structures from their endomorphism monoids.
\newblock {\em Israel Journal of Mathematics}, 224(1):57--82, 2018.

\bibitem[BM07]{MacphersonBarbina}
Silvia Barbina and Dugald Macpherson.
\newblock Reconstruction of homogeneous relational structures.
\newblock {\em Journal of Symbolic Logic}, 72(3):792--802, 2007.

\bibitem[BP16]{BartoPinskerDichotomy}
Libor Barto and Michael Pinsker.
\newblock The algebraic dichotomy conjecture for infinite domain constraint
  satisfaction problems.
\newblock In {\em Proceedings of the 31th {A}nnual {IEEE} {S}ymposium on
  {L}ogic in {C}omputer {S}cience -- {LICS}'16}, pages 615--622, 2016.

\bibitem[BP20]{Topo}
Libor Barto and Michael Pinsker.
\newblock Topology is irrelevant.
\newblock {\em SIAM Journal on Computing}, 49(2):365--393, 2020.

\bibitem[BPP17]{Reconstruction}
Manuel Bodirsky, Michael Pinsker, and Andr\'{a}s Pongr\'acz.
\newblock Reconstructing the topology of clones.
\newblock {\em Transactions of the American Mathematical Society},
  369:3707--3740, 2017.

\bibitem[BPP21]{BPP-projective-homomorphisms}
Manuel Bodirsky, Michael Pinsker, and Andr\'{a}s Pongr\'acz.
\newblock Projective clone homomorphisms.
\newblock {\em Journal of Symbolic Logic}, 86(1):148--161, 2021.

\bibitem[BTVG17]{behrisch-truss-vargas}
Mike Behrisch, John~K. Truss, and Edith Vargas-Garc\'{i}a.
\newblock Reconstructing the topology on monoids and polymorphism clones of the
  rationals.
\newblock {\em Studia Logica}, 105(1):65--91, 2017.

\bibitem[BVG21]{behrisch-vargasgarcia-stronger}
Mike Behrisch and Edith Vargas-Garc\'{\i}a.
\newblock On a stronger reconstruction notion for monoids and clones.
\newblock {\em Forum Mathematicum}, 33(6):1487--1506, 2021.

\bibitem[EH90]{EvansHewitt}
David~M. Evans and Paul~R. Hewitt.
\newblock Counterexamples to a conjecture on relative categoricity.
\newblock {\em Annals of Pure and Applied Logic}, 46(2):201--209, 1990.

\bibitem[EJM{\etalchar{+}}]{EJMMMP-zariski}
Luke Elliott, Julius Jonu\v{s}as, Zachary Mesyan, James~D. Mitchell, Micha\l{}
  Morayne, and Yann Péresse.
\newblock Automatic continuity, unique {P}olish topologies, and {Z}ariski
  topologies on monoids and clones.
\newblock {\em Transactions of the American Mathematical Society}.
\newblock To appear, Preprint arXiv:1912.07029.

\bibitem[EJM{\etalchar{+}}23]{EJMPP-polish}
Luke Elliott, Julius Jonu\v{s}as, James~D. Mitchell, Yann Péresse, and Michael
  Pinsker.
\newblock Polish topologies on endomorphism monoids of relational structures.
\newblock {\em Advances in Mathematics}, 431:109214, 2023.

\bibitem[Gau67]{Gaughan}
Edward~D. Gaughan.
\newblock Topological group structures of infinite symmetric groups.
\newblock {\em Proceedings of the National Academy of Sciences of the United
  States of America}, 58:907--910, 1967.

\bibitem[GJP19]{Pseudo-loop}
Pierre Gillibert, Julius Jonu\v{s}as, and Michael Pinsker.
\newblock Pseudo-loop conditions.
\newblock {\em Bulletin of the London Mathematical Society}, 51(5):917--936,
  2019.

\bibitem[Her98]{Herwig98}
Bernhard Herwig.
\newblock Extending partial isomorphisms for the small index property of many
  $\omega$-categorical structures.
\newblock {\em Israel Journal of Mathematics}, 107:93--123, 1998.

\bibitem[HHLS93]{HodgesHodkinsonLascarShelah}
Wilfrid Hodges, Ian Hodkinson, Daniel Lascar, and Saharon Shelah.
\newblock The small index property for $\omega$-stable $\omega$-categorical
  structures and for the random graph.
\newblock {\em Journal of the London Mathematical Society}, S2-48(2):204--218,
  1993.

\bibitem[Hod97]{Hodges}
Wilfrid Hodges.
\newblock {\em A shorter model theory}.
\newblock Cambridge University Press, Cambridge, 1997.

\bibitem[KR07]{KechrisRosendal}
Alexander~S. Kechris and Christian Rosendal.
\newblock Turbulence, amalgamation and generic automorphisms of homogeneous
  structures.
\newblock {\em Proceedings of the London Mathematical Society}, 94(3):302--350,
  2007.

\bibitem[Las91]{Lascar}
Daniel Lascar.
\newblock Autour de la propri\'et\'e du petit indice.
\newblock {\em Proceedings of the London Mathematical Society}, 62(1):25--53,
  1991.

\bibitem[LW80]{LachlanWoodrow}
Alistair~H. Lachlan and Robert~E. Woodrow.
\newblock Countable ultrahomogeneous undirected graphs.
\newblock {\em Transactions of the AMS}, 262(1):51--94, 1980.

\bibitem[MR16]{maissel-rubin}
Jonah Maissel and Matatyahu Rubin.
\newblock Reconstruction theorems for semigroups of functions which contain all
  transpositions of a set and for clones with the same property.
\newblock Preprint arXiv:1606.06417, 2016.

\bibitem[PP16]{PechPech-HomeoMon}
Christian Pech and Maja Pech.
\newblock On automatic homeomorphicity for transformation monoids.
\newblock {\em Monats\-hefte f\"{u}r Mathematik}, 179(1):129--148, 2016.

\bibitem[PP18]{pech-saturated}
Christian Pech and Maja Pech.
\newblock Reconstructing the topology of the elementary self-embedding monoids
  of countable saturated structures.
\newblock {\em Studia Logica}, 106(3):595--613, 2018.

\bibitem[PS]{PaoliniShelahSSIPFreeHom}
Gianluca Paolini and Saharon Shelah.
\newblock The strong small index property for free homogeneous structures.
\newblock In {\em Research Trends in Contemporary Logic}. College Publications.
\newblock To appear, Preprint ArXiv 1703.10517.

\bibitem[PS19]{PaoliniShelahReconstructing}
Gianluca Paolini and Saharon Shelah.
\newblock Reconstructing structures with the strong small index property up to
  bi-definability.
\newblock {\em Fundamenta Mathematicae}, 247:25--35, 2019.

\bibitem[PS20]{PaoliniShelahCountStbl}
Gianluca Paolini and Saharon Shelah.
\newblock Automorphism groups of countable stable structures.
\newblock {\em Fundamenta Mathematicae}, 248:301--307, 2020.

\bibitem[RS07]{RosendalSolecki}
Christian Rosendal and S\l{}awomir Solecki.
\newblock Automatic continuity of homomorphisms and fixed points on metric
  compacta.
\newblock {\em Israel Journal of Mathematics}, 162:349--371, 2007.

\bibitem[Rub94]{Rubin}
Matatyahu Rubin.
\newblock On the reconstruction of $\omega$-categorical structures from their
  automorphism groups.
\newblock {\em Proceedings of the London Mathematical Society}, 3(69):225--249,
  1994.

\bibitem[She84]{Shelah84}
Saharon Shelah.
\newblock Can you take {Solovay}'s inaccessible away?
\newblock {\em Israel Journal of Mathematics}, 48(1):1--47, 1984.

\bibitem[Sol70]{Solovay}
Robert~M. Solovay.
\newblock A model of set theory in which every set of reals is {L}ebesgue
  measurable.
\newblock {\em Annals of Mathematics}, 92:1--56, 1970.

\bibitem[Tru89]{Truss}
John~K. Truss.
\newblock Infinite permutation groups. {I}{I}. {Subgroups} of small index.
\newblock {\em Journal of Algebra}, 120(2):494--515, 1989.

\bibitem[Wil70]{Willard-topology}
Stephen Willard.
\newblock {\em General topology}.
\newblock Addison-Wesley Publishing Co., Reading, Mass.-London-Don Mills, Ont.,
  1970.

\end{thebibliography}
\end{document}